\documentclass[11pt]{amsart}
\usepackage{amsmath,amsthm,amssymb}
\usepackage[pdftex]{graphicx}
\usepackage{tikz}
\usepackage{comment}

\theoremstyle{definition}
\newtheorem{dfn}{Definition}[section]
\newtheorem{ex}[dfn]{Example}
\newtheorem{algo}[dfn]{Algorithm}
\newtheorem{rem}[dfn]{Remark}
\newtheorem{method}[dfn]{Method}

\theoremstyle{plain}
\newtheorem{thm}[dfn]{Theorem}
\newtheorem{prop}[dfn]{Proposition}
\newtheorem{cor}[dfn]{Corollary}
\newtheorem{lem}[dfn]{Lemma}

\newcommand{\dilog}[2]
{\left[\begin{matrix}
#1 \\
#2 \\
\end{matrix}\right]}
\newcommand{\mindilog}[2]
{[\begin{smallmatrix}
#1 \\
#2 \\
\end{smallmatrix}]}

\newcommand{\dmindilog}[2]
{\left[\begin{smallmatrix}
#1 \\
#2 \\
\end{smallmatrix}\right]}

\newcommand{\bi}[2]
{\left(\begin{smallmatrix}
#1\\
#2\\
\end{smallmatrix}\right)}

\newcommand{\mpmatrix}[2]
{\left(\begin{smallmatrix}
#1\\
#2\\
\end{smallmatrix}\right)}

\newcommand{\orderedprod}{\mathop{\overrightarrow{\prod}}}

\renewcommand{\ker}{\textup{Ker}\:}

\title[Explicit Forms in Lower Degrees]{Explicit Forms in Lower Degrees of rank 2 Cluster Scattering Diagrams}
\author{Ryota Akagi}
\email{ryota.akagi.e6@math.nagoya-u.ac.jp}
\date{}

\begin{document}
\maketitle
\begin{abstract}
In this paper, we study wall elements of rank 2 cluster scattering diagrams based on dilogarithm elements. We derive two major results. First, we give a method to calculate wall elements in lower degrees. By this method, we may see the explicit forms of wall elements including the Badlands, which is the complement of $G$-fan. In this paper, we write one up to 7 degrees. Also, by using this method, we derive some walls independent of their degrees. Second, we find a certain admissible form of them. In the proof of these facts, we introduce a matrix action on a structure group, which we call a similarity transformation, and we argue the relation between this action and ordered products.
\end{abstract}
\tableofcontents
\section{Introduction}
\subsection{Background}
Cluster scattering diagrams (CSDs, for short) were introduced by \cite{GHKK18}. They have great effects on cluster algebra theory, which was intoroduced by \cite{FZ02}. For example, the sign coherence of $c$-vectors and the Laurent positivity, both of which are important properties of cluster algebras, were shown by using CSDs. Roughly speaking, a CSD $\mathfrak{D}$ is a set of walls, and a wall contains a certain element of the structure group $G$ of $\mathfrak{D}$, which is a non-abelian group. In particular, $G$ has dilogarithm elements $\Psi[n]$, which are defined in Definition~\ref{def: dirogarithm elements}, and they play an important role in CSDs. In this paper, we concentrate on CSDs of rank 2. We write $\Psi[n]$ by $\mindilog{a}{b}$ for $n=(a,b)$. Then, the consistency condition of a CSD of type $(\delta_1,\delta_2)$, which is the most fundamental property of a CSD, has the following form, where $u_{(a,b)}(\delta_1,\delta_2)$ are some nonnegative rational numbers:
\begin{equation}\label{eq: ordered product}
\dilog{0}{1}^{\delta_2}\dilog{1}{0}^{\delta_1}=\dilog{1}{0}^{\delta_1}\biggl\{\orderedprod_{\substack{j; (a_j,b_j) \in \mathbb{Z}_{\geq 1}^2}} \dilog{a_j}{b_j}^{u_{(a_j,b_j)}(\delta_1,\delta_2)}\biggr\} \dilog{0}{1}^{\delta_2}.
\end{equation}
The right hand side is a product such that $\frac{a_j}{b_j}\geq\frac{a_i}{b_i}$ for any $j<i$. It is called the {\em strongly ordered product expressions} of $\mindilog{0}{1}^{\delta_2}\mindilog{1}{0}^{\delta_1}$. Moreover, in \cite{Nak23}, it is known that the above equality is obtained by applying the pentagon relation (possibly infinitely many times):
\begin{equation}\label{eq: intro pentagon}
\dilog{x}{y}^{\gamma}\dilog{z}{w}^{\gamma}=\dilog{z}{w}^{\gamma}\dilog{x+z}{y+w}^{\gamma}\dilog{x}{y}^{\gamma},
\end{equation}
where $\gamma^{-1}=yz-xw$. The explicit value of $u_{(a,b)}(\delta_1,\delta_2)$ is well known when $\delta_1\delta_2 \leq 4$. When $\delta_1\delta_2 \leq 3$, a CSD is of finite type, and the product in (\ref{eq: ordered product}) is finite \cite{GHKK18, Nak23}. On the other hand, when $\delta_1\delta_2=4$, a CSD is of affine type, and the product in (\ref{eq: ordered product}) is infinite \cite{Rei12, Rea20, Nak23, Mat21}.  Also, the explicit forms of $u_{(a,b)}(\delta_1,\delta_2)$ are known in the case $\delta_1=\delta_2$ and $a=b$ \cite{Rei23}. However, they are known only few cases. In particular, when $\delta_1\delta_2 \geq 5$, there is the region which is complement of the $G$-fan (that is so called {\em the Badlands}). The walls in the $G$-fan correspond to the cluster algebra theory, inparticular, $c$-vectors and $g$-vectors \cite[Thm.~6.13]{Nak23}. On the other hand, the structure in the Badland is hardly known. It is expected that every $u_{(a,b)}(\delta_1,\delta_2)$ is positive for such $(a,b)$ belonging to the Badlands \cite{GHKK18, Nak23}.
\subsection{Main results and ideas}
In this paper, we treat $u_{(a,b)}(\delta_1,\delta_2)$ as a function of $\delta_1$ and $\delta_2$. The main purpose is to describe $u_{(a,b)}(\delta_1,\delta_2)$ explicitly. In order to emphasize that $u_{(a,b)}(\delta_1,\delta_2)$ is a function of $\delta_1$ and $\delta_2$, we write $\delta_1=m$ and $\delta_2=n$. Namely, we mainly consider $u_{(a,b)}(m,n)$ as a function of $(m,n) \in \mathbb{Z}_{\geq 0}^2$.
\par
The main idea to obtain some results for CSDs is the similality transformation, which is a group homomorphism defined by matrices $F \in \mathrm{Mat}_{2}(\mathbb{Z}_{\geq 0})$ with $|F| \neq 0$. By applying this action to dilogarithm elements, we have
\begin{equation}
F\dilog{a}{b}=\left[F\begin{pmatrix}a \\ b \\\end{pmatrix}\right]^{1/|F|},
\end{equation}
where $(a,b) \in \mathbb{Z}_{\geq 0}^2$. This action is compatible for the pentagon relation and ordered products. In particular, by applying this action to (\ref{eq: ordered product}), we have
\begin{equation}
\begin{aligned}
&\ \left(F\dilog{0}{1}\right)^{\delta_2}\left(F\dilog{1}{0}\right)^{\delta_1}\\
=&\ \left(F\dilog{1}{0}\right)^{\delta_1}\biggl\{\orderedprod_{\substack{j; (a_j,b_j) \in \mathbb{Z}_{\geq 1}^2}} \left(F\dilog{a_j}{b_j}\right)^{u_{(a_j,b_j)}(\delta_1,\delta_2)} \biggr\}\left(F\dilog{0}{1}\right)^{\delta_2}.
\end{aligned}
\end{equation}
This equality is the key to derive strongly ordered product expressions in lower degrees.
\par
Let us see the main results. As the first result, we give a method to calculate $u_{(a,b)}(m,n)$ explicitly in order of $a+b$ (Method~\ref{method: obtain the explicit forms}). More directly, we obtain the following recurrence relations:
\newtheorem*{recurrence}{Proposition~\ref{prop: recurrence}}
\begin{recurrence}
Let $l \in \mathbb{Z}_{\geq 1}$, and let $(a,b) \in N^{+}$ with $\deg(a,b)=l+1$. Let $C_{(m,1)}$ and $C_{(m,n)}$ be the products which is defined by (\ref{eq: C of (m,1)}) and (\ref{eq: C of (m,n)}), respectively. The following two statements hold.\\
\textup{(a)}\ By applying Algorithm~\ref{oa} to $C_{(m,1)}$ repeatedly, we give a method to obtain the recurrence relation:
\begin{equation}
u_{(a,b)}(m+1,1)=u_{(a,b)}(m,1)+p(m),
\end{equation}
where $p(m)$ is some polynomial in $m$.\\
\textup{(b)}\ By applying Algorithm~\ref{oa} to $C_{(m,n)}$ repeatedly, we obtain the recurrence relation:
\begin{equation}
u_{(a,b)}(m,n+1)=u_{(a,b)}(m,n)+u_{(a,b)}(m,1)+p'(m,n),
\end{equation}
where $p'(m,n)$ is some polynomial in $m$ and $n$.
\par
Moreover, $p(m)$ and $p'(m,n)$ are determined by the data of $u_{(x,y)}(m,n)$ with $\deg(x,y)\leq l$ as functions of $m$ and $n$.
\end{recurrence}
More strongly, we can show that $\gcd(a,b)p'(m,n)$ can be expressed as the following form:
\begin{equation}
\gcd(a,b)p'(m,n)=\sum_{\substack{0 \leq k \leq A,\\ 0 \leq l \leq B}} \alpha_{k,l}\binom{m}{k}\binom{n}{l}\quad(A,B,\alpha_{k,l} \in \mathbb{Z}_{\geq 0}).
\end{equation}
In particular, polynomials of the above form are often used in this paper. We name them {\em polynomials in binomial coefficients} ({\em PBCs}, for short), and we derive some their properties in Section~\ref{Sec: PBC}. By using this method up to $a+b \leq 5$, we obtain the following explicit forms:
\begin{equation*}
\begin{aligned}
&\quad\dilog{0}{1}^n \dilog{1}{0}^m\\
&\equiv\dilog{1}{0}^{m}\dilog{4}{1}^{\bi{m}{4}\bi{n}{1}} \dilog{3}{1}^{\bi{m}{3}\bi{n}{1}}\dilog{2}{1}^{\bi{m}{2}\bi{n}{1}}\\
&\times\dilog{3}{2}^{2\bi{m}{2}\bi{n}{2}+\bi{m}{3}\bi{n}{1}+6\bi{m}{3}\bi{n}{2}}\dilog{1}{1}^{\bi{m}{1}\bi{n}{1}}\\
&\times\dilog{2}{2}^{2\bi{m}{2}\bi{n}{2}}\dilog{2}{3}^{\bi{m}{1}\bi{n}{3}+2\bi{m}{2}\bi{n}{2}+6\bi{m}{2}\bi{n}{3}} \dilog{1}{2}^{\bi{m}{1}\bi{n}{2}}\\
&\times\dilog{1}{3}^{\bi{m}{1}\bi{n}{3}} \dilog{1}{4}^{\bi{m}{1}\bi{n}{4}} \dilog{0}{1}^{n} \hspace{100pt} \mod G^{>5}.
\end{aligned}
\end{equation*}
In principle, we may proceed to any order $a+b$. However, the calculation of this method becomes complicated rapidly when the order is larger. In Example~\ref{ex: explicit forms}, we write one up to $a+b \leq 7$. 
\par
As the second result, we give the following restriction for a possible form of $u_{(a,b)}(m,n)$.
\newtheorem*{main theorem1}{Theorem~\ref{main thm1}}
\begin{main theorem1}
Let $a$ and $b$ be positive integers. Then, we express
\begin{equation*}
u_{({a},{b})}(m,n) = \gcd(a,b)^{-1} \sum_{\substack{1 \leq i \leq a,\\ 1 \leq j \leq b}} \alpha_{({a},{b})}(i,j) \binom{m}{i} \binom{n}{j},
\end{equation*}
where $\alpha_{({a},{b})}(i,j)$ are nonnegative integers independent of $m$ and $n$.
\end{main theorem1}
Thus, for each $(a,b)$, if we determine the $ab$ factors $\alpha_{(a,b)}(i,j)$, $u_{(a,b)}(m,n)$ is completely determined. Moreover, by the following claim, the special values $u_{(a,b)}(k,l)$, where $1 \leq k \leq a$ and $1 \leq l \leq b$, suffice to determine $u_{(a,b)}(m,n)$ as a function of $(m,n) \in \mathbb{Z}_{\geq 0}^2$.
\newtheorem*{inverse formula}{Proposition~\ref{prop: inverse formula}}
\begin{inverse formula}
Let $a$ and $b$ be positive integers. Then, for any $1 \leq k \leq a$ and $1 \leq l \leq b$, it holds that
\begin{equation*}
\gcd(a,b)^{-1}\alpha_{(a,b)}(k,l) = \sum_{\substack{1 \leq i \leq k,\\ 1 \leq j \leq l}} (-1)^{i+j+k+l} \binom{k}{i} \binom{l}{j} u_{(a,b)}(i,j).
\end{equation*}
\end{inverse formula}
\par
Last, we obtain the general formula of $u_{(a,2)}(m,n)$ as follows.
\newtheorem*{b=2}{Theorem~\ref{b=2}}
\begin{b=2}
For any $a \in \mathbb{Z}_{> 0}$, we have
\begin{equation}\label{relation b=2}
\begin{aligned}
&u_{(a,2)}(m,n)\\
=\ &\sum_{\frac{a}{2} < k \leq a}\left\lceil \frac{2k-a}{2} \right\rceil\binom{2k-a}{\left\lceil \frac{2k-a}{2} \right\rceil}\binom{k}{2k-a} \binom{m}{k}\binom{n}{2}\\
\ &\qquad+\sum_{\frac{a}{2}+1<k \leq a}\left\{\frac{2k-a}{2}\binom{2k-a-1}{\lceil\frac{2k-a-1}{2}\rceil}-2^{2k-a-2}\right\}\binom{k}{2k-a} \binom{m}{k}\binom{n}{1}.
\end{aligned}
\end{equation}
In the above relation, $\lceil x \rceil$ is the least integer more than or equal to $x \in \mathbb{Q}$.
\end{b=2}
\subsection{The structure of the paper}
As we can see from Theorem~\ref{main thm1}, binomial coefficients play an important role in this paper. In Section~\ref{Sec: PBC}, we show some equalities and properties which we use later.
\par
In Section~\ref{Sec: DEs and CSDs of rank 2}, we recall the definitions and notations of CSDs.
\par
In Section~\ref{Sec: Similarity transformations}, we introduce a similarity transformation, and consider the relation between a similarity transformation and ordering products.
\par
In Section~\ref{Sec: Construction}, we give the method to derive $u_{(a,b)}(m,n)$ explicitly.
\par
The latter sections, we give some properties of $u_{(a,b)}(m,n)$. The most contents are independent from each other.

\subsection*{Acknowledgement}
The author is greateful to Professor Tomoki Nakanishi for careful reading and useful advices and comments. The author thank Peigen Cao for important advices, in particular, the similarity transformations in Section~\ref{Sec: Similarity transformations}. Many statements and their proofs are clarified by thier advices.

\section{Polynomials in binomial coefficients}\label{Sec: PBC}
Binomial coefficients play an important role in exponents of ordering products. So, in this section, we prove some equalities to use later.
\begin{dfn}\label{dfn: PBC}
Let $m \in \mathbb{Z}$ and $k \in \mathbb{Z}_{\geq 0}$. Then, the binomial coefficients are defined by
\begin{eqnarray}
\binom{m}{k} = \begin{cases}
\frac{m(m-1)(m-2)\cdots(m-k+1)}{k(k-1)\cdots2\cdot1} & k \geq 1,\\
1 & k=0.\\
\end{cases}
\end{eqnarray}
We may also view them as polynomials in an indeterminate $m$. In this case, the degree of $\binom{m}{k} \in \mathbb{Q}[m]$ is $k$. Then, we call the following polynomial $f$ in $m$ and $n$ a {\em polynomial in binomial coefficients} ({\em PBC} for short).
\begin{eqnarray}
f(m,n) = \sum_{\begin{smallmatrix}0 \leq k,l\end{smallmatrix}} \alpha_{k,l} \binom{m}{k} \binom{n}{l} \in \mathbb{Q}[m,n] \quad(\alpha_{k,l} \in \mathbb{Z}).
\end{eqnarray}
Moreover, if $\alpha_{k,l} \geq 0$ for any $k$ and $l$, we call $f$ a {\em nonnegative PBC}. For any nonnegative PBC $f$, if $f \neq 0$ as a polynomial, we call $f$ a {\em positive PBC}.
\end{dfn}
By definition, a sum of two nonnegative (resp. positive) PBCs is also a nonnegative (resp. positive) PBC. Also, the equalities $(m+1)\binom{m}{k}=(k+1)\binom{m+1}{k}$ and $(m-k)\binom{m}{k}=(k+1)\binom{m}{k+1}$ hold.
\par
When $m$ and $n$ are viewed as indeterminates, the set $\left\{\binom{m}{k}\binom{n}{l}\right\}_{0 \leq k,l} $ is a basis of $\mathbb{Q}[m,n]$. So, every polynomial $f(m,n) \in \mathbb{Q}[m,n]$ can be expressed as $f(m,n) = \sum_{0 \leq k,l} \gamma_{k,l} \binom{m}{k}\binom{n}{l}$, where $\gamma_{k,l} \in \mathbb{Q}$. By the following claim, we may distinguish PBCs from mere polynomials.
\begin{lem}\label{nnPBClem}
Let $f(m,n)$ be a polynomial. Then, the following two conditions are equivalent.
\begin{itemize}
\item[\textup{(a)}] A polynomial $f(m,n)$ is a PBC.
\item[\textup{(b)}] For any $(u,v) \in \mathbb{Z}_{\geq 0}$, $f(u,v) \in \mathbb{Z}$ holds.
\end{itemize}
\end{lem}

\begin{proof}
$\textup{(a)} \Rightarrow \textup{(b)}$ is immediately shown by $\binom{u}{k} \in \mathbb{Z}$ for any $u,k \in \mathbb{Z}_{\geq 0}$. We show $\neg\textup{(a)} \Rightarrow \neg\textup{(b)}$. Suppose that a polynomial $f(m,n)=\sum_{0 \leq k,l} \alpha_{k,l}\binom{m}{k}\binom{n}{l}$ has non-integer coefficients $\alpha_{i,j} \notin \mathbb{Z}$. We chose $(i_0,j_0)$ as $i+j$ is the smallest among such $(i,j)$. Consider $f(i_0,j_0)$. Then, we have
\begin{equation}
\begin{aligned}
f(i_0,j_0) &= \sum_{0 \leq k,l} \alpha_{k,l} \binom{i_0}{k} \binom{j_0}{l}\\
&=\sum_{\substack{0 \leq k \leq i_0,\\ 0 \leq l \leq j_0}} \alpha_{k,l} \binom{i_0}{k} \binom{j_0}{l}\quad\left(\binom{i_0}{k}=0\ \textup{if}\ i_0<k\right)\\
&=\alpha_{i_0,j_0}+\sum_{\substack{0 \leq k \leq i_0,\\ 0 \leq l \leq j_0,\\(k,l) \neq (i_0,j_0)}} \alpha_{k,l} \binom{i_0}{k} \binom{j_0}{l}.
\end{aligned}
\end{equation}
The second term on the RHS is an integer since $\alpha_{k,l} \in \mathbb{Z}$ for any $(k,l)$ in the sum. By the assumption, the first term $\alpha_{i_0,j_0}$ is not an integer. Thus, $f(i_0,j_0) \notin \mathbb{Z}$ holds.
\end{proof}
For any $a,k \in \mathbb{Z}_{\geq 0}$, we can easily check the following equalities as polynomials in $m$ and $n$ (e.g., \cite[Identity~1, Identity~11, Identity~57]{Spi19}).
\begin{eqnarray}
\label{a}
\binom{m}{k}+\binom{m}{k+1}&=&\binom{m+1}{k+1},\\
\label{c}
m\binom{m}{k} &=& (k+1)\binom{m}{k+1} + k\binom{m}{k},\\
\label{binomial m+n}
\binom{m+n}{a} &=& \sum_{\begin{smallmatrix}0 \leq a_1,a_2\\ a_1+a_2=a\end{smallmatrix}} \binom{m}{a_1} \binom{n}{a_2}.
\end{eqnarray}
Moreover, the following equality holds when $m \in \mathbb{Z}_{\geq 1}$ and $k \in \mathbb{Z}_{\geq 0}$ (e.g., \cite[Identity~58]{Spi19}).
\begin{equation}\label{sum of BCs}
\sum_{j=0}^{m-1} \binom{j}{k} = \binom{m}{k+1}.
\end{equation}
The following lemma is a generalization of (\ref{c}).
\begin{lem}[Product formula]\label{binomial lemma}
For any $s,r \in \mathbb{Z}_{\geq 0}$, the following equalities hold as the element of $\mathbb{Q}[m]$.\\
\begin{equation}\label{binomial lemma 2-prod}
\begin{aligned}
\binom{m}{s}\binom{m}{r} =& \sum_{0 \leq k \leq s} \binom{s}{k} \binom{r+k}{s} \binom{m}{r+k},\\
=& \sum_{\max(s,r) \leq k \leq s+r} \binom{s}{k-r} \binom{k}{s} \binom{m}{k}.
\end{aligned}
\end{equation}
Moreover, for any $s,s',r,r' \in \mathbb{Z}_{\geq 0}$, $\displaystyle{\left\{\binom{m}{s}\binom{n}{s'}\right\}\cdot\left\{\binom{m}{r}\binom{n}{r'}\right\}}$ is a positive PBC in $m$ and $n$.
\end{lem}

\begin{proof}
The equality
\begin{equation}
\sum_{0 \leq k \leq s} \binom{s}{k} \binom{r+k}{s} \binom{m}{r+k} = \sum_{\max(s,r) \leq k \leq s+r} \binom{s}{k-r} \binom{k}{s} \binom{m}{k}
\end{equation}
can be shown by replacing $k$ with $k-r$. We prove the first one by the induction on $s$. If $s=0$, the equality is obvious. (Both sides are equal to $\binom{m}{r}$.) We assume that $\binom{m}{s}\binom{m}{r} = \sum_{0 \leq k \leq s} \binom{s}{k} \binom{r+k}{s} \binom{m}{r+k}$ for some $s$. Then, by the inductive assumption, we have
\begin{equation}
\begin{aligned}\label{2.3-1}
&\binom{m}{s+1}\binom{m}{r}=\frac{m-s}{s+1}\binom{m}{s}\binom{m}{r}
=\frac{m-s}{s+1}\sum_{0 \leq k \leq s} \binom{s}{k} \binom{r+k}{s} \binom{m}{r+k}.\\
\end{aligned}
\end{equation}
By using (\ref{c}), we have
\begin{equation}\label{eq: 12'}
m\binom{m}{r+k}=(r+k+1)\binom{m}{r+k+1}+(r+k)\binom{m}{r+k}.
\end{equation}
Thus, (\ref{2.3-1}) can be rearrenged to the following form:
\begin{equation}\label{eq: 2.3-3}
\begin{aligned}
\binom{m}{s+1}\binom{m}{r}
=&\ \frac{1}{s+1}\left\{\sum_{0 \leq k \leq s} \binom{s}{k} \binom{r+k}{s} m\binom{m}{r+k}\right.\\
&\ \qquad-\left.\sum_{0 \leq k \leq s} s\binom{s}{k} \binom{r+k}{s} \binom{m}{r+k}\right\}\\
\overset{(\ref{eq: 12'})}{=}&\ \frac{1}{s+1}\left\{\sum_{0 \leq k \leq s} \binom{s}{k} \binom{r+k}{s} (r+k+1)\binom{m}{r+k+1}\right.\\
&\ \qquad+\left.\sum_{0 \leq k \leq s}(r+k-s)\binom{s}{k} \binom{r+k}{s}\binom{m}{r+k}\right\}.\\
\end{aligned}
\end{equation}
Now, the first term on the RHS can be written as follows:
\begin{equation}\label{eq: prod 1}
\begin{aligned}
&\sum_{0 \leq k \leq s} \binom{s}{k} \binom{r+k}{s} (r+k+1)\binom{m}{r+k+1}\\
=\ &\sum_{1 \leq k \leq s+1} \binom{s}{k-1} \binom{r+k-1}{s} (r+k) \binom{m}{r+k}\\
=\ &\binom{r+s}{s}(r+s+1)\binom{m}{r+s+1}\\
&\qquad+\sum_{1 \leq k \leq s} \binom{s}{k-1} \binom{r+k-1}{s} (r+k) \binom{m}{r+k}
\end{aligned}
\end{equation}
Since $(r+s+1)\binom{r+s}{s}=(s+1)\binom{r+s+1}{s+1}$ and $(r+k)\binom{r+k-1}{s}=(s+1)\binom{r+k}{s+1}$, we have
\begin{equation}
\begin{aligned}
&\sum_{0 \leq k \leq s} \binom{s}{k} \binom{r+k}{s} (r+k+1)\binom{m}{r+k+1}\\
=\ &(s+1)\binom{r+s+1}{s+1}\binom{m}{r+s+1}\\
\ &\qquad+\sum_{1 \leq k \leq s} (s+1)\binom{s}{k-1} \binom{r+k}{s+1}\binom{m}{r+k}.\\
\end{aligned}
\end{equation}
Similarly, by using $(r+k-s)\binom{r+k}{s}=(s+1)\binom{r+k}{s+1}$, the second term on the RHS can be written as follows:
\begin{equation}
\begin{aligned}
&\sum_{0 \leq k \leq s}(r+k-s)\binom{s}{k} \binom{r+k}{s}\binom{m}{r+k}\\
=\ &(s+1)\sum_{0 \leq k \leq s}\binom{s}{k} \binom{r+k}{s+1}\binom{m}{r+k}\\
=\ &(s+1)\binom{r}{s+1}\binom{m}{r} + (s+1)\sum_{1 \leq k \leq s}\binom{s}{k} \binom{r+k}{s+1}\binom{m}{r+k}.
\end{aligned}
\end{equation}
Hence, putting these expressions to the last line of (\ref{eq: 2.3-3}), we have
\begin{equation}
\begin{aligned}
&\binom{m}{s+1}\binom{m}{r}\\
=\ & \binom{r+s+1}{s+1}\binom{m}{r+s+1}+ \binom{r}{s+1}\binom{m}{r}\\
\ &\qquad +\sum_{1 \leq k \leq s} \left\{\binom{s}{k-1}+\binom{s}{k}\right\} \binom{r+k}{s+1}\binom{m}{r+k}\\
\overset{(\ref{a})}{=}\ &\binom{r+s+1}{s+1}\binom{m}{r+s+1}+ \binom{r}{s+1}\binom{m}{r}\\
\ &\qquad +\sum_{1 \leq k \leq s} \binom{s+1}{k}\binom{r+k}{s+1}\binom{m}{r+k}\\
=\ &\sum_{0 \leq k \leq s+1} \binom{s+1}{k}\binom{r+k}{s+1}\binom{m}{r+k}.
\end{aligned}
\end{equation}
The second statement follows from the following:
\begin{equation}
\begin{aligned}
&\ \left\{\binom{m}{s}\binom{n}{s'}\right\}\cdot\left\{\binom{m}{r}\binom{n}{r'}\right\}=\left\{\binom{m}{s}\binom{m}{r}\right\}\cdot\left\{\binom{n}{s'}\binom{n}{r'}\right\}\\
=&\ \sum_{\max(s,r) \leq k \leq s+r} \binom{s}{k-r} \binom{k}{s} \binom{m}{k}\sum_{\max(s',r') \leq l \leq s'+r'} \binom{s'}{k'-r'} \binom{l}{s'} \binom{n}{l}\\
=&\ \sum_{k,l} \left\{\binom{s}{k-r} \binom{k}{s} \binom{s'}{l-r'} \binom{l}{s'}\right\} \binom{m}{k}\binom{n}{l},
\end{aligned} 
\end{equation}
where $\binom{s}{k-r} \binom{k}{s} \binom{s'}{l-r'} \binom{l}{s'} \geq 0$.
\end{proof}

\begin{cor}\label{preserve the product}
Let $f_1(m,n),f_2(m,n),\dots,f_r(m,n) \in \mathbb{Q}[m,n]$ be positive PBCs. Then, $\prod_{j=1}^r f_{j}(m,n)$ is also a positive PBC.
\end{cor}
\begin{proof}
It suffices to show the case $r=2$. Let $f(m,n)=\sum_{0 \leq k,l} \alpha_{k,l} \binom{m}{k}\binom{n}{l}$ and $g(m,n) = \sum_{0 \leq k',l'} \beta_{k',l'}\binom{m}{k'}\binom{n}{l'}$, where $\alpha_{k,l}, \beta_{k',l'} \in \mathbb{Z}_{\geq 0}$. Then,
\begin{equation}
f(m,n)g(m,n)=\sum_{0 \leq k,k',l,l'} \alpha_{k,l}\beta_{k',l'} \binom{m}{k}\binom{n}{l}\binom{m}{k'}\binom{n}{l'}.
\end{equation}
By Lemma \ref{binomial lemma}, every $\binom{m}{k}\binom{n}{l}\binom{m}{k'}\binom{n}{l'}$ is a positive PBC. So, $f(m,n)g(m,n)$ is a positive PBC.
\end{proof}
For any PBC $f(m,n)$,  we may consider a composition
\begin{equation}
\binom{f(m,n)}{k}=\frac{f(m,n)(f(m,n)-1)\cdots(f(m,n)-k+1)}{k\cdot(k-1)\cdots2\cdot1}.
\end{equation}
Then, we have a following decomposition.
\begin{lem}\label{PBC lem-1}
Let $a$ be a nonnegative integer, and let $f(m,n)=\sum_{j=1}^{r} u_{j} \binom{m}{k_j}\binom{n}{l_j}\ (u_j,k_j,l_j \in \mathbb{Z}_{\geq 0})$ be a positive PBC. Then, the following equality holds.
\begin{equation}
\binom{f(m,n)}{a} = \sum_{\begin{smallmatrix}0 \leq a_j, \\ \sum_{j} a_j = a \end{smallmatrix}} \prod_{j=1}^r\binom{u_{j}\binom{m}{k_j}\binom{n}{l_j}}{a_j}.
\end{equation}
\end{lem}
\begin{proof}
It is immediately shown by (\ref{binomial m+n}).
\end{proof}
Every factor of the above decomposition is also a positive PBC.
\begin{lem}\label{astu}
Let $a,s,t,u \in \mathbb{Z}_{\geq 0}$. Then, a polynomial
\begin{equation}
f(m,n)=\binom{u\binom{m}{s}\binom{n}{t}}{a}
\end{equation}
is a nonnegative PCB in $m$ and $n$.
\end{lem}
\begin{proof}
If $u=0$, the claim is immediately shown by definition. Assume $u > 0$ . For any $p,q \in \mathbb{Z}$, we can easily check that $f(p,q) \in \mathbb{Z}$. Thus, by Lemma~\ref{nnPBClem}, $f(m,n)$ is a PBC. Now, we fix $u > 0$ and $s,t \geq 0$. Since its degrees in $m$ and $n$ are $sa$ and $ta$ respectively, we may express
\begin{equation}
\binom{u\binom{m}{s}\binom{n}{t}}{a} = \sum_{\begin{smallmatrix}0 \leq k \leq sa,\\ 0 \leq l \leq ta\end{smallmatrix}} \alpha_{k,l}^a \binom{m}{k}\binom{n}{l} \quad (\alpha_{k,l}^a \in \mathbb{Z}).
\end{equation}
We show the following two claims.
\begin{itemize}
\item[(a)] If $u\binom{k}{s}\binom{l}{t}<a$, then $\alpha^{a}_{k,l}=0$ holds.
\item[(b)] If $u\binom{k}{s}\binom{l}{t} \geq a$, then $\alpha^{a}_{k,l}>0$ holds.
\end{itemize}
(a) In this case, for any $k',l' \in \mathbb{Z}_{\geq 0}$ such that $k' \leq k$ and $l' \leq l$, we have $u\binom{k'}{s}\binom{l'}{t} \leq u\binom{k}{s}\binom{l}{t} < a$. It implies 
\begin{equation}
\binom{u\binom{k'}{s}\binom{l'}{t}}{a}=0.
\end{equation}
Thus, we have
\begin{equation}
0=\sum_{\substack{0 \leq i \leq sa,\\ 0 \leq j \leq ta}}\alpha_{i,j}^{a} \binom{k'}{i}\binom{l'}{j}=\sum_{\substack{0 \leq i \leq k',\\ 0 \leq j \leq l'}}\alpha_{i,j}^{a} \binom{k'}{i}\binom{l'}{j} \quad \left(\binom{k'}{i}=0\ \textup{if}\ k' < i.\right)
\end{equation}
for any $k' \leq k$ and $l' \leq l$. Considering $(k',l')=(0,0)$, we have $\alpha_{0,0}^a=0$. Next, considering $(k',l')=(1,0)$, we have $\alpha_{0,0}^a+\alpha_{1,0}^a=0$, and it implies $\alpha_{1,0}^a=0$. Repeating this process, we have $\alpha_{k,l}^a=0$.\\
(b) We show the claim by the induction on $a$. For $a=0$, we have 
\begin{equation}
\dbinom{u\binom{m}{s}\binom{n}{t}}{0}=1.
\end{equation}
Thus, $\alpha^0_{0,0}=1>0$ holds. Suppose that the claim holds for some $a \geq 0$. We show that $\alpha_{k,l}^{a+1} > 0$ when $u\binom{k}{s}\binom{l}{t} \geq a+1$. We have the following equalities:
\begin{equation}\label{a+1-a}
\begin{aligned}
\ &(a+1)\binom{u\binom{m}{s}\binom{n}{t}}{a+1}=\left({u\binom{m}{s}\binom{n}{t}-a}\right)\binom{u\binom{m}{s}\binom{n}{t}}{a}\\
=\ &\left(u\binom{m}{s}\binom{n}{t}-a\right)\sum_{\begin{smallmatrix}0 \leq k \leq sa,\\ 0 \leq l \leq ta\end{smallmatrix}} \alpha_{k,l}^a \binom{m}{k}\binom{n}{l}\\
=\ &u\sum_{\begin{smallmatrix}0 \leq k \leq sa,\\ 0 \leq l \leq ta\end{smallmatrix}} \alpha_{k,l}^a \binom{m}{s}\binom{m}{k}\binom{n}{t}\binom{n}{l}-a\sum_{\begin{smallmatrix}0 \leq k \leq sa,\\ 0 \leq l \leq ta\end{smallmatrix}} \alpha_{k,l}^a \binom{m}{k}\binom{n}{l}.
\end{aligned}
\end{equation}
By (\ref{binomial lemma 2-prod}), the first term can be written as
\begin{equation}\label{eq: astu-3}
\begin{aligned}
&\sum_{\begin{smallmatrix}0 \leq k \leq sa,\\ 0 \leq l \leq ta\end{smallmatrix}} \alpha_{k,l}^a \left\{\binom{m}{s}\binom{m}{k}\right\}\left\{\binom{n}{t}\binom{n}{l}\right\}\\
\overset{(\ref{binomial lemma 2-prod})}{=}&\sum_{\begin{smallmatrix}0 \leq k \leq sa,\\ 0 \leq l \leq ta\end{smallmatrix}} \alpha_{k,l}^a\left\{\sum_{i=0}^{s} \binom{s}{i} \binom{k+i}{s} \binom{m}{k+i}\right\} \left\{\sum_{j=0}^{t} \binom{t}{j} \binom{l+j}{t} \binom{n}{l+j}\right\}\\
=&\ \sum_{\substack{0 \leq i \leq s,\\ 0 \leq j \leq t}}\sum_{\begin{smallmatrix}0 \leq k \leq sa,\\ 0 \leq l \leq ta\end{smallmatrix}} \alpha_{k,l}^a\binom{s}{i} \binom{k+i}{s} \binom{t}{j} \binom{l+j}{t} \binom{m}{k+i}\binom{n}{l+j}\\
=&\ \sum_{\substack{0 \leq i \leq s,\\ 0 \leq j \leq t}}\sum_{\substack{ i \leq k \leq sa+i, \\  j \leq l \leq ta+j \\}}\alpha_{k-i,l-j}^a \binom{s}{i} \binom{k}{s}  \binom{t}{j} \binom{l}{t} \binom{m}{k}\binom{n}{l}\\
=&\ \sum_{\substack{0 \leq i \leq s,\\ 0 \leq j \leq t}}\sum_{\substack{ 0 \leq k \leq sa+i, \\  0 \leq l \leq ta+j \\}}\alpha_{k-i,l-j}^a \binom{s}{i} \binom{k}{s}  \binom{t}{j} \binom{l}{t} \binom{m}{k}\binom{n}{l}.
\end{aligned}
\end{equation}
In the above last equality, we use $\binom{k}{s}=\binom{l}{t}=0$ for any $k < i \leq s$ and $l < j \leq t$. We decompose the region of the latter sum as follows:
\begin{equation}
\sum_{\substack{0 \leq k \leq sa+i,\\ 0 \leq l \leq ta+j}}=\sum_{\substack{0 \leq k \leq sa,\\ 0 \leq l \leq ta}}+\sum_{\substack{sa<k\leq sa+i,\\ \textup{or}\ ta < l \leq ta+j}}.
\end{equation}
Namely, we consider
\begin{equation}
\begin{aligned}
&\ \sum_{\begin{smallmatrix}0 \leq k \leq sa,\\ 0 \leq l \leq ta\end{smallmatrix}} \alpha_{k,l}^a \left\{\binom{m}{s}\binom{m}{k}\right\}\left\{\binom{n}{t}\binom{n}{l}\right\}\\
\overset{(\ref{eq: astu-3})}{=}&\ \sum_{\substack{0 \leq i \leq s,\\ 0 \leq j \leq t}}\sum_{\substack{ 0 \leq k \leq sa, \\  0 \leq l \leq ta\\}}\alpha_{k-i,l-j}^a \binom{s}{i} \binom{k}{s}  \binom{t}{j} \binom{l}{t} \binom{m}{k}\binom{n}{l}\\
&\qquad +\sum_{\substack{0 \leq i \leq s,\\ 0 \leq j \leq t}}\sum_{\substack{sa<k\leq sa+i,\\ \textup{or}\ ta<l\leq ta+j\\}}\alpha_{k-i,l-j}^a \binom{s}{i} \binom{k}{s}  \binom{t}{j} \binom{l}{t} \binom{m}{k}\binom{n}{l}.\\
\end{aligned}
\end{equation}
Then, in the first term, $i,j,k$, and $l$ are indeprndent. Thus, we can exchange the order of the sum. In the second term, we have
\begin{equation}
\sum_{\substack{0 \leq i \leq s,\\ 0 \leq j \leq t}}\sum_{\substack{sa<k\leq sa+i,\\ \textup{or}\ ta<l\leq ta+j\\}}=\sum_{\substack{sa<k\leq s(a+1),\\ \textup{or}\ ta<l\leq t(a+1)}}\sum_{\substack{k-sa\leq i \leq s,\\ l-ta\leq j \leq t}}.
\end{equation}
Thus, we have
\begin{equation}
\begin{aligned}
&\ \sum_{\begin{smallmatrix}0 \leq k \leq sa,\\ 0 \leq l \leq ta\end{smallmatrix}} \alpha_{k,l}^a \left\{\binom{m}{s}\binom{m}{k}\right\}\left\{\binom{n}{t}\binom{n}{l}\right\}\\
=&\ \sum_{\substack{ 0 \leq k \leq sa, \\  0 \leq l \leq ta\\}}\biggl\{\sum_{\substack{0 \leq i \leq s,\\ 0 \leq j \leq t}}\alpha_{k-i,l-j}^a \binom{s}{i} \binom{k}{s}  \binom{t}{j} \binom{l}{t}\biggr\}\binom{m}{k}\binom{n}{l}\\
&\qquad +\sum_{\substack{sa<k\leq s(a+1),\\ \textup{or}\ ta<l\leq t(a+1)\\}}\biggl\{\sum_{\substack{k-sa \leq i \leq s,\\ l-ta \leq j \leq t}}\alpha_{k-i,l-j}^a \binom{s}{i} \binom{k}{s}  \binom{t}{j}\binom{l}{t}\biggr\} \binom{m}{k}\binom{n}{l}.\\
\end{aligned}
\end{equation}
Putting the last expression to the last line of (\ref{a+1-a}), we have
\begin{equation}\label{2.6-a}
\begin{aligned}
&\ (a+1)\dbinom{u\binom{m}{s}\binom{n}{t}}{a+1}\\
=&\ \sum_{\substack{ 0 \leq k \leq sa, \\  0 \leq l \leq ta\\}}\biggl\{u\sum_{\substack{0 \leq i \leq s,\\ 0 \leq j \leq t}}\alpha_{k-i,l-j}^a \binom{s}{i} \binom{k}{s}  \binom{t}{j} \binom{l}{t}-a\alpha_{k,l}^a\biggr\}\binom{m}{k}\binom{n}{l}\\
&\qquad +\sum_{\substack{sa<k\leq s(a+1),\\ \textup{or}\ ta<l\leq t(a+1)\\}}\biggl\{u\sum_{\substack{k-sa \leq i \leq s,\\ l-ta \leq j \leq t}}\alpha_{k-i,l-j}^a \binom{s}{i} \binom{k}{s}  \binom{t}{j}\binom{l}{t}\biggr\} \binom{m}{k}\binom{n}{l}.\\
\end{aligned}
\end{equation}
If $k > sa$ or $l > ta$, then we have
\begin{equation}
\begin{aligned}
(a+1)\alpha_{k,l}^{a+1}=u\sum_{\substack{k-sa \leq i \leq s,\\ l-ta \leq j \leq t}}\alpha_{k-i,l-j}^{a}\binom{s}{i}\binom{k}{s}\binom{t}{j}\binom{l}{t} > 0.
\end{aligned}
\end{equation}
This is because, for $i=k-sa$ and $j=l-ta$, $\alpha^{a}_{k-i,l-j}=\alpha^{a}_{sa,ta}>0$.
If $k \leq sa$ and $l \leq ta$, then we have
\begin{equation}
\begin{aligned}
&(a+1)\alpha_{k,l}^{a+1}=u\sum_{\substack{0 \leq i \leq s,\\ 0 \leq j \leq t}}\alpha_{k-i,l-j}^{a}\binom{s}{i}\binom{k}{s}\binom{t}{j}\binom{l}{t}-a\alpha_{k,l}^a\\
=\ &\alpha_{k,l}^a\biggl\{u\binom{k}{s}\binom{l}{t}-a\biggr\}+u\sum_{\substack{0 \leq i \leq s,\\ 0 \leq j \leq t,\\ (i,j) \neq (0,0)}}\alpha_{k-i,l-j}^{a}\binom{s}{i}\binom{k}{s}\binom{t}{j}\binom{l}{t}.\\
\end{aligned}
\end{equation}
is positive since $u\binom{k}{s}\binom{l}{t}-a \geq 1$ and $\alpha_{k,l}^a > 0$. Hence, $\alpha_{k,l}^{a+1} > 0$ when $u\binom{k}{s}\binom{l}{t} \geq a+1$. This completes the proof.
\end{proof}

We have the main conclusion in this section.
\begin{prop}\label{composition preserve filling}
If $f,g,$ and $h$ are positive PBCs, then $f(g(m,n),h(m,n))$ is also a positive PBC.
\end{prop}

\begin{proof}
Since the sum of positive PBCs is also a positive PBC, it suffices to show the case of $f(m,n)=\binom{m}{a}\binom{n}{b}$ for $a,b \in \mathbb{Z}_{\geq 0}$. In this case, we have $f(g(m,n),h(m,n))=\binom{g(m,n)}{a}\binom{h(m,n)}{b}$. Moreover, by Corollary~\ref{preserve the product}, it suffices to show that $\binom{g(m,n)}{a}$ is a positive PBC. It is immediately shown by Lemma~\ref{PBC lem-1}, Lemma~\ref{astu}, and Corollary~\ref{preserve the product}. 
\end{proof}

\section{Diligarithm elements and cluster scattering diagrams of rank 2}\label{Sec: DEs and CSDs of rank 2}
In this section, we summarize the definitions and properties of cluster scattering diagrams which were introduced by \cite{GHKK18}. We concentrate on them of rank 2, and most notations mainly follow from \cite{Nak23}.
\subsection{Dilogarithm elements and ordered products}\label{Dilogarithm elements and ordering products}
\begin{dfn}[Fixed data and seed]
We define a {\em fixed data} $\Gamma=(N,N^{\circ},\{,\},\delta_1,\delta_2)$ and a {\em seed} $\mathfrak{s}=(e_1,e_2)$ as follows:
\begin{itemize}
\item A lattice $N \cong \mathbb{Z}^2$ with a skew-symmetric bilinear form $\{,\}: N \times N  \to \mathbb{Q}$.
\item Positive integers $\delta_1,\delta_2 \in \mathbb{Z}_{>0}$, and a basis $(e_1,e_2)$ of $N$. They satisfy $\{\delta_ie_i,e_j\} \in \mathbb{Z}$ for $i,j=1,2$.
\item A sublattice $N^{\circ} = \mathbb{Z}(\delta_1e_1)\oplus\mathbb{Z}(\delta_2e_2) \subset N$.
\end{itemize}
\end{dfn}
For given fixed data $\Gamma$ and seed $\mathfrak{s}$ as above, we have dual lattices $M=\mathrm{Hom}_{\mathbb{Z}}(N,\mathbb{Z})$ and $M^{\circ}=\mathrm{Hom}_{\mathbb{Z}}(N^{\circ},\mathbb{Z})$, and we define a real vector space $M_{\mathbb{R}}=M\otimes_{\mathbb{Z}}\mathbb{R}$. We regard
\begin{equation}
M \subset M^{\circ} \subset M_{\mathbb{R}}.
\end{equation}
Also, we have the dual basis $(e^*_1,e^*_2)$ of $M$. Let $f_i=e^*_i/\delta_i$. Then, $(f_1,f_2)$ is a basis of $M^\circ$. We define the canonical pairing
\begin{equation}
\begin{aligned}
\langle\ ,\ \rangle: M_{\mathbb{R}} \times N &\to \mathbb{R},\\
\left\langle \sum_{i=1}^2 \alpha_i f_i, \sum_{j=1}^2 \beta_j e_j \right\rangle &= \sum_{i=1}^2 \delta_{i}^{-1}\alpha_i\beta_i.
\end{aligned}
\end{equation}
Let
\begin{equation}
N^{+} = \left\{ \sum_{i=1}^2 a_ie_i\ \middle|\ a_i \in \mathbb{Z}_{\geq 0}, \sum_{i=1}^2 a_i > 0 \right\}.
\end{equation}
We define the degree function $\deg : N^{+} \to \mathbb{Z}_{> 0}$ as
\begin{equation}
\deg\left(\sum_{i=1}^2 a_ie_i\right) = \sum_{i=1}^2 a_i.
\end{equation}
For any integer $l > 0$, we define the following sets.
\begin{equation}
\begin{aligned}
(N^{+})^{\leq l} &= \{n \in N^{+} \mid \deg(n) \leq l\} , \quad (N^{+})^{> l} = \{n \in N^{+} \mid \deg(n) > l\},\\
N^{+}_{\mathrm{pr}} &= \{ n \in N^{+} \mid \mathrm{For\ any\ } j \in \mathbb{Z}_{>1}, n/j \notin N^{+}\}.
\end{aligned}
\end{equation}
\begin{dfn}[Normalization factor]\label{dfn: normalization factor}
Let $\mathfrak{s}$ be a seed for a fixed data $\Gamma$. For any $n \in N^+$, we define $\delta(n)$ as the smallest positive rational number such that $\delta(n)n \in N^\circ$, and we call it the {\em normalization factor} of $n$ with respect to $(\Gamma,\mathfrak{s})$.
\end{dfn}
\begin{dfn}[Structure group]
Let $\mathfrak{g}$ be an $N^{+}\textup{-graded Lie algebra}$ over $\mathbb{Q}$ with generators $X_n\:(n \in N^{+})$ as follows:
\begin{eqnarray}
\mathfrak{g}=\bigoplus_{n \in N^{+}} \mathfrak{g}_n, \quad \mathfrak{g}_n=\mathbb{Q}X_n,\\
\;[X_n,X_{n'}]=\left\{n,n'\right\}X_{n+n'}\label{eq: Lie bracket}.
\end{eqnarray}
For each integer $l \in \mathbb{Z}_{>0}$, we define an ideal $\mathfrak{g}^{>l}$ of $\mathfrak{g}$ as
\begin{equation}
\mathfrak{g}^{>l}= \bigoplus_{n \in (N^{+})^{> l}}\mathfrak{g}_n,
\end{equation}
and we define the quotient
\begin{equation}
\mathfrak{g}^{\leq l}=\mathfrak{g}/\mathfrak{g}^{>l}.
\end{equation}
We define the group
\begin{equation}
G^{\leq l}=\bigl\{\exp(X) \mid X \in \mathfrak{g}^{\leq l} \bigr\}
\end{equation}
whose product is given by the {\em Baker-Campbell-Hausdorff formula} (e.g., \cite[\S V.5]{Jac79}).
\begin{equation}\label{eq: the BCH formula}
\begin{aligned}
&\ \exp(X)\exp(Y)\\
=&\ \exp\left(X+Y+\frac{1}{2}[X,Y]+\frac{1}{12}[X,[X,Y]]-\frac{1}{12}[Y,[X,Y]]+\cdots\right).
\end{aligned}
\end{equation}
Since the canonical projection $\pi_{l',l}: \mathfrak{g}^{\leq l'} \to \mathfrak{g}^{\leq l}\ (l'>l)$ induces the canonical projection $\pi_{l',l}: G^{\leq l'} \to G^{\leq l}$, we can consider the inverse limit of $\{\pi_{l+1,l}\}$, and we obtain a group
\begin{equation}
G=\lim_{\longleftarrow} G^{\leq l}
\end{equation}
with the canonical projection $\pi_l: G \to G^{\leq l}$. This group $G$ is called the {\em structure group} corresponding to $(\Gamma, \mathfrak{s})$. We define $G^{> l}=\ker \pi_{l}$.\par
\end{dfn}
\begin{dfn}
For any $g,g' \in G$ and $l \in \mathbb{Z}_{\geq 1}$, we write $g \equiv g' \mod G^{> l}$ when $\pi_{l}(g)=\pi_{l}(g')$, and we say $g$ is equal to $g'$ in $G^{\leq l}$.
\end{dfn}
By definition, for any $g,g' \in G$, $g=g$ is equivalent to $g \equiv g' \mod G^{>l}$ for any $l \in \mathbb{Z}_{\geq 1}$. 
\par
For any $g=\exp(X) \in G$ and $c \in \mathbb{Q}$, we define $g^{c}=\exp(cX)$. Then, $g^{0}=\mathrm{id}$ and $g^{c}g^{c'}=g^{c+c'}$ hold.\par
Every structure group $G$ is determined by a fixed data $\Gamma$ and a seed $\mathfrak{s}$. However, there is some redundancy.
\begin{dfn}[Exchange matrix]
For any fixed data $\Gamma$ and seed $\mathfrak{s}$, define a {\em exchange matrix} $B_{\Gamma,\mathfrak{s}}$ associated to $\Gamma$ and $\mathfrak{s}$ by
\begin{equation}
B_{\Gamma,\mathfrak{s}}=\begin{pmatrix}
0 & \{\delta_1e_1,e_2\}\\
\{\delta_2e_2,e_1\} & 0\\
\end{pmatrix}.
\end{equation}
\end{dfn}
\begin{prop}[{\cite[Prop.~1.23]{Nak23}}]\label{prop: B-equivalence}
Let $\Gamma$ and $\Gamma'$ be fixed data, and let $\mathfrak{s}$ and $\mathfrak{s}'$ be seeds for $\Gamma$ and $\Gamma'$, respectively. Let $G$ and $G'$ be structure groups corresponding to $(\Gamma,\mathfrak{s})$ and $(\Gamma',\mathfrak{s}')$, respectively. If $B_{\Gamma,\mathfrak{s}}=B_{\Gamma',\mathfrak{s}'}$, then $G$ and $G'$ are isomorphic.
\end{prop}
If we focus on a structure group, it suffices to consider the case of $\{e_2,e_1\}=1$. For any fixed data $\Gamma'=(N,(N^{\circ})',\{,\}',\delta'_1,\delta'_2)$ and seed $\mathfrak{s}=(e_1,e_2)$ such that $\{e_2,e_1\}' > 0$ (if $\{e_2,e_1\}'<0$, interchange $e_1$ and $e_2$), we define another fixed data $\Gamma=(N,N^{\circ},\{,\},\delta_1,\delta_2)$ by
\begin{equation}\label{eq: rescaling}
\begin{aligned}
&\delta_1=-\{\delta'_1e_1,e_2\}',\quad\delta_2'=\{\delta'_2e_2,e_1\}',\quad\{e_2,e_1\}=1,\\
&N^{\circ}=\left\{a(\delta_1e_1)+b(\delta_2e_2)\; |\; a,b \in \mathbb{Z} \right\} \subset N.
\end{aligned}
\end{equation}
Then, $\mathfrak{s}$ is also a seed for $\Gamma$. Moreover, because of Proposition~\ref{prop: B-equivalence}, the structure group corresponding to $(\Gamma,\mathfrak{s})$ is isomorphic to the one corresponding to $(\Gamma',\mathfrak{s})$. From now on, we fix a seed $(e_1,e_2)$ satisfying $\{e_2,e_1\}=1$. Then, we may view $N = \mathbb{Z}^2$, $N^{+}=\mathbb{Z}_{\geq 0}^2 \backslash \{(0,0)\}$, and $M_{\mathbb{R}}=\mathbb{R}^2$ as follows:
\begin{equation}\label{eq: identities}
\begin{aligned}
&N \to \mathbb{Z}^{2}, \quad& ae_1+be_2 \mapsto (a,b),\\
&M_{\mathbb{R}} \to \mathbb{R}^{2}  & \alpha f_1 + \beta f_2 \mapsto (\alpha,\beta).
\end{aligned}
\end{equation}
In the above notations and our assumption, the skew-bilinear form $\{,\} : N \times N \to \mathbb{Q}$ may be viewed as
\begin{equation}\label{eq: skew bilinear form}
\{(a,b),(c,d)\}=bc-ad=-\left|\begin{matrix}
a & c\\
b & d\\
\end{matrix}\right|.
\end{equation}
Moreover, for any $(a,b) \in N^{+}$, the normalization factor $\delta(a,b)$ only depends on the data of $(\delta_1,\delta_2)$. So, we call it the normalization factor with respect to $(\delta_1,\delta_2)$. In particular, when $(\delta_1,\delta_2)=(1,1)$, we write it by $d(a,b)$. This is given by
\begin{equation}
d(a,b)=\frac{1}{\gcd(a,b)}.
\end{equation}
\begin{dfn}[Dilogarithm element]\label{def: dirogarithm elements}
For each $n \in N^{+}$, we define the {\em dilogarithm element} for $n$ as
\begin{equation}\label{eq: dilogarithm elements}
\Psi[n]=\exp\left(\sum_{j=1}^\infty \frac{(-1)^{j+1}}{j^2}X_{jn}\right) \in G.
\end{equation}
\end{dfn}
Let $n \in N^{+}$. Then, the lowest term of $\Psi[n]$ is $X_n$. Thus, for any positive integer $l < \deg(n)$, it holds that $\pi_{l}(\Psi[n])=\mathrm{id}$.
\par
In this paper, the following relations may be viewed as fundamental relations.
\begin{prop}[{\cite[Prop.~III.1.14]{Nak23}}]\label{pentagon relation}
Let $n,n' \in N^{+}$. Then, the following relations hold.\\
\textup{(a)}\:\:If $\{n',n\}=0$, then for any $\gamma,\gamma' \in \mathbb{Q}$, it holds that
\begin{equation}\label{exchange}
\Psi[n']^{\gamma'}\Psi[n]^{\gamma} = \Psi[n]^{\gamma}\Psi[n']^{\gamma'}.
\end{equation}
\textup{(b)}\:{\textup{(pentagon relation)}}\:If $\{n',n\}=\gamma^{-1} \in \mathbb{Q}\backslash\{0\}$, it holds that
\begin{equation}\label{eq-pentagon relation}
\Psi[n']^{\gamma}\Psi[n]^{\gamma} = \Psi[n]^{\gamma}\Psi[n+n']^{\gamma}\Psi[n']^{\gamma}.
\end{equation}
\end{prop}
Also, by \ref{eq: the BCH formula}, the following relation is immediately shown for any $n,n' \in N^{+}$, $\gamma,\gamma'' \in \mathbb{Q}$, and $l < \deg(n+n')$.
\begin{equation}\label{eq: exchange of higher degree}
\Psi[n']^{\gamma'} \Psi[n]^{\gamma} \equiv \Psi[n]^{\gamma}\Psi[n']^{\gamma'} \mod G^{> l}.
\end{equation}
Now, we introduce the notation
\begin{equation}
\dilog{a}{b} = \Psi[(a,b)]\quad((a,b) \in N^{+}).
\end{equation}
Also, we define the degree $\deg\mindilog{a}{b}$ by $\deg\left(\begin{smallmatrix}a\\b\\\end{smallmatrix}\right)=a+b$.
\begin{dfn}\label{dfn: order on N+}
We define the total order $\leq$ on $N^{+}$ as follows:
\begin{equation}
\begin{aligned}
(a,b) \leq (c,d) \Leftrightarrow&\ \{(a,b),(c,d)\}<0,\\
&\ \textup{or there exists $k \in \mathbb{Q}_{\geq 1}$ such that}\ (c,d)=k(a,b).
\end{aligned}
\end{equation}
Moreover, we also write it as $\mindilog{a}{b} \leq \mindilog{c}{d}$.
\end{dfn}
Note that, if $(c,d)=k(a,b)$, then $\{(a,b),(c,d)\}=0$. Thus, if $(a,b) \leq (c,d)$, then $\{(a,b),(c,d)\} \leq 0$ holds.\par
Let $J$ be a ordered countable set, and let $D=\left\{\mindilog{a_j}{b_j}^{u_j}\ |\ j \in J, (a_j,b_j) \in N^{+}, u_j \in \mathbb{Q}\right\}$ be a set. Suppose that for any $l \in \mathbb{Z}_{\geq 1}$, the set $J^{\leq l}= \{ j \in J\ |\ \deg(a_j,b_j) \leq l \}$ is finite, and we write $J^{\leq l}=\{j^l_0<j^l_1<\cdots<j^l_s\}$. Then, we define 
\begin{equation}
\prod_{j \in J} \dilog{a_j}{b_j}^{u_j} = \lim_{l \to \infty} \left(\dilog{a_{j^l_0}}{b_{j^l_0}}^{u_{j^l_0}}\dilog{a_{j^l_1}}{b_{j^l_1}}^{u_{j^l_1}}\cdots\dilog{a_{j^l_s}}{b_{j^l_s}}^{u_{j^l_s}}\right).
\end{equation}
\begin{dfn}
Let $J$ be an ordered countable set. Then, a product
\begin{equation}
\prod_{j \in J} \dilog{a_j}{b_j}^{u_j} \quad \left(({a_j},{b_j}) \in N^{+}, u_j \in \mathbb{Q}\right)
\end{equation}
is said to be {\em ordered} (resp. {\em anti-ordered}) if $\mindilog{a_i}{b_i} \leq \mindilog{a_j}{b_j}$ (resp. $\mindilog{a_i}{b_i} \geq \mindilog{a_j}{b_j}$) for any $i < j$ in $J$.
\end{dfn}
In particular, if $(a_i,b_i)<(a_j,b_j)$ for any $i < j \in J$, we say that the product $\prod_{j \in J} \mindilog{a_j}{b_j}^{u_j}$ is {\em strongly ordered}, and we write it by
\begin{equation}
\orderedprod_{j \in J}\dilog{a_j}{b_j}^{u_j}.
\end{equation}
Every ordered product becomes the strongly ordered product by gathering same dilogarithm elements.

\subsection{Cluster scattering diagrams}\label{Cluster scattering diagrams}
We introduce the notation
\begin{equation}
\sigma(m)=\mathbb{R}_{\geq 0}m \quad (m \in M_{\mathbb{R}}\backslash\{0\}).
\end{equation}
\begin{dfn}[Wall]
A {\em wall} $w=(\mathfrak{d}, g)_n$ for a seed $\mathfrak{s}$ consists of the following:
\begin{itemize}
\item $n \in N^{+}_{\mathrm{pr}}$ (The {\em normal vector} of $w$).
\item $\mathfrak{d} = \sigma(m)$ or $\mathfrak{d}=\sigma(m) \cup \sigma(-m)$, where $m \in M_{\mathbb{R}}\backslash\{0\}$ satisfies $\langle m, n \rangle$=0 (The {\em support} of $w$).
\item $g \in G$ is expressed as the following form (The {\em wall element} of $w$).
\begin{equation}\label{eq: parallel element}
g=\exp\left(\sum_{j=1}^\infty c_{j}X_{jn}\right) \quad (c_j \in \mathbb{Q}).
\end{equation}
\end{itemize}
\end{dfn}
It is known that the product of dilogarithm elements $\prod_{k = 1}^{\infty} \Psi[kn]^{a_k}\ (a_k \in \mathbb{Q})$ can be expressed as the form of (\ref{eq: parallel element}) \cite[Prop.~1.13]{Nak23}. 

\begin{dfn}
The group homomorphism $p^*: N \to M^{\circ}$ is defined by
\begin{equation}
p^*(n)=\{\cdot,n\}.
\end{equation}
Then, a wall $w=(\mathfrak{d},g)_{n}$ is {\em incoming} (resp. {\em outgoing}) if $p^*(n) \in \mathfrak{d}$ (resp. $p^{*}(n) \notin \mathfrak{d}$).  
\end{dfn}

\begin{dfn}[Scattering diagram]
A {\em scattering diagram} $\mathfrak{D}=\{w_{\lambda}\}_{\lambda \in \Lambda}$ is a collection of walls such that the following conditions hold.
\begin{itemize}
\item The index set $\Lambda$ is countable.
\item For each integer $l \in \mathbb{Z}_{>0}$, there are only finitely many walls $w_{\lambda}$ such that $\pi_{l}(g_{\lambda}) \neq \mathrm{id}$.
\end{itemize}
We define the {\em support} of $\mathfrak{D}$ by
\begin{equation}
\begin{aligned}
\mathrm{Supp}(\mathfrak{D}) &= \bigcup_{\lambda \in \Lambda} \mathfrak{d}_{\lambda}.
\end{aligned}
\end{equation}
\end{dfn}

\begin{dfn}[Admissible curve]
Let $\mathfrak{D}=\{w_{\lambda} = (\mathfrak{d}_{\lambda}, g_{\lambda})_{n_{\lambda}}\}_{\lambda \in \Lambda}$ be a scattering diagram. We say that a smooth curve $\gamma: [0,1] \to M_{\mathbb{R}}$ is {\em admissible} for $\mathfrak{D}$ if it satisfies the following conditions:
\begin{itemize}
\item $\gamma(0),\gamma(1) \notin \mathrm{Supp}(\mathfrak{D})$, and $\gamma(t) \neq 0$ for any $t$.
\item If $\gamma$ and $\mathfrak{d}_{\lambda}$ intersect, then $\gamma$ intersects $\mathfrak{d}_{\lambda}$ transversally.
\end{itemize}
\end{dfn}
Let $\gamma$ be an admissible curve for $\mathfrak{D}$. For each positive integer $l$, there exist only finitely many walls $w_i = (\mathfrak{d}_i,g_i)_{n_i}$ ($i=1,2,\dots,s$) such that $\gamma$ intersects $\mathfrak{d}_i$ and $\pi_{l}(g_i) \neq {\rm id}$. Let $t_i$ be a real number such that $\gamma(t_i)$ be the intersection of $\gamma$ and $\mathfrak{d}_i$, and assume
\begin{equation}
0<t_1 \leq t_2 \leq \cdots \leq t_s <1.
\end{equation}
We define the intersection sign $\epsilon_i$ by
\begin{equation}
\epsilon_i=\begin{cases}
1 & \langle n_i, \gamma'(t_i) \rangle<0,\\
-1 & \langle n_i, \gamma'(t_i) \rangle>0.\\
\end{cases}
\end{equation}
\begin{dfn}[Path ordered product]
As the above notations, we define the {\em path ordered product} $\mathfrak{p}_{\gamma,\mathfrak{D}}$ as
\begin{equation}
\mathfrak{p}_{\gamma,\mathfrak{D}}=\lim_{l \to \infty} (g_s^{\epsilon_s} \cdots g_1^{\epsilon_1}) \in G.
\end{equation}
\end{dfn}
\begin{dfn}[Equivalence]
Let $\mathfrak{D}$ and $\mathfrak{D}'$ be scattering diagrams. We say that $\mathfrak{D}$ and $\mathfrak{D}'$ are {\em equivalent} if $\mathfrak{p}_{\gamma, \mathfrak{D}}=\mathfrak{p}_{\gamma, \mathfrak{D}'}$ for any admissible curve $\gamma$ for both $\mathfrak{D}$ and $\mathfrak{D}'$,
\end{dfn}
\begin{dfn}[Consistency]\label{def: consistency}
Let $\mathfrak{D}$ be a scattering diagram. We say that $\mathfrak{D}$ is {\em consistent} if $\mathfrak{p}_{\gamma,\mathfrak{D}}=\mathfrak{p}_{\gamma',\mathfrak{D}}$ for any admissible curve $\gamma$ and $\gamma'$ with the same endpoints.
\end{dfn}
\begin{thm}[{\cite[Thm.~1.21]{GHKK18}}]\label{thm: CSD}
For any fixed data $\Gamma$ and seed $\mathfrak{s}$, there exists a consistent scattering diagram $\mathfrak{D}_{\Gamma,\mathfrak{s}}$ satisfying the following properties:
\begin{itemize}
\item Both $(e_1^{\perp}, \Psi[e_1]^{\delta_1})_{e_1}$ and $(e_2^{\perp}, \Psi[e_2]^{\delta_2})_{e_2}$ are incoming walls of $\mathfrak{D}_{\Gamma,\mathfrak{s}}$.
\item Every wall except for $(e_1^{\perp}, \Psi[e_1]^{\delta_1})_{e_1}$ and $(e_2^{\perp}, \Psi[e_2]^{\delta_2})_{e_2}$ is outgoing.
\end{itemize}
Moreover, a scattering diagram satisfying the above properties is unique up to the equivalence.
\end{thm}
The above scattering diagram $\mathfrak{D}_{\Gamma,\mathfrak{s}}$ is called a {\em cluster scattering diagram} ({\em CSD} for short) for $\Gamma$ and $\mathfrak{s}$. Recall  the transformation (\ref{eq: rescaling}) of $\Gamma$. Under our assumption $\{e_2,e_1\}=1$, a CSD $\mathfrak{D}_{\Gamma,\mathfrak{s}}$ essentially depends on only positive integers $\delta_1$ and $\delta_2$. So, we write $\mathfrak{D}_{\Gamma,\mathfrak{s}}$ by $\mathfrak{D}_{\delta_1,\delta_2}$.\par
Let $(f_1,f_2)$ be the dual basis of $(\delta_1e_1,\delta_2,e_2)$, and let $\sigma=\{af_1+bf_2 \in M_{\mathbb{R}}\ |\ a>0, b<0 \}$. Then, the second condition in Theorem~\ref{thm: CSD} implies that there are no outgoing walls in the region $M_\mathbb{R}\backslash\sigma$. Thus, we may consider the admissible curves $\mathfrak{p}_{\gamma_{+},\mathfrak{D}_{\delta_1,\delta_2}}$ and $\mathfrak{p}_{\gamma_{-},\mathfrak{D}_{\delta_1,\delta_2}}$ as Figure~\ref{fig: CSD}.

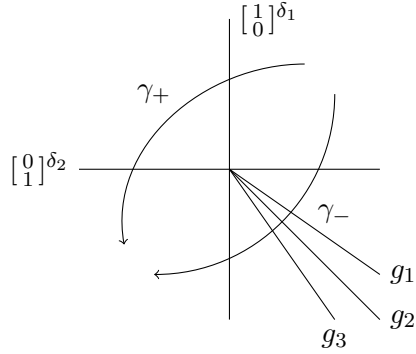
\begin{figure}[htbp]
\centering
\begin{tikzpicture}[scale=2]

\coordinate (O) at (0,0);
\coordinate (A) at (1,0);
\coordinate[label=right:{$\mindilog{1}{0}^{\delta_1}$}] (B) at (0,1);
\coordinate[label=left:{$\mindilog{0}{1}^{\delta_2}$}] (C) at (-1,0);
\coordinate (D) at (0,-1);
\coordinate[label=right:{$g_1$}] (E_1) at (1,-0.7);
\coordinate[label=right:{$g_2$}] (E_2) at (1,-1);
\coordinate[label=below:{$g_3$}] (E_3) at (0.7,-1);

\draw (O)--(A);
\draw (O)--(B);
\draw (O)--(C);
\draw (O)--(D);
\draw (O)--(E_1);
\draw (O)--(E_2);
\draw (O)--(E_3);

\draw[->] (0.5,0.7) to [out=180,in=100] (-0.7,-0.5);
\draw (-0.5,0.5)node{$\gamma_+$};

\draw[->] (0.7,0.5) to [out=-90,in=0] (-0.5,-0.7);
\draw (0.7,-0.3)node{$\gamma_-$};
\end{tikzpicture}
\caption{Cluster scattering diagram}\label{fig: CSD}
$\mathfrak{p}_{\gamma_{+},\mathfrak{D}}=\mindilog{0}{1}^{\delta_2}\mindilog{1}{0}^{\delta_1}$, $\mathfrak{p}_{\gamma_{-},\mathfrak{D}}=\mindilog{1}{0}^{\delta_1}g_3g_2g_1\mindilog{0}{1}^{\delta_2}$.\\
Moreover, $g_i$ can be expressed as $\prod_{k=1}^{\infty} \Psi[kn_i]^{u_k}$, where $n_i \in N^{+}$ is the normal vector and $u_{k} \in \mathbb{Q}$.
\end{figure}
Moreover, $\mathfrak{p}_{\gamma_{+},\mathfrak{D}_{\delta_1,\delta_2}}=\mindilog{0}{1}^{\delta_2}\mindilog{1}{0}^{\delta_1}$ holds. Thus, the consistency condition of a CSD is equivalent to $\mindilog{0}{1}^{\delta_2}\mindilog{1}{0}^{\delta_1}=\mathfrak{p}_{\gamma_{-},\mathfrak{D}_{\delta_1,\delta_2}}$.\par
For any $\delta_{1},\delta_{2} \in \mathbb{Z}_{> 0}$, all wall elements in $\mathfrak{D}_{\delta_1,\delta_2}$ can be expressed as the form $\Psi[n]^{s\delta(n)}$, where $n \in N^{+}$ and $s \in \mathbb{Z}_{> 0}$. The resulting CSD is called the {\em positive realization} \cite{Nak23}. The following theorem is a key for a positive realization.
\begin{thm}[{\cite[Prop.~III.5.4]{Nak23}}]\label{ordering lemma}
Let
\begin{equation}
C^{\textup{in}} = \Psi[n'_k]^{s'_k\delta(n'_k)} \cdots \Psi[n'_1]^{s'_1\delta(n'_1)} \quad (n'_j \in N^{+}, s'_j \in \mathbb{Z}_{>0})
\end{equation}
be any finite anti-ordered product. Then, there exists an unieque strongly ordered product
\begin{equation}
C^{\textup{out}}=\orderedprod_{j}\Psi[n_j]^{s_j\delta(n_j)} \quad (n_j \in N^{+}, s_j \in \mathbb{Z}_{>0})
\end{equation}
which is equal to $C^{\textup{in}}$ as the element of $G$. Moreover, $n_j$ satisfies $n'_1 \leq n_j \leq n'_k$.
\end{thm}
In the above notations, we call $C^{\textup{out}}$ the {\em strongly ordered product expression} of $C^{\textup{in}}$. In particular, for any $\delta_1, \delta_2 \in \mathbb{Z}^2_{> 0}$, an anti-ordered product $\mindilog{0}{1}^{\delta_2}\mindilog{1}{0}^{\delta_1}$ is expressed as the (possibly infinte) strongly ordered product
\begin{equation}\label{eq: consistency}
\dilog{0}{1}^{\delta_2}\dilog{1}{0}^{\delta_1}=\orderedprod_{j \in J} \dilog{a_j}{b_j}^{s_{(a_j,b_j)}\delta(a_j,b_j)}.
\end{equation}
Then, the CSD $\mathfrak{D}_{\delta_1,\delta_2}$ is described as $\mathfrak{D}_{\delta_1,\delta_2} = \{{\bf w}_{e_1},{\bf w}_{e_2}\}\cup\{{\bf w}_{(a,b)}\}_{(a,b) \in N^{+}_{\mathrm{pr}}\backslash\{(1,0),(0,1)\}}$, where ${\bf w}_{e_1}=(e_1^{\perp}, \mindilog{1}{0}^{\delta_1})_{e_1}$, ${\bf w}_{e_2}=(e_2^{\perp}, \mindilog{0}{1}^{\delta_2})_{e_2}$, and, for any $(a,b) \in N^{+}_{\mathrm{pr}}\backslash\{(1,0),(0,1)\}$, ${\bf w}_{(a,b)}=(\sigma(\delta_2b,-\delta_1a),g_{(a,b)})_{(a,b)}$ with
\begin{equation}
g_{(a,b)}=\prod_{k=1}^{\infty} \dilog{ka}{kb}^{s_{(ka,kb)}\delta(ka,kb)}.
\end{equation}
In this CSD, (\ref{eq: consistency}) coincides with the consistency condition. The main purpose of this paper is to describe these exponents $s_{(ka,kb)}\delta(ka,kb)$ explicitly.
\par
The following property is known.
\begin{lem}[{\cite[Prop.~III.1.24.~(b)]{Nak23}}]\label{lem: the ends of ordering products}
Let $\delta_1$ and $\delta_2$ be positive integers. Then, the ordered product of $\mindilog{0}{1}^{\delta_2}\mindilog{1}{0}^{\delta_1}$ is expressed as
\begin{equation}
\dilog{1}{0}^{\delta_1}\dilog{\delta_1}{1}^{\delta_2}\cdots\dilog{1}{\delta_2}^{\delta_1}\dilog{0}{1}^{\delta_2}.
\end{equation}
\end{lem}

\section{Similarity transformations in the structure group}\label{Sec: Similarity transformations}
One of the most important properties in CSDs is the consistency condition (\ref{eq: consistency}), which is the relation in the structure group $G$. By Theorem~\ref{ordering lemma}, we may describe it by dilogarithm elements. In this section, we introduce an action on $G$, and we apply it for the consistency conditions.
\par
We fix a seed $(e_1,e_2)$ satisfying $\{e_2,e_1\}=1$, and we view $N=\mathbb{Z}^2$ and $N^{+}=\mathbb{Z}_{\geq 0}^2\backslash\{(0,0)\}$ as (\ref{eq: identities}).
For any $\left(\begin{smallmatrix}a\\b\end{smallmatrix}\right) \in N$ and matrix $F \in \mathrm{Mat}_2(\mathbb{Z})$, $F\mpmatrix{a}{b} \in N$ is defined by the usual matrix multiplication.\par
This definition and Proposition~\ref{prop: group hom} are due to Peigen Cao.
\begin{dfn}[Similarity transformations]
Let $F \in \textup{Mat}_2(\mathbb{Z}_{\geq 0})$ with $|F| \neq 0$. Then, we define the linear action of $F$ on $\mathfrak{g}$ by
\begin{equation}\label{eq: action on g}
FX = \frac{1}{|F|}\sum_{n \in N^{+}} c_nX_{Fn} \quad \bigl(X = \sum_{n \in N^{+}} c_nX_n \in \mathfrak{g},\ c_n \in \mathbb{Q}\bigr).
\end{equation}
Moreover, we define the action of $F$ on $G$ by
\begin{equation}\label{eq: action on G}
Fg=\exp(FX) \quad (g=\exp(X) \in G,\ X \in \mathfrak{g}).
\end{equation}
We call it the {\em similarity transformation} on $G$ by $F$.
\end{dfn}
\begin{prop}\label{prop: group hom}
Let $F \in \mathrm{Mat}_2(\mathbb{Z}_{\geq 0})$ with $|F|\neq0$. Then, the following statements hold.\\
\textup{(a)}\ For any $n,n' \in N^{+}$, it holds that
\begin{equation}
\left\{Fn,Fn'\right\}=|F|\{n,n'\}.
\end{equation}
\textup{(b)}\ For any $X,Y \in \mathfrak{g}$, it holds that
\begin{equation}
F[X,Y]=[FX,FY].
\end{equation}
\textup{(c)}\ For any $g,g' \in G$, it holds that
\begin{equation}
F(gg')=(Fg)(Fg').
\end{equation}
Namely, the similarity transformation by $F$ is a group homomorphism on $G$.
\end{prop}
\begin{proof}
(a)\ Let $n=(a,b)$ and $n'=(c,d)$. We have
\begin{equation}
\begin{aligned}
&\left\{F\left(\begin{matrix}a\\b\\\end{matrix}\right),F\left(\begin{matrix}c\\d\\\end{matrix}\right)\right\} \overset{(\ref{eq: skew bilinear form})}{=}  -\left|\begin{matrix}F\left(\begin{matrix}a\\b\\\end{matrix}\right)&F\left(\begin{matrix}c\\d\\\end{matrix}\right)\end{matrix}\right| = -|F|\left|\begin{matrix}a & c\\b & d\\\end{matrix}\right|\\
=&\ |F|\left\{\left(\begin{matrix}a\\b\\\end{matrix}\right),\left(\begin{matrix}c\\d\\\end{matrix}\right)\right\}.
\end{aligned}
\end{equation}
(b)\ By the linearity of this action, it suffices to show that $F[X_{n},X_{n'}]=[FX_{n},FX_{n'}]$ for any $n,n' \in N^{+}$. We have
\begin{equation}
\begin{aligned}
[FX_{n},FX_{n'}]&\overset{(\ref{eq: action on g})}=\left[\frac{1}{|F|}X_{Fn},\frac{1}{|F|}X_{Fn'}\right]=\frac{1}{|F|^2}[X_{Fn},X_{Fn'}]\\
&\overset{(\ref{eq: Lie bracket})}= \frac{1}{|F|^2}\{Fn,Fn'\}X_{F(n+n')}\overset{\textup{(a)}}=\frac{1}{|F|}\{n,n'\}X_{F(n+n')}\\
&\overset{(\ref{eq: action on g})}=F\left(\{n,n'\}X_{n+n'}\right)\overset{(\ref{eq: Lie bracket})}=F[X_n,X_{n'}].
\end{aligned}
\end{equation}
(c)\ By (b), the action on $\mathfrak{g}$ preserves the Lie bracket. It implies that the action on $G$ also preserves the product defined by the Backer-Campbell-Hausdorff formula.
\end{proof}
We apply this transformation for dilogarithm elements. For any $(a,b) \in N^{+}$, we have
\begin{equation}
\begin{aligned}
F\dilog{a}{b}&\overset{(\ref{eq: dilogarithm elements})}=F\exp\left(\sum_{j=1}^{\infty}\frac{(-1)^{j+1}}{j^2}X_{j\mpmatrix{a}{b}}\right)=\exp\left(\frac{1}{|F|}\sum_{j=1}^{\infty}\frac{(-1)^{j+1}}{j^2}X_{jF\mpmatrix{a}{b}}\right)\\
&=\exp\left(\sum_{j=1}^{\infty}\frac{(-1)^{j+1}}{j^2}X_{jF\mpmatrix{a}{b}}\right)^{1/|F|}=\left[F\begin{pmatrix}a\\b\\\end{pmatrix}\right]^{1/|F|}.
\end{aligned}
\end{equation}
Next, we try to apply this transformation for ordered products.
\begin{lem}\label{action for pentagon relation}
Let $(a,b),(c,d) \in N^{+}$, and let $F \in \mathrm{Mat}_2(\mathbb{Z}_{\geq 0})$ with $|F| > 0$. Then, the following relations hold.\\
\textup{(a).}
\begin{equation}
\deg F\left(\begin{matrix}a \\ b \\\end{matrix}\right) \geq \deg \left(\begin{matrix}a \\ b \\\end{matrix}\right).
\end{equation}
Moreover, if $F \neq \left(\begin{smallmatrix} 1 & 0\\ 0 & 1\\\end{smallmatrix}\right)$ and $a,b>0$, then
\begin{equation}
\deg F\left(\begin{matrix}a \\ b \\\end{matrix}\right) > \deg \left(\begin{matrix}a \\ b \\\end{matrix}\right).
\end{equation}
\textup{(b).}\ If $(c,d) > (a,b)$, then
\begin{equation}
F\left(\begin{matrix}c \\ d \end{matrix}\right)> F\left(\begin{matrix} a \\ b \\\end{matrix}\right).
\end{equation}
\end{lem}
In the above lamma, the order in (b) is defined in Definition~\ref{dfn: order on N+}.
\begin{proof}
(a)\ Let $F=\left(\begin{smallmatrix}\alpha & \gamma\\ \beta & \delta\end{smallmatrix}\right)$ with $\alpha,\beta,\gamma,\delta \geq 0$. Then, since $|F| \neq 0$, we have $\alpha+\beta, \gamma+\delta \geq 1$. Thus, we have
\begin{equation}\label{eq: degree lemma}
\begin{aligned}
\deg F\begin{pmatrix}a \\ b \\\end{pmatrix}
&= \deg \begin{pmatrix} a\alpha + b \gamma \\ a\beta+b\gamma \end{pmatrix}=a(\alpha+\beta)+b(\gamma+\delta)\\
&\geq a+b = \deg\begin{pmatrix}a\\b\end{pmatrix}.
\end{aligned}
\end{equation}
Hence, the first statement holds. If $|F| \neq \left(\begin{smallmatrix} 1 & 0\\ 0 & 1\\\end{smallmatrix}\right)$, either $\alpha+\beta\geq2$ or $\gamma+\delta\geq2$ holds. Thus, by a similar argument to (\ref{eq: degree lemma}), we have $\deg F\mpmatrix{a}{b} > \deg\mpmatrix{a}{b}$.
\\
(b)\ The inequality $(c,d)>(a,b)$ implies either $\{\left(\begin{smallmatrix}c \\ d\\\end{smallmatrix}\right),\left(\begin{smallmatrix}a \\ b\\\end{smallmatrix}\right)\} > 0$ or $(c,d)=k(a,b)$ with $k > 1$. If $\{\left(\begin{smallmatrix}c \\ d\\\end{smallmatrix}\right),\left(\begin{smallmatrix}a \\ b\\\end{smallmatrix}\right)\} > 0$, then $\{F\left(\begin{smallmatrix}c \\ d\\\end{smallmatrix}\right),F\left(\begin{smallmatrix}a \\ b\\\end{smallmatrix}\right)\} = |F| \{\left(\begin{smallmatrix}c \\ d\\\end{smallmatrix}\right),\left(\begin{smallmatrix}a \\ b\\\end{smallmatrix}\right)\} > 0$. If $(c,d)=k(a,b)$, then $F\binom{c}{d}=k(F\binom{a}{b})$. Thus, $F\binom{c}{d} > F\binom{a}{b}$ holds.
\end{proof}
Let $m,n \in \mathbb{Z}_{\geq 1}$. Consider the equality
\begin{equation}
\dilog{0}{1}^n\dilog{1}{0}^m = \dilog{1}{0}^m \left\{\orderedprod_{j \in J} \dilog{a_j}{b_j}^{u_j}\right\}\dilog{0}{1}^n.
\end{equation}
By Theorem~\ref{ordering lemma}, the above equality exists.
Let $F \in \mathrm{Mat}_2(\mathbb{Z}_{\geq 0})$ with $|F|>0$, and we act $F$ for the above equality. Then, by Proposition~\ref{prop: group hom}~(c), we have
\begin{equation}\label{eq: F-action on a CSD}
\left(F\dilog{0}{1}\right)^n\left(F\dilog{1}{0}\right)^m = \left(F\dilog{1}{0}\right)^m \left\{\orderedprod_{j \in J} \left(F\dilog{a_j}{b_j}\right)^{u_j}\right\}\left(F\dilog{0}{1}\right)^n.
\end{equation}
Moreover, by Lemma~\ref{action for pentagon relation}~(b), the RHS is strongly ordered. More strongly, we have the following statement.

\begin{prop}\label{prop: proceed lemma}
Let $F=\left(\begin{smallmatrix} a & c\\ b & d \\ \end{smallmatrix}\right)$ be a matrix with $(a,b),(c,d) \in N^{+}$ and $|F| > 0$. Assume $F \neq \left(\begin{smallmatrix}1 & 0\\ 0 & 1\\ \end{smallmatrix}\right)$. Let $m,n \in \mathbb{Z}_{> 0}$, and let $C = \mindilog{1}{0}^m\{\orderedprod_{j \in J} \mindilog{a_{j}}{b_{j}}^{u_{j}}\}\mindilog{0}{1}^n$ be the strongly ordered product expression of $\mindilog{0}{1}^n\mindilog{1}{0}^m$. Then, for any $l \in \mathbb{Z}_{>0}$, it holds that
\begin{equation}\label{eq: proceed lemma}
\begin{aligned}
&\ \dilog{c}{d}^{n/|F|}\dilog{a}{b}^{m/|F|} \\
\equiv&\ \dilog{a}{b}^{m/|F|}\biggl\{\orderedprod_{\substack{j \in J,\\ \deg(a_j,b_j) \leq l}}\left(F\dilog{a_j}{b_j}\right)^{u_{j}}\biggr\}\dilog{c}{d}^{n/|F|}   \mod G^{> l+1}.
\end{aligned}
\end{equation}
\end{prop}
Proposition~\ref{prop: proceed lemma} says that, if we find the strongly ordered expression of $\mindilog{0}{1}^n\mindilog{1}{0}^m$ in $G^{\leq l}$, that is,
\begin{equation}
\dilog{0}{1}^n\dilog{1}{0}^m \equiv \dilog{1}{0}^m\biggl\{\orderedprod_{\deg(a_j,b_j) \leq l} \dilog{a_j}{b_j}^{u_j}\biggr\}\dilog{0}{1}^n \mod G^{> l},
\end{equation}
then, for any anti-ordered product of the form $\mindilog{c}{d}^{n'/\gcd(a,b)}\mindilog{a}{b}^{m'/\gcd(a,b)}$ ($m',n' \in \mathbb{Z}_{>0}$) except $\mindilog{0}{1}^{n'}\mindilog{1}{0}^{m'}$, we can find the strongly ordered product expression (\ref{eq: proceed lemma}) in $G^{\leq l+1}$. The difference between $G^{\leq l}$ and $G^{\leq l+1}$ is essential in this paper.
\begin{proof}
By (\ref{eq: F-action on a CSD}), we have
\begin{equation}\label{eq: proceed lemma 1}
\begin{aligned}
&\ \dilog{c}{d}^{n/|F|}\dilog{a}{b}^{m/|F|} = \dilog{a}{b}^{m/|F|}\biggl\{\orderedprod_{\substack{j \in J}}\left(F\dilog{a_j}{b_j}\right)^{u_{j}}\biggr\}\dilog{c}{d}^{n/|F|}.
\end{aligned}
\end{equation}
By Lemma~\ref{lem: the ends of ordering products}, both $a_j$ and $b_j$ are positive integers for any $j \in J$. Moreover, by the assumptions, $F \neq \left(\begin{smallmatrix}1 & 0\\ 0 & 1\\ \end{smallmatrix}\right),\left(\begin{smallmatrix}0 & 1\\ 1 & 0\\ \end{smallmatrix}\right)$ holds. Because of Lemma~\ref{action for pentagon relation}~(b), if $\deg(a_j,b_j) \geq l+1$, then $\deg F\mpmatrix{a_j}{b_j} \geq l+2$, namely, $\pi_{l+1}(F\mindilog{a_j}{b_j})=\mathrm{id}$ holds. Thus, in $G^{\leq l+1}$, we may eliminate all factors $F\mindilog{a_j}{b_j}^{u_j}$ satisfying $\deg(a_j,b_j) \geq l+1$. Then, we have 
\begin{equation}
\prod_{\substack{j \in J}}\left(F\dilog{a_j}{b_j}\right)^{u_{j}} \equiv \prod_{\substack{j \in J,\\ \deg(a_j,b_j) \leq l}}\left(F\dilog{a_j}{b_j}\right)^{u_{j}} \mod G^{>l+1}.
\end{equation}

Putting the above expression to the (\ref{eq: proceed lemma 1}), we have
\begin{equation}
\begin{aligned}
&\ \dilog{c}{d}^{n/|F|}\dilog{a}{b}^{m/|F|} \\
\equiv&\ \dilog{a}{b}^{m/|F|}\biggl\{\orderedprod_{\substack{j \in J,\\ \deg(a_j,b_j) \leq l}}\left(F\dilog{a_j}{b_j}\right)^{u_{j}}\biggr\}\dilog{c}{d}^{n/|F|}   \mod G^{> l+1}.
\end{aligned}
\end{equation}
\end{proof}

\section{Calculation method and admissible forms of exponents}\label{Sec: Construction}
Let $(a,b) \in N^{+}$ and let $m,n \in \mathbb{Z}_{\geq 0}$.
Then, we define the rational number $u_{(a,b)}(m,n)$ as the exponent of $\mindilog{a}{b}$ in the strongly ordered product expression of $\mindilog{0}{1}^n\mindilog{1}{0}^m$. Also, we define $\tilde{u}_{(a,b)}(m,n)=d(a,b)^{-1}u_{(a,b)}(m,n)$. Namely, $\tilde{u}_{(a,b)}(m,n)$ is defined by
\begin{equation}\label{eq: definition of u(a,b)}
\dilog{0}{1}^n\dilog{1}{0}^m=\orderedprod_{(a,b) \in N^{+}} \dilog{a}{b}^{d(a,b)\tilde{u}_{(a,b)}(m,n)}.
\end{equation}
In this section, we introduce a method to calculate $u_{(a,b)}(m,n)$ as a function of $m$ and $n$, and we show the following property based on this method.
\begin{thm}\label{thm: exponents PBC}
For any $(a,b) \in N^{+}$, $\tilde{u}_{(a,b)}(m,n)$ is expressed as a nonnegative PBC in $m$ and $n$.
\end{thm}
\subsection{Calculation method}
By Theorem~\ref{ordering lemma} for $\delta_1=\delta_2=1$, for any $m,n \in \mathbb{Z}_{>0}$, $u_{(a,b)}(m,n) \in d(a,b)\mathbb{Z}_{\geq 0}$ holds, and it implies that $\tilde{u}_{(a,b)}(m,n) \in \mathbb{Z}_{\geq 0}$. Recall that $d(a,b)=1/\gcd(a,b)$ is the normalization factor of $(a,b)$ with respect to $(1,1)$.\par
\begin{ex}
Since the ordered product of $\mindilog{1}{0}^m$ is itself, we have
\begin{equation}\label{eq: exponents of n=0}
u_{(a,b)}(m,0)=\begin{cases}
m & (a,b)=(1,0),\\
0 & \textup{otherwise}.
\end{cases}
\end{equation}
By Lemma~\ref{lem: the ends of ordering products}, we have
\begin{equation}\label{eq: exponents of b=0}
u_{(a,0)}(m,n)=\begin{cases}
m & a=1,\\
0 & a\neq1,
\end{cases}
\quad
u_{(0,b)}(m,n)=\begin{cases}
n & b=1,\\
0 & b\neq0.
\end{cases}
\end{equation}
Next, we find $u_{(1,1)}(m,n)$.
By (\ref{eq: exchange of higher degree}), the term $\mindilog{1}{1}$ is commutive for every factor in $G^{\leq 2}$. Thus, we have
\begin{equation}
\begin{aligned}
\dilog{0}{1}^{n}\dilog{1}{0}^{m} &= \dilog{0}{1}^{n-1}\left(\dilog{0}{1}\dilog{1}{0}\right)\dilog{1}{0}^{m-1}\overset{(\ref{eq-pentagon relation})}= \dilog{0}{1}^{n-1}\dilog{1}{0}\dilog{1}{1}\dilog{0}{1}\dilog{1}{0}^{m-1}\\
&\overset{(\ref{eq-pentagon relation})}= \dilog{0}{1}^{n-2} \dilog{1}{0} \dilog{1}{1} \left(\dilog{0}{1} \dilog{1}{1}\right) \dilog{0}{1} \dilog{1}{0}^{m-1}\\
&\overset{(\ref{eq: exchange of higher degree})} \equiv \dilog{0}{1}^{n-2} \dilog{1}{0} \dilog{1}{1}^2 \dilog{0}{1}^2 \dilog{1}{0}^{m-1} \mod G^{>2}.
\end{aligned}
\end{equation}
By repeating this rearrengement until we obtain the strongly ordered product expression, the equality $\mindilog{0}{1}\mindilog{1}{0}=\mindilog{1}{0}\mindilog{1}{1}\mindilog{0}{1}$ is used $mn$ times, and it implies that the factor $\mindilog{1}{1}$ is produced $mn$ times. Thus, we obtain
\begin{equation}
\dilog{0}{1}^n\dilog{1}{0}^m \equiv \dilog{1}{0}^m \dilog{1}{1}^{mn} \dilog{0}{1}^n.
\end{equation}
We have
\begin{equation}\label{eq: u_(1,1)(m,n)}
u_{(1,1)}(m,n)=mn.
\end{equation}
\end{ex}
Note that Theorem~\ref{thm: exponents PBC} holds for any $(a,b)$ with $a+b=1,2$.
\par
\begin{dfn}\label{dfn: complete}
Let $C=\prod_{j \in J}\mindilog{a_j}{b_j}^{u_j}$ be a finite product. Then, we define the {\em stable part} $C^{\mathrm{stab}}$ and the {\em unstable part} $\hat{C}$ of $C$ as follows:
\begin{itemize}
\item Let $j_0$ be the largest element in $J$ such that there exists $k < j_0$ with $\mindilog{a_{k}}{b_{k}} \geq \mindilog{a_{j_0}}{b_{j_0}}$.
\item Define ${C}^{\mathrm{stab}} = \prod_{j_0 < j} \mindilog{a_j}{b_j}^{u_j}$, and $\hat{C}=\prod_{j \leq j_0} \mindilog{a_j}{b_j}^{u_j}$.
\end{itemize}
If there are no such $j_0$, or equivalently, if $C$ is strongly ordered, we define $C^{\mathrm{stab}}=C$ and $\hat{C}=\mathrm{id}$.
\end{dfn}
By definition, the following statements hold for any finite product $C$.
\begin{itemize}
\item $C=\hat{C}{C}^{\mathrm{stab}}$.
\item The stable part ${C}^{\mathrm{stab}}$ is either a strongly ordered product or $\mathrm{id}$.
\item Let $\mindilog{x}{y}$ and $\mindilog{z}{w}$ be dilogarithm elements appearing in ${C}^{\mathrm{stab}}$ and $\hat{C}$, respectively. Then, $\mindilog{z}{w} < \mindilog{x}{y}$ holds.
\end{itemize}
Let $\hat{C}'$ be the strongly ordered product expression of $\hat{C}$. Then, by Theorem~\ref{ordering lemma}, every dilogarithm element $\mindilog{z}{w}$ appearing in $\hat{C}'$ is smaller than all dilogarithm elements appearing in $C^{\mathrm{stab}}$. Thus, $\hat{C}'{C}^{\mathrm{stab}}$ is the strongly ordered product expression of $C$.\par
In order to obtain the explicit forms of $u_{(a,b)}(m,n)$ as the function of $m$ and $n$, we often consider the product
\begin{equation}\label{eq: product with polynomial}
C=\prod_{j \in \bar{J}} \dilog{a_j}{b_j}^{d(a_j,b_j)f_j(m,n)}, 
\end{equation}
where
\begin{itemize}
\item the index set $\bar{J}$ is finite.
\item for each $j \in \bar{J}$, $f_j: \mathbb{Z}_{\geq 0} \times \mathbb{Z}_{\geq 0} \to \mathbb{Z}_{\geq 0}$ is a function.
\item $m$ and $n$ are integer variables.
\end{itemize}
In this case, we view $\mindilog{a_j}{b_j}^{d(a_j,b_j)f_j(m,n)}$ as a factor of $C$ for each $j \in \bar{J}$.
\begin{lem}\label{ordering algorithm'}
Let $l \in \mathbb{Z}_{\geq 1}$. Let $C$ be a product with the above form, and let $\displaystyle{\hat{C} = \prod_{j \in J}\dmindilog{a_j}{b_j}^{d(a_j,b_j)f_j(m,n)}}$ be the unstable part of $C$. Let $\mindilog{x}{y}$ be the greatest dilogarithm element appearing in $\hat{C}$. Namely, $\mindilog{x}{y} \geq \mindilog{a_j}{b_j}$ holds for any $j \in J$. Now, we assume the following conditions:
\begin{itemize}
\item[a.] $\deg\mindilog{a_j}{b_j} \leq l+1$ for any $j \in J$.
\item[b.] If $\deg\mindilog{a_j}{b_j} \leq l$, then $f_j(m,n)$ can be expressed as a nonnegative PBC in $m$ and $n$.
\item[c.] $x \neq 0$ or $b_j \neq 0$ for any $\mindilog{a_j}{b_j} \neq \mindilog{x}{y}$.
\end{itemize}
Let $\hat{J}=J\backslash\{ j \in J \mid \mindilog{a_j}{b_j} = \mindilog{x}{y}\}$. Then, by applying Algorithm \ref{oa} below, we obtain the products 
\begin{equation}
\hat{C}'=\left(\prod_{j \in J'} \dilog{a'_j}{b'_j}^{d(a'_j,b'_j)f'_j(m,n)}\right) \dilog{x}{y}^{d(x,y)g(m,n)}
\end{equation}
and $C'=\hat{C}'{C}^{\mathrm{stab}}$ which satisfy the following conditions.
\begin{itemize}
\item[A.] $\hat{C} \equiv \hat{C}' \mod G^{> l+1}$. It implies that $C \equiv C' \mod G^{> l+1}$.
\item[B.] $\deg\mindilog{a'_j}{b'_j} \leq l+1$ and $\mindilog{a'_j}{b'_j} < \mindilog{x}{y}$. In particular, the stable part of $C'$ includes $\mindilog{x}{y}^{d(x,y)g(m,n)}{C}^{\mathrm{stab}}$.
\item[C.] The index set $\hat{J}$ can be embedded in $J'$ as an ordered set, and it satisfies the following properties:
\begin{itemize} 
\item For any $j \in \hat{J} \subset J'$, it holds that $\mindilog{a'_j}{b'_j}^{d(a'_j,b'_j)f'_j(m,n)}=\mindilog{a_j}{b_j}^{d(a_j,b_j)f_j(m,n)}$.
\item For any $j \in J'\backslash \hat{J}$, $f'_j(m,n)$ is expressed as a nonnegative PBC in $m$ and $n$.
\end{itemize}
\item[D.] Every dilogarithm element $\mindilog{a'_j}{b'_j}$ and the index set $J'$ are independent of $m$ and $n$.
\item[E.] $g(m,n)=\displaystyle{\sum_{\substack{j \in J,\\ \dmindilog{a_j}{b_j}=\dmindilog{x}{y}}} f_j(m,n)}$.
\end{itemize}
\end{lem}
\begin{algo}\label{oa}\noindent
\begin{itemize}
\item[\textbf{Step 0}.] Let $D=\hat{C}\mindilog{x}{y}^{d(x,y)g(m,n)}$ with $g(m,n)=0$.
\item[\textbf{Step 1}.] Let $\mindilog{x}{y}^{d(x,y)f(x,y)}$ be the second factor of $D$ from the right hand side such that its dilogarithm element is $\mindilog{x}{y}$. Namely, every factor $\mindilog{z}{w}^{d(z,w)f'(m,n)}$ on the right side of this $\mindilog{x}{y}^{d(x,y)f(x,y)}$ satisfies $\mindilog{z}{w} \neq \mindilog{x}{y}$ except for the factor $\mindilog{x}{y}^{d(x,y)g(m,n)}$ on the right end. Let $\mindilog{a}{b}^{d(a,b)f'(m,n)}$ be the right adjacent factor of $\mindilog{x}{y}^{d(x,y)f(x,y)}$.
\begin{itemize}
\item[\textbf{Step 1.1}.] If $\mindilog{a}{b}^{d(a,b)f'(m,n)} \neq \mindilog{x}{y}^{d(x,y)g(m,n)}$, let $F = \left(\begin{smallmatrix} a & x\\ b & y\\\end{smallmatrix}\right)$. By the assumptions, $(a,b)<(x,y)$, that is, $|F|=ay-bx \geq 0$ holds. Moreover, we have $F \neq I$ since $x \neq 0$ or $b \neq 0$. Proceed to (i) or (ii).
\begin{itemize}
\item[\textbf{(i)}.] If $|F|=0$, replace $\mindilog{x}{y}^{d(x,y)}\mindilog{a}{b}^{f'(m,n)}$ with $\mindilog{a}{b}^{d(a,b)f'(m,n)} \mindilog{x}{y}^{d(x,y)f(m,n)}$.
(Apply (\ref{exchange}).) Back to Step 1.
\item[\textbf{(ii)}.] If $|F| > 0$, replace $\mindilog{x}{y}^{d(x,y)f(m,n)}\mindilog{a}{b}^{d(a,b)f'(m,n)}$ with
\begin{equation}\label{eq: X}
\begin{aligned}
\dilog{a}{b}^{d(a,b)f'(m,n)}\biggl\{\orderedprod_{\begin{smallmatrix}p,q \in \mathbb{Z}_{\geq 1},\\ \deg(p,q) \leq l, \\ \deg F\left(\begin{smallmatrix}p\\q\end{smallmatrix}\right) \leq l+1\end{smallmatrix}}\left(F\dilog{p}{q}\right)^{v_{(p,q)}(m,n)}\biggr\}\dilog{x}{y}^{d(x,y)f(m,n)}
\end{aligned}
\end{equation}
where
\begin{equation}
v_{(p,q)}(m,n)=u_{({p},{q})}\left(d(a,b)|F|f'(m,n),d(x,y)|F|f(m,n)\right).
\end{equation}
Back to Step~1.
\end{itemize}
\item[\textbf{Step 1.2}.] If $\mindilog{a}{b}^{d(a,b)f'(m,n)}=\mindilog{x}{y}^{d(x,y)g(m,n)}$, replace $\mindilog{x}{y}^{d(x,y)f(m,n)} \mindilog{x}{y}^{d(x,y)g(m,n)}$ with $\mindilog{x}{y}^{d(x,y)(g(m,n)+f(m,n))}$, and we set $g(m,n)$ as $f(m,n)+g(m,n)$. If there exist a factor $\mindilog{z}{w}^{d(z,w)h(z,w)}$ such that $\mindilog{z}{w} = \mindilog{x}{y}$ and it is not at the right end, back to Step~1. Otherwise, proceed to Step~2.
\end{itemize}
\item[\textbf{Step 2}.] If every dilogarithm element appearing in $D$ is not $\mindilog{x}{y}$ except for the one at the right end, let $\hat{C}'=D$, and finish this algorithm.
\end{itemize}
\end{algo}
The replacement in Step~1.1~(ii) follows from the following relation and Proposition~\ref{prop: proceed lemma}.
\begin{equation}\label{eq: reason of step1.1(ii)}
\begin{aligned}
&\ \dilog{0}{1}^{d(x,y)|F|f_{j_0}(m,n)}\dilog{1}{0}^{d(a,b)|F|f'(m,n)}\\
\equiv&\ \dilog{1}{0}^{d(a,b)|F|f'(m,n)}\biggl\{\orderedprod_{\begin{smallmatrix}p,q \in \mathbb{Z}_{\geq 1},\\ \textup{deg}(p,q) \leq l\end{smallmatrix}}\dilog{p}{q}^{v_{(p,q)}(m,n)}\biggr\} \dilog{0}{1}^{d(x,y)|F|f_{j_0}(m,n)}\\
&\hspace{230pt}{\mod G^{> l}.}
\end{aligned}
\end{equation}
Roughly speaking, we change an anti-ordered pair $\mindilog{x}{y}^{d(x,y)f(m,n)}\mindilog{a}{b}^{d(a,b)f'(m,n)}$ to the strongly ordered product expression step by step, and we push $\mindilog{x}{y}^{d(x,y)f(m,n)}$ out to the right end. By Proposition~\ref{prop: proceed lemma}, this operation uses the information of  the strongly ordered product expression of $\mindilog{0}{1}^{n}\mindilog{1}{0}^m$ in $G^{\leq l}$. In particular, we do not use the data $u_{(p,q)}(m,n)$ for $p+q=l+1$.
\begin{prop}\label{prop: finish ordering algorithm}
On the assumptions of Lemma~\ref{ordering algorithm'}, Algorithm~\ref{oa} never fails, and finishes finitely many times.
\end{prop}
The number of $(x,y) \in N^{+}$ satisfying $\deg(x,y) \leq l+1$ is finite. Thus, by applying Algoithm~\ref{oa} repeatedly, we obtain the strongly ordered product expression of $C$ in $G^{\leq l+1}$ finitely many times. Moreover, by Lemma~\ref{ordering algorithm'}~C and D, the exponent of $\mindilog{a}{b}$ in the strongly ordered product expression of $C$ is
\begin{equation}
\sum_{j;\mindilog{a_{j}}{b_j}=\mindilog{a}{b}\ \mathrm{in}\ C} d(a,b)f_j(m,n) + d(a,b)f(m,n)
\end{equation}
for some nonnegative PBC $f(m,n)$.
\par
Based on this algorithm, we give a method to calculate $u_{(a,b)}(m,n)$. We can see the example of this method in Section~\ref{sec: examples}.
\begin{method}\label{method: obtain the explicit forms}
Let $l \in \mathbb{Z}_{\geq 2}$, and suppose that $\tilde{u}_{(x,y)}(m,n)$ is a nonnegative PBC for any $(x,y) \in N^{+}$ with $\deg(x,y) \leq l$. By applying Algorithm~\ref{oa} to the following products $C_{(m,1)}$ and $C_{(m,n)}$, we may calculate $u_{(a,b)}(m,n)$ with $\deg(a,b)=l+1$.\par
First, $C_{(m,1)}$ is defined as follows:
\begin{equation}\label{important relation n=1}
\begin{aligned}
\dilog{0}{1} \dilog{1}{0}^{m+1} &= \left(\dilog{0}{1}\dilog{1}{0}\right)\dilog{1}{0}^m\\
&\overset{(\ref{eq-pentagon relation})}{=}\dilog{1}{0} \dilog{1}{1} \left(\dilog{0}{1}\dilog{1}{0}^m\right)\\
&\overset{(\ref{eq: definition of u(a,b)})}{\equiv}{{\dilog{1}{0} \dilog{1}{1} \dilog{1}{0}^m\biggl(\orderedprod_{\begin{smallmatrix} (x,y) \in N^{+},\\ x+y \leq l+1, \\  x,y \geq 1,\\ \end{smallmatrix}} \dilog{x}{y}^{d(x,y)\tilde{u}_{(x,y)}(m,1)}\biggl)}}\dilog{0}{1} \\
&\hspace{170pt}\mod G^{>l+1}.
\end{aligned}
\end{equation}
Let
\begin{equation}\label{eq: C of (m,1)}
C_{(m,1)}=\dilog{1}{0} \dilog{1}{1} \dilog{1}{0}^m \biggl(\orderedprod_{\begin{smallmatrix} (x,y) \in N^{+},\\x+y \leq l+1, \\ x,y \geq 1 \end{smallmatrix}} \dilog{x}{y}^{d(x,y)\tilde{u}_{(x,y)}(m,1)}\biggl).
\end{equation}
Next, $C_{(m,n)}$ is defined as follows:
\begin{equation}\label{important relation}
\begin{aligned}
&\ \dilog{0}{1}^{n+1} \dilog{1}{0}^m = \dilog{0}{1} \left(\dilog{0}{1}^n \dilog{1}{0}^m\right)\\
\overset{(\ref{eq: definition of u(a,b)})}{\equiv}&\ \left(\dilog{0}{1} \dilog{1}{0}^m\right) \biggl(\orderedprod_{\begin{smallmatrix} (x,y) \in N^{+},\\ x+y \leq l+1,\\ x,y \geq 1 \end{smallmatrix}} \dilog{x}{y}^{d(x,y)\tilde{u}_{(x,y)}(m,n)}\biggr)\dilog{0}{1}^n\\
\overset{(\ref{eq: definition of u(a,b)})}{\equiv}&\  \dilog{1}{0}^m \biggl(\orderedprod_{\begin{smallmatrix} (z,w) \in N^{+},\\ z+w \leq l+1 \\ z,w \geq 1 \end{smallmatrix}} \dilog{z}{w}^{d(z,w)\tilde{u}_{(z,w)}(m,1)}\biggr)\dilog{0}{1}\\
&\qquad\times\biggl(\orderedprod_{\begin{smallmatrix} (x,y) \in N^{+},\\ x+y \leq l+1\\ x,y \geq 1 \end{smallmatrix}} \dilog{x}{y}^{d(x,y)\tilde{u}_{(x,y)}(m,n)}\biggr)\dilog{0}{1}^n\mod G^{>l+1}.\\
\end{aligned}
\end{equation}
Then, $C_{(m,n)}$ is defined by
\begin{equation}\label{eq: C of (m,n)}
\biggl(\orderedprod_{\begin{smallmatrix} (z,w) \in N^{+},\\ z+w \leq l+1 \\ z,w \geq 1 \end{smallmatrix}} \dilog{z}{w}^{d(z,w)\tilde{u}_{(z,w)}(m,1)}\biggr)\dilog{0}{1}\biggl(\orderedprod_{\begin{smallmatrix}(x,y) \in N^{+},\\x+y \leq l+1\\ x,y \geq 1 \end{smallmatrix}} \dilog{x}{y}^{d(x,y)\tilde{u}_{(x,y)}(m,n)}\biggr)\dilog{0}{1}^n.
\end{equation}
By Theorem~\ref{ordering lemma} and Theorem~\ref{thm: exponents PBC}, $C_{(m,1)}$ and $C_{(m,n)}$ satisfy the assumptions of Lemma~\ref{ordering algorithm'}. Let $(a,b) \in N^{+}$ with $3 \leq a+b=l+1$. Then, the factor $\mindilog{a}{b}^{*}$ in the initial $C_{(m,1)}$ is only $\mindilog{a}{b}^{d(a,b)\tilde{u}_{(a,b)}(m,1)}=\mindilog{a}{b}^{u_{(a,b)}(m,n)}$, and the ones in the initial $C_{(m,n)}$ are only $\mindilog{a}{b}^{d(a,b)\tilde{u}_{(a,b)}(m,1)}=\mindilog{a}{b}^{u_{(a,b)}(m,1)}$ and $\mindilog{a}{b}^{d(a,b)\tilde{u}_{(a,b)}(m,n)}=\mindilog{a}{b}^{u_{(a,b)}(m,n)}$. The method to calculate $u_{(a,b)}(m,n)$ is as follows:
\begin{itemize}
\item[1.] Apply Algorithm~\ref{oa} to $C_{(m,1)}$ repeatedly until $C_{(m,1)}$ becomes the strongly ordered product.
\item[2.] After the operation 1, the exponent of $\mindilog{a}{b}$ is $u_{(a,b)}(m,1)+d(a,b)f(m)$ for some nonnegative PBC $f(m)$.
Thus, we obtain the relation
\begin{equation}\label{eq: method 2 for n=1}
\begin{aligned}
&u_{(a,b)}(m+1,1)=u_{(a,b)}(m,1)+d(a,b)f(m)\\
\Leftrightarrow\ &u_{(a.b)}(m+1,1)-u_{(a,b)}(m,1)=d(a,b)f(m),
\end{aligned}
\end{equation}
and it implies that
\begin{equation}\label{eq: reccurence formula for n=1}
\begin{aligned}
u_{(a,b)}(m,1)&=u_{(a,b)}(0,1)+\sum_{j=0}^{m-1} \{u_{(a,b)}(j+1,1)-u_{(a,b)}(j,1)\}\\
&\overset{(\ref{eq: method 2 for n=1})}=u_{(a,b)}(0,1)+d(a,b)\sum_{j=0}^{m-1} f(j)\overset{(\ref{eq: exponents of n=0})}{=} d(a,b)\sum_{j=0}^{m-1} f(j).
\end{aligned}
\end{equation}
Note that this $f(m)$ is determined by the data of $u_{(x,y)}(m,n)$ for $x+y \leq l$. Thus, we may find the explicit form of $u_{(a,b)}(m,1)$. 
\item[3.] Apply Algorithm~\ref{oa} to $C_{(m,n)}$ repeatedly until $C_{(m,n)}$ becomes the strongly ordered product.
\item[4.] After the operation 3, the exponent of $\mindilog{a}{b}$ is $u_{(a,b)}(m,n)+u_{(a,b)}(m,1)+d(a,b)f'(m,n)$ for some nonnegative PBC $d(a,b)f'(m,n)$. By a similar argument, we obtain the explicit form
\begin{equation}\label{eq: reccurence formula for m,n}
\begin{aligned}
u_{(a,b)}(m,n)&=\sum_{j=0}^{n-1}u_{(a,b)}(m,1)+d(a,b)\sum_{j=0}^{n-1}f'(m,j)\\
&=u_{(a,b)}(m,1)n+d(a,b)\sum_{j=0}^{n-1}f'(m,j).
\end{aligned}
\end{equation}
\end{itemize}
\end{method}
To summarize, the following proposition holds.
\begin{prop}\label{prop: recurrence}
Let $l \in \mathbb{Z}_{\geq 1}$, and let $(a,b) \in N^{+}$ with $\deg(a,b)=l+1$. Let $C_{(m,1)}$ and $C_{(m,n)}$ be the products which is defined by (\ref{eq: C of (m,1)}) and (\ref{eq: C of (m,n)}), respectively. The following two statements hold.\\
\textup{(a)}\ By applying Algorithm~\ref{oa} to $C_{(m,1)}$ repeatedly, we obtain the recurrence relation:
\begin{equation}
u_{(a,b)}(m+1,1)=u_{(a,b)}(m,1)+d(a,b)f(m),
\end{equation}
where $f(m)$ is some nonnegative PBC in $m$.\\
\textup{(b)}\ By applying Algorithm~\ref{oa} to $C_{(m,n)}$ repeatedly, we obtain the recurrence relation:
\begin{equation}
u_{(a,b)}(m,n+1)=u_{(a,b)}(m,n)+u_{(a,b)}(m,1)+d(a,b)f'(m,n),
\end{equation}
where $f'(m,n)$ is some nonnegative PBC in $m$ and $n$.\par
Moreover, $f(m)$ and $f'(m,n)$ are determined by the data of $u_{(x,y)}(m,n)$ with $\deg(x,y)\leq l$ as functions of $m$ and $n$.
\end{prop}
\subsection{Proof of Theorem~\ref{thm: exponents PBC}}
Here, we prove Theorem~\ref{thm: exponents PBC}, Lemma~\ref{ordering algorithm'}, and Proposition~\ref{prop: finish ordering algorithm}.
We show these claims by the induction on the degree $l$. For any $l \in \mathbb{Z}_{\geq 1}$, we define the statements $(\ref{thm: exponents PBC})_{l}$, $(\ref{ordering algorithm'})_l$, and $(\ref{prop: finish ordering algorithm})_{l}$ as follows:
\begin{itemize}
\item[$(\ref{thm: exponents PBC})_{l}$] Theorem~\ref{thm: exponents PBC} holds for $\deg(a,b) \leq l$.
\item[$(\ref{ordering algorithm'})_l$] Lemma~\ref{ordering algorithm'} holds for this $l$.
\item[$(\ref{prop: finish ordering algorithm})_{l}$]  Proposition~\ref{prop: finish ordering algorithm} holds for this $l$.
\end{itemize}
The statement $(\ref{thm: exponents PBC})_{1}$ hlolds by (\ref{eq: exponents of b=0}). Thus, it suffices to show the following three statements:
\begin{equation*}
\begin{aligned}
&(\ref{thm: exponents PBC})_{l} \Rightarrow (\ref{prop: finish ordering algorithm})_{l}, \quad (\ref{thm: exponents PBC})_{l}, (\ref{prop: finish ordering algorithm})_{l} \Rightarrow (\ref{ordering algorithm'})_{l},\\
&(\ref{thm: exponents PBC})_{l}, (\ref{ordering algorithm'})_{l}, (\ref{prop: finish ordering algorithm})_{l} \Rightarrow (\ref{thm: exponents PBC})_{l+1}.
\end{aligned}
\end{equation*}
Let us start the proof of $(\ref{thm: exponents PBC})_{l} \Rightarrow (\ref{prop: finish ordering algorithm})_{l}$.
\begin{proof}[Proof of $(\ref{thm: exponents PBC})_{l} \Rightarrow (\ref{prop: finish ordering algorithm})_{l}$]
We prove that Algorithm~\ref{oa} never fails. There are two claims in Step~1.1~(ii).
\par
Since the replacement in Step~1.1~(ii) is derived from the equality (\ref{eq: reason of step1.1(ii)}), we should show that this equality holds for any $m,n \in \mathbb{Z}_{\geq 0}$. It suffices to show the following claim.\\
\\
\textbf{Claim~1.}\ Both $d(a,b)|F|f'(m,n)$ and $d(x,y)|F|f(m,n)$ are nonnegative integers for any $m,n \in \mathbb{Z}_{\geq 0}$.
\\
\par
Second, we should show that the assumptions of this algorithm hold after any replacement.
The assumption~a and c are obvious. Thus, we should show the assumption~b. If $\deg\mindilog{a}{b}=l+1$ or $\deg\mindilog{x}{y}=l+1$, the product (\ref{eq: reason of step1.1(ii)}) is $\mindilog{a}{b}^{d(x,y)f'(m,n)}\mindilog{x}{y}^{d(x,y)f(m,n)}$. Thus, the assumption~b holds. Hence, it suffices to show the following claim.
\\
\\
\textbf{Claim~2.}\ Assume $\deg(a,b), \deg(x,y) \leq l$. Consider the factor
\begin{equation}
\left(F\dilog{p}{q}\right)^{v_{(p,q)}(m,n)}=\left[F\begin{pmatrix} p \\ q\\ \end{pmatrix}\right]^{v_{(p,q)}(m,n)/|F|},
\end{equation}
where $\deg(p,q) \leq l$, $F=\left(\begin{smallmatrix} a & x \\ b & y\\ \end{smallmatrix}\right)$ and
\begin{equation}
v_{(p,q)}(m,n)=u_{({p},{q})}\left(d(a,b)|F|f'(m,n),d(x,y)|F|f(m,n)\right).
\end{equation}
Then, this exponent $\frac{v_{(p,q)}(m,n)}{|F|}$ can be expressed as $d(F\mpmatrix{p}{q})h(m,n)$ for some nonnegative PBC $h(m,n)$.
\\
\begin{proof}[Proof of Claim 1.]
We show $d(a,b)|F|f'(m,n) \in \mathbb{Z}_{\geq 0}$. By the assumptions, we have $f'(m,n) \in \mathbb{Z}$ for any $m,n \in \mathbb{Z}_{\geq 0}$. Now, we have the equality
\begin{equation}
d(a,b)|F|=\frac{a}{\gcd(a,b)}y-\frac{b}{\gcd(a,b)}x \in \mathbb{Z}_{\geq 0}.
\end{equation}
Since $d(a,b)|F|, f'(m,n) \in \mathbb{Z}_{\geq 0}$, we have $d(a,b)|F|f'(m,n) \in \mathbb{Z}_{\geq 0}$ for any $m,n \in \mathbb{Z}_{\geq 0}$.
\end{proof}
\begin{proof}[Proof of Claim 2]
We show that
\begin{equation}\label{eq: object to prove algo}
\begin{aligned}
&\ \frac{1}{d\left(F\mpmatrix{p}{q}\right)}\frac{v_{(p,q)}(m,n)}{|F|}\\
=&\ \frac{1}{d\left(F\mpmatrix{p}{q}\right)}\frac{d(p,q)\tilde{u}_{({p},{q})}\left(d(a,b)|F|f'(m,n),d(x,y)|F|f(m,n)\right)}{|F|}
\end{aligned}
\end{equation}
is a nonnegative PBC. First, we show that (\ref{eq: object to prove algo}) is expreesed as $\sum_{0 \leq k,l} \gamma_{k,l} \binom{m}{k}\binom{n}{l}$ with $\gamma_{k,l} \in \mathbb{Q}_{\geq 0}$. By the assumption~b and Claim~1, both $d(a,b)|F|f'(m,n)$ and $d(x,y)|F|f(m,n)$ are nonegative PBCs. By the assumption~$(\ref{thm: exponents PBC})_{l}$, $\tilde{u}_{(p,q)}(m,n)$ is expressed as a nonnegative PBC. Thus, by Proposition~\ref{composition preserve filling}, we may express
\begin{equation}
\begin{aligned}
&\ \tilde{u}_{(p,q)}\left(d(a,b)|F|f'(m,n),d(x,y)|F|f(m,n)\right)\\
=&\ \sum_{0 \leq k,l} \alpha_{k,l}\binom{m}{k}\binom{n}{l}
\end{aligned}
\end{equation}
for some nonnegative integers $\alpha_{k,l} \in \mathbb{Z}_{\geq 0}$. We have
\begin{equation}
\begin{aligned}
&\ \frac{1}{d\left(F\mpmatrix{p}{q}\right)}\frac{v_{(p,q)}(m,n)}{|F|}\\
=&\ \frac{d(p,q)}{|F|d(F\mpmatrix{p}{q})} \tilde{u}_{(p,q)}\left(d(a,b)|F|f'(m,n),d(x,y)|F|f(m,n)\right)\\
=&\ \frac{d(p,q)}{|F|d(F\mpmatrix{p}{q})} \sum_{0 \leq k,l} \alpha_{k,l} \binom{m}{k}\binom{n}{l}.
\end{aligned}
\end{equation}
Let $\gamma_{k,l} = \frac{d(p,q)}{|F|d(F\mpmatrix{p}{q})}\alpha_{k,l} \in \mathbb{Q}_{\geq 0}$. Then, we have $\frac{1}{d\left(F\mpmatrix{p}{q}\right)}\frac{v_{(p,q)}(m,n)}{|F|}=\sum_{0 \leq k,l} \gamma_{k,l} \binom{m}{k}\binom{n}{l}$. Thus, it suffices to show that $\frac{1}{d\left(F\mpmatrix{p}{q}\right)}\frac{v_{(p,q)}(m,n)}{|F|}$ is a PBC. Recall that $\frac{v_{(p,q)}(m,n)}{|F|}$ is the exponent of $\left[F\binom{p}{q}\right]$ in the strongly ordered product expression of (\ref{eq: X}). Consider Theorem~\ref{ordering lemma} for $\delta_1=\delta_2=1$. By the assumptions b and c, $C^{\textup{in}}=\mindilog{x}{y}^{d(x,y)(m,n)}\mindilog{a}{b}^{d(a,b)f'(m,n)}$ satisfies the assumption of Theorem~\ref{ordering lemma}. By (\ref{eq: X}) and Theorem~\ref{ordering lemma},
\begin{equation}
\frac{v_{(p,q)}(m,n)}{|F|} \in d\left(F\mpmatrix{p}{q}\right)\mathbb{Z}\ \Leftrightarrow\ \frac{1}{d(F\mpmatrix{p}{q})}\frac{v_{(p,q)}(m,n)}{|F|} \in \mathbb{Z}
\end{equation}
holds for any $m,n \in \mathbb{Z}_{\geq 0}$. So, by Lemma~\ref{nnPBClem}, $\frac{1}{d(F\mpmatrix{p}{q})}\frac{v_{(p,q)}(m,n)}{|F|}$ is a PBC. This completes the proof.
\end{proof}
Next, we show that Algorithm~\ref{oa} finishes in a finite number of steps. By applying Step~1.1~(i) and (ii), the number of factors on the right side of $\mindilog{x}{y}^{d(x,y)f(m,n)}$ decreases by 1. Moreover, the product is always finite for each operation. Thus, this algorithm finishes finitely many times.
\end{proof}
Next, we show $(\ref{thm: exponents PBC})_{l}, (\ref{prop: finish ordering algorithm})_{l} \Rightarrow (\ref{ordering algorithm'})_{l}$.
\begin{proof}[Proof of $(\ref{thm: exponents PBC})_{l}, (\ref{prop: finish ordering algorithm})_{l} \Rightarrow (\ref{ordering algorithm'})_{l}$]
The statements A, B, and D are shown by considering each step in Algorithm~\ref{oa}. Thus, we need to prove C and E.\\
C: In Step~1.1~(i) and Step~1.1~(ii), the anti-ordered pair $\mindilog{x}{y}^{d(x,y)f(m,n)}\mindilog{a}{b}^{d(a,b)f'(m,n)}$ is replaced with the strongly ordered product not changing $\mindilog{a}{b}^{d(a,b)f'(m,n)}$. Thus, every factor $\mindilog{a_j}{b_j}^{d(a_j,b_j)f_{j}(m,n)}$ ($(a_j,b_j) \neq (x,y)$) in $\hat{C}$ exists in $D$. Moreover, for any factors $\mindilog{a_i}{b_i}^{d(a_i,b_i)f_i(m,n)}$ and $\mindilog{a_{j}}{b_j}^{d(a_j,b_j)f_j(m,n)}$ such that $(a_i,b_i),(a_jb_j)\neq (x,y)$, if $\mindilog{a_i}{b_i}^{d(a_i,b_i)f_i(m,n)}$ appears before $\mindilog{a_{j}}{b_j}^{d(a_j,b_j)f_j(m,n)}$ in $\hat{C}$, then $\mindilog{a_i}{b_i}^{d(a_i,b_i)f_i(m,n)}$ appears before $\mindilog{a_{j}}{b_j}^{d(a_j,b_j)f_j(m,n)}$ in $D$. Furthermore, this algorithm does not finish until all $\mindilog{x}{y}^{*}$ disappear in $D$ except for the right end. Thus, the index set $\hat{J}$ can be embedded in $J'$ preserving its order, and for any $j \in \hat{J}$, $\mindilog{a'_j}{b'_j}^{d(a'_j,b'_j)f'_j(m,n)}=\mindilog{a_j}{b_j}^{d(a_j,b_j)f_j(m,n)}$ holds. Let $j \in J'\backslash\hat{J}$. Then, by the proof of Claim~2, $f'_j(m,n)$ is expressed as a nonnegative PBC.
\\
E: In Step~1.1~(i) and (ii), every factor $\mindilog{x}{y}^{d(x,y)f(m,n)}$ in $D$ moves to the last $\mindilog{x}{y}^{d(x,y)g(m,n)}$ without changing its exponent, and new factors $\mindilog{x}{y}^{*}$ are not produced. Thus, E holds.
\end{proof}
Last, we prove $(\ref{thm: exponents PBC})_{l}, (\ref{ordering algorithm'})_{l}, (\ref{prop: finish ordering algorithm})_{l} \Rightarrow (\ref{thm: exponents PBC})_{l+1}$.
\begin{proof}[Proof of $(\ref{thm: exponents PBC})_{l}, (\ref{ordering algorithm'})_{l}, (\ref{prop: finish ordering algorithm})_{l} \Rightarrow (\ref{thm: exponents PBC})_{l+1}$]
This is immediately shown by (\ref{eq: reccurence formula for n=1}) and (\ref{eq: reccurence formula for m,n}). Note that, for any PBC $f(m,n)=\sum_{0\leq k,l} \alpha_{k,l} \binom{m}{k}\binom{n}{l}$ in $m$ and $n$, $\sum_{j=0}^{n-1} f(m,j)$ is also expressed as a PBC as follows:
\begin{equation}
\sum_{j=0}^{n-1} f(m,j)=\sum_{0\leq k,l} \alpha_{k,l} \binom{m}{k} \sum_{j=0}^{n-1} \binom{j}{l} \overset{(\ref{sum of BCs})}= \sum_{0\leq k,l} \alpha_{k,l} \binom{m}{k}\binom{n}{l+1}.
\end{equation}
\end{proof}

\section{Examples in lower degrees}\label{sec: examples}
In Section~\ref{Sec: Construction}, we introduced a method to calculate $u_{(a,b)}(m,n)$ (Method~\ref{method: obtain the explicit forms}). In this section, we see some examples.
\begin{ex}\label{ex: explicit forms}
By (\ref{eq: u_(1,1)(m,n)}), we obtain
\begin{equation}\label{eq: mod 2}
\dilog{0}{1}^{n}\dilog{1}{0}^{m} \equiv \dilog{1}{0}^{m}\dilog{1}{1}^{mn}\dilog{0}{1}^{n} \mod G^{>2}.
\end{equation}
For later, we write underlines on anti-ordered pairs where we apply the retation. First, consider $a+b=3$. Let $u_{m,n}=u_{({2},{1})}(m,n)$ and $v_{m,n}=u_{({1},{2})}(m,n)$, namely,
\begin{equation}
\dilog{0}{1}^n \dilog{1}{0}^m \equiv \dilog{1}{0}^m \dilog{2}{1}^{u_{m,n}} \dilog{1}{1}^{mn} \dilog{1}{2}^{v_{m,n}} \dilog{0}{1}^n \quad \mod G^{>3}.
\end{equation}
Then, for any $m \in \mathbb{Z}_{\geq 0}$, we obtain
\begin{equation}\label{eq: conc. ex. mod 3. A}
\begin{aligned}
\dilog{0}{1} \dilog{1}{0}^{m+1} 
&\overset{(\ref{important relation n=1})}{\equiv} \dilog{1}{0} \underline{\dilog{1}{1} \dilog{1}{0}^m} \dilog{2}{1}^{u_{m,1}} \dilog{1}{1}^{m} \dilog{1}{2}^{v_{m,1}} \dilog{0}{1}\\
\end{aligned}
\end{equation}
Now, we apply Step~1.1~(ii) to $\mindilog{1}{1}\mindilog{1}{0}^m$. Namely, let $F=\left(\begin{smallmatrix}1 & 1 \\ 0 & 1 \\ \end{smallmatrix}\right)$, and we view
\begin{equation}
\dilog{1}{1} \dilog{1}{0}^m = F\left(\dilog{0}{1}\dilog{1}{0}^m\right).
\end{equation}
Then, by using Proposition~\ref{prop: proceed lemma} and the result of (\ref{eq: mod 2}), it holds that
\begin{equation}
\begin{aligned}
\dilog{1}{1}\dilog{1}{0}^m &\equiv F\left(\dilog{1}{0}^m \dilog{1}{1}^{m}\dilog{0}{1}\right)\\
&\equiv \dilog{1}{0}^m \dilog{2}{1}^m \dilog{1}{1}\quad \mod G^{>3}.
\end{aligned}
\end{equation}
Putting this relation to the last line of (\ref{eq: conc. ex. mod 3. A}), we have
\begin{equation}
\dilog{0}{1} \dilog{1}{0}^{m+1} \equiv \dilog{1}{0}^{m+1} \dilog{2}{1}^m \underline{\dilog{1}{1} \dilog{2}{1}^{u_{m,1}}} \dilog{1}{1}^{m} \dilog{1}{2}^{v_{m,1}} \dilog{0}{1} \quad \mod G^{>3}.
\end{equation}
Next, we apply Step~1.1~(ii) to $\mindilog{1}{1} \mindilog{2}{1}^{u_{m,1}}$. Since $\deg((1,1)+(2,1))=\deg(3,2) > 3$, we have 
\begin{equation}\label{eq: high degree}
\dilog{1}{1}\dilog{2}{1}^{u_{m,1}}\equiv\dilog{2}{1}^{u_{m,1}}\dilog{1}{1} \mod G^{>3}.
\end{equation}
Thus, we have
\begin{equation}
\dilog{0}{1}\dilog{1}{0}^{m+1}{\equiv} \dilog{1}{0}^{m+1} \dilog{2}{1}^{m+u_{m,1}} \dilog{1}{1}^{m+1} \dilog{1}{2}^{v_{m,n}} \dilog{0}{1}\quad \mod G^{>3}.
\end{equation}
The RHS is strongly ordered. So, we have $u_{m+1,1}=m+u_{m,1}$ and $v_{m+1,1}=v_{m,1}$. Moreover, we obtain
\begin{eqnarray}
&&u_{m,1}=u_{0,1}+\sum_{k=0}^{m-1} (u_{k+1,1}-u_{k,1})=\sum_{k=0}^{m-1} \binom{k}{1} \overset{(\ref{sum of BCs})}{=} \binom{m}{2},\\
&&v_{m,1}=v_{m-1,1}=\cdots=v_{0,1}=0.
\end{eqnarray}
Next, by (\ref{important relation}), we have
\begin{equation}
\begin{aligned}
\dilog{0}{1}^{n+1}\dilog{1}{0}^m \equiv \dilog{1}{0}^m \dilog{2}{1}^{u_{m,1}}\dilog{1}{1}^{m}\dilog{1}{2}^{v_{m,1}}\underline{\dilog{0}{1} \dilog{2}{1}^{u_{m,n}}} \dilog{1}{1}^{mn} \dilog{1}{2}^{v_{m,n}} \dilog{0}{1}^{n}\\ 
\hspace{200pt}\mod G^{>3}.
\end{aligned}
\end{equation}
By applying Step~1.1~(ii) to $\mindilog{0}{1} \mindilog{2}{1}^{u_{m,n}}$, we have
\begin{equation}\label{eq: ex m,n,3}
\begin{aligned}
\dilog{0}{1}^{n+1}\dilog{1}{0}^m \equiv \dilog{1}{0}^m \dilog{2}{1}^{u_{m,1}}\dilog{1}{1}^{m}\dilog{1}{2}^{v_{m,1}}\dilog{2}{1}^{u_{m,n}} \underline{\dilog{0}{1} \dilog{1}{1}^{mn}} \dilog{1}{2}^{v_{m,n}} \dilog{0}{1}^{n}\\ 
\hspace{200pt}\mod G^{>3}.
\end{aligned}
\end{equation}
Next, apply Step~1.1~(ii) to $\mindilog{0}{1} \mindilog{1}{1}^{mn}$. Let $F=\left(\begin{smallmatrix} 1 & 0 \\ 1 & 1\\ \end{smallmatrix}\right)$, and we view
\begin{equation}
\dilog{0}{1} \dilog{1}{1}^{mn} = F\left(\dilog{0}{1}\dilog{1}{0}^{mn}\right).
\end{equation}
Thus, it holds that
\begin{equation}
\begin{aligned}
\dilog{0}{1}\dilog{1}{1}^{mn} &\equiv F\left(\dilog{1}{0}^{mn}\dilog{1}{1}^{mn}\dilog{0}{1}\right)\\
&\equiv \dilog{1}{1}^{mn} \dilog{1}{2}^{mn} \dilog{0}{1}  \mod G^{>3}.
\end{aligned}
\end{equation}
Putting the last line to (\ref{eq: ex m,n,3}), we have
\begin{equation}
\begin{aligned}
&\dilog{0}{1}^{n+1}\dilog{1}{0}^m\\
\equiv& \dilog{1}{0}^m \dilog{2}{1}^{u_{m,1}}\dilog{1}{1}^{m}\dilog{1}{2}^{v_{m,1}}\dilog{2}{1}^{u_{m,n}}\dilog{1}{1}^{mn} \dilog{1}{2}^{mn} \dilog{0}{1} \dilog{1}{2}^{v_{m,n}} \dilog{0}{1}^{n}\\
&\hspace{230pt}\mod G^{>3}.
\end{aligned}
\end{equation}
Then, by a similar discussion of (\ref{eq: high degree}), we make this product strongly ordered by only exchange relations. Thus, we have
\begin{equation}
\begin{aligned}
\dilog{0}{1}^{n+1} \dilog{1}{0}^m &\equiv \dilog{1}{0}^m \dilog{2}{1}^{u_{m,1}+u_{m,n}} \dilog{1}{1}^{m(n+1)} \dilog{1}{2}^{mn+v_{m,n}} \dilog{0}{1}^{n+1}\\
&\hspace{180pt}\mod G^{>3}.
\end{aligned}
\end{equation}
It implies that
\begin{eqnarray}
u_{m,n}=\binom{m}{2}n,\\
v_{m,n}=m\binom{n}{2},
\end{eqnarray}
and
\begin{equation}
\dilog{0}{1}^n \dilog{1}{0}^m \equiv \dilog{1}{0}^m \dilog{2}{1}^{\binom{m}{2}n} \dilog{1}{1}^{mn} \dilog{1}{2}^{m\binom{n}{2}} \dilog{0}{1}^n\quad \mod G^{>3}.
\end{equation}
For any $a,b \in \mathbb{Z}_{> 0}$, we can obtain $u_{(a,b)}(m,n)$ by the above method in principle. However, it becomes harder to complete this calculation when its degree $a+b$ is larger. The  following relation is the result of $G^{\leq 7}$.
\begin{equation}\label{example deg=7}
\begin{aligned}
&\ \dilog{0}{1}^n \dilog{1}{0}^m\\
\equiv&\ \dilog{1}{0}^{m} \dilog{6}{1}^{\bi{m}{6}\bi{n}{1}} \dilog{5}{1}^{\bi{m}{5}\bi{n}{1}} \dilog{4}{1}^{\bi{m}{4}\bi{n}{1}} \dilog{3}{1}^{\bi{m}{3}\bi{n}{1}}\\
\times&\ \dilog{5}{2}^{3\bi{m}{3}\bi{n}{2}+4\bi{m}{4}\bi{n}{1}+24\bi{m}{4}\bi{n}{2}+7\bi{m}{5}\bi{n}{1}+30\bi{m}{5}\bi{n}{2}}\\
\times&\ \dilog{2}{1}^{\bi{m}{2}\bi{n}{1}} \dilog{4}{2}^{6\bi{m}{3}\bi{n}{2}+2\bi{m}{4}\bi{n}{1}+12\bi{m}{4}\bi{n}{2}} \dilog{3}{2}^{\substack{2\bi{m}{2}\bi{n}{2}+\bi{m}{3}\bi{n}{1}\\\quad+6\bi{m}{3}\bi{n}{2}}}\\
\times&\ \dilog{4}{3}^{\substack{2\bi{m}{2}\bi{n}{2}+8\bi{m}{2}\bi{n}{3}+30\bi{m}{3}\bi{n}{2}+72\bi{m}{3}\bi{n}{3}\\\quad+\bi{m}{4}\bi{n}{1}+48\bi{m}{4}\bi{n}{2}+96\bi{m}{4}\bi{n}{3}}}\\
\times&\ \dilog{1}{1}^{\bi{m}{1}\bi{n}{1}} \dilog{2}{2}^{2\bi{m}{2}\bi{n}{2}} \dilog{3}{3}^{6\bi{m}{2}\bi{n}{3}+6\bi{m}{3}\bi{n}{2}+18\bi{m}{3}\bi{n}{3}} \\
\times&\ \dilog{3}{4}^{\substack{\bi{m}{1}\bi{n}{4}+2\bi{m}{2}\bi{n}{2}+30\bi{m}{2}\bi{n}{3}+48\bi{m}{2}\bi{n}{4}\\\quad+8\bi{m}{3}\bi{n}{2}+72\bi{m}{3}\bi{n}{3}+96\bi{m}{3}\bi{n}{4}}}\\
\times&\ \dilog{2}{3}^{\bi{m}{1}\bi{n}{3}+2\bi{m}{2}\bi{n}{2}+6\bi{m}{2}\bi{n}{3}} \dilog{1}{2}^{\bi{m}{1}\bi{n}{2}} \dilog{2}{4}^{\substack{2\bi{m}{1}\bi{n}{4}+6\bi{m}{2}\bi{n}{3}\\\quad+12\bi{m}{2}\bi{n}{4}}}\\
\times&\ \dilog{2}{5}^{4\bi{m}{1}\bi{n}{4}+7\bi{m}{1}\bi{n}{5}+3\bi{m}{2}\bi{n}{3}+24\bi{m}{2}\bi{n}{4}+30\bi{m}{2}\bi{n}{5}}\\
\times&\ \dilog{1}{3}^{\bi{m}{1}\bi{n}{3}} \dilog{1}{4}^{\bi{m}{1}\bi{n}{4}} \dilog{1}{5}^{\bi{m}{1}\bi{n}{5}} \dilog{1}{6}^{\bi{m}{1}\bi{n}{6}} \dilog{0}{1}^{n} \mod G^{>7}.
\end{aligned}
\end{equation}
\end{ex}
\begin{rem}
In \cite[Theorem 5.2]{Rei12}, when $\delta_1=\delta_2=\delta$, it was shown that there exists the wall $(\sigma(1,-1),f)_{(1,1)}$ in a CSD $\mathfrak{D}_{\delta,\delta}$, where
\begin{equation}\label{eq: result of Rei12}
\begin{aligned}
&\ f =\left(\sum_{k=0}^{\infty} \frac{1}{(\delta^2-2\delta)k+1}\binom{(\delta-1)^2k}{k}x^{k(-\delta,\delta)}\right)^{\delta}\\
=&\ \bigl(1+x^{(-\delta,\delta)}+(\delta-1)^2x^{2(-\delta,\delta)}+\frac{1}{2}(\delta-1)^2(3\delta^2-6\delta+2)x^{3(-\delta,\delta)}+\cdots\bigr)^{\delta}.
\end{aligned}
\end{equation}
This function $f \in \mathbb{Q}((x_1,x_2))$ corresponds to the element $g \in G$ by the group homomorphism induced by the following map \cite[Lemma~1.3]{GHKK18}.
\begin{equation}
\begin{aligned}
G &\to \mathbb{Q}((x_1,x_2)),\\
\Psi[n] &\mapsto (1+x^{p^*(n)})^{\delta(n)},
\end{aligned}
\end{equation}
where $\delta(n)$ is the normalization factor with respect to $(\delta,\delta)$, that is, $\delta(a,b)=\delta/\gcd(a,b)$. Then, in (\ref{example deg=7}), the wall element
\begin{equation}
g=\dilog{1}{1}^{\delta^2} \dilog{2}{2}^{2\bi{\delta}{2}^2} \dilog{3}{3}^{12\bi{\delta}{2}\bi{\delta}{3}+18\bi{\delta}{3}^2}\cdots
\end{equation}
corresponds to
\begin{equation}
\begin{aligned}
&\ (1+x^{(-\delta,\delta)})^{\delta}(1+x^{2(-\delta,\delta)})^{\frac{2}{\delta}\times2\binom{\delta}{2}^2}(1+x^{3(-\delta,\delta)})^{\frac{3}{\delta}\times\left(12\bi{\delta}{2}\bi{\delta}{3}+18\bi{\delta}{3}^2\right)}\cdots\\
=&\ \bigl(1+x^{(-\delta,\delta)}+(\delta-1)^2x^{2(-\delta,\delta)}+\frac{1}{2}(\delta-1)^2(3\delta^2-6\delta+2)x^{3(-\delta,\delta)}+\cdots\bigr)^{\delta}.
\end{aligned}
\end{equation}
So, this result agrees with the result of (\ref{eq: result of Rei12}) in the lower degrees.
\end{rem}

\section{Bounded property of PBCs for ordered products}\label{Sec: Bounded property of PBCs for ordered products}
In Section~\ref{Sec: Construction}, we showed that every exponent $u_{(a,b)}(m,n)$ is essentially expressed as a nonnegative PBC in $m$ and $n$. In this section, we give a property about its degree.
\par
Recall that, by Theorem~\ref{thm: exponents PBC}, we can express
\begin{equation}\label{eq: non-bounded expression}
u_{(a,b)}(m,n)=d(a,b)\sum_{\substack{0 \leq i \leq A,\\ 0 \leq j \leq B}} \alpha_{(a,b)}(i,j) \binom{m}{i}\binom{n}{j}
\end{equation}
for some $\alpha_{(a,b)}(i,j),A,B \in \mathbb{Z}_{\geq 0}$. More strongly, we have the following statement.
\begin{thm}[Bounded property]\label{main thm1}
Let $a$ and $b$ be positive integers. Then, we can express
\begin{equation}
u_{(a,b)}(m,n) = d(a,b) \sum_{\substack{1 \leq i \leq a,\\ 1 \leq j \leq b}} \alpha_{(a,b)}(i,j) \binom{m}{i} \binom{n}{j},
\end{equation}
where $\alpha_{(a,b)}(i,j)$ are nonnegative integers independent of $m$ and $n$.
\end{thm}
Namely, $(i,j)$ can be restricted as $1 \leq i \leq a$ and $1 \leq j \leq b$ in the sum of (\ref{eq: non-bounded expression}).
\par
We show the following two claims.\\
\textbf{Claim~1.}\ For any $i=0,1,\dots,A$ and $j=0,1,\dots,B$, it holds that $\alpha_{(a,b)}(i,0)=\alpha_{(a,b)}(0,j)=0$.
\\
\textbf{Claim~2.}\ For any $(k,l) \in \mathbb{Z}_{\geq 0}^2$ with $k> a$ or $l >b$, it holds that $\alpha_{(a,b)}(k,l)=0$. 
\begin{proof}[Proof of Claim~1]
By (\ref{eq: exponents of n=0}), we have
\begin{equation}
u_{(a,b)}(A,0)=0.
\end{equation}
By (\ref{eq: non-bounded expression}), we have
\begin{equation}
u_{(a,b)}(A,0)=d(a,b)\sum_{\substack{0 \leq i \leq A}}\alpha_{(a,b)}(i,0)\binom{A}{i}=0.
\end{equation}
Since $\alpha_{(a,b)}(i,0) \geq 0$ and $\binom{A}{i}>0$, we obtain $\alpha_{(a,b)}(i,0)=0$ for any $i=0,1,\dots,A$. Similarly, we have $\alpha_{(a,b)}(0,j)=0$ for any $j=0,1,\dots,B$. Thus, we have
\begin{equation}
u_{(a,b)}(m,n) = d(a,b)\sum_{\substack{1 \leq i \leq A,\\ 1 \leq j \leq B}}\alpha_{(a,b)}(i,j) \binom{m}{i} \binom{n}{j}.
\end{equation}
\end{proof}
Before proving Claim~2, we show the following lemma.
\begin{lem}\label{lem: the form of algo}
Let $l \in \mathbb{Z}_{\geq 1}$, and let $C_{(m,n)}$ be the product which is defined by (\ref{eq: C of (m,n)}) in $G^{\leq l+1}$. Let $\mindilog{x}{y}$ be a dilogarithm element with $x+y \leq l+1$, and let $\mindilog{z}{w}$ be the greatest dilogarithm element such that $\mindilog{z}{w}<\mindilog{x}{y}$ and $z+w \leq l+1$. Let $D$ be the product which is obtained by applying Algorithm~\ref{oa} to $C_{(m,n)}$ repeatedly until $\mindilog{x}{y}^{*}$ is in the stable part. Then, the form of $D$ is as follows:
\begin{equation}
D=\cdots\dilog{z}{w}^{u_{(z,w)}(m,n)}D^{\mathrm{stab}}.
\end{equation}
\end{lem}
In the above expression, $D^{\mathrm{stab}}=\mindilog{x}{y}^{*}\cdots\mindilog{0}{1}^{n+1}$ is the stable part of $D$, which is defined in Definition~\ref{dfn: complete}.
\begin{proof}
Since $\mindilog{z}{w}$ is the greatest element such that $\mindilog{z}{w} < \mindilog{x}{y}$ and $z+w \leq l+1$, the form of initial $C_{(m,n)}$ is
\begin{equation}
\cdots \dilog{z}{w}^{u_{(z,w)}(m,n)}\dilog{x}{y}^{u_{(x,y)}(m,n)} \prod_{(x,y)<(p,q)} \dilog{p}{q}^{u_{(p,q)}(m,n)}.
\end{equation}
In particular, every dilogarithm element on the right side of $\mindilog{z}{w}^{u_{(z,w)}(m,n)}$ is greater than $\mindilog{z}{w}$. When we apply Algorithm~\ref{oa} repeatedly, every factor $\mindilog{u}{v}^{*}$ appearing in the right side of $\mindilog{z}{w}^{u_{(z,w)}(m,n)}$ satisfies $\mindilog{z}{w}<\mindilog{u}{v}$. Thus, if the form of $D$ is $\cdots\mindilog{z}{w}^{u_{(z,w)}(m,n)}\cdots\mindilog{u}{v}^{*}\cdots{D}^{\mathrm{stab}}$, it holds that $\mindilog{z}{w}<\mindilog{u}{v}<\mindilog{x}{y}$. (Note that ${D}^{\mathrm{stab}}$ has the factor $\mindilog{x}{y}^{*}$.) It contradicts the assumptions of $\mindilog{z}{w}$.
\end{proof}
Let us prove Caim~2. Let $f(m,n)$ be a polynomial or a PBC. We write the degree of $f(m,n)$ as a polynomial in $n$ by $\deg_n(f(m,n))$.
\begin{proof}[Proof of Claim~2]
Suppose that the claim does not hold; in other words, there exists $\alpha_{(a,b)}(k,l) > 0$ such that $k > a$ or $l > b$. Suppose $a+b$ is smallest among such $(a,b)$, and there exists $l > b$ such that $\alpha_{(a,b)}(k,l)>0$. (If there exists such $k > a$, we can do a similar argument.) Then, since
\begin{equation}
\begin{aligned}
&\ u_{(a,b)}(m,n+1)-u_{(a,b)}(m,n)\\
=&\ d(a,b)\sum_{\substack{1 \leq i \leq A,\\ 1 \leq j \leq B}} \alpha_{(a,b)}(i,j)\binom{m}{i}\left\{\binom{n+1}{j}-\binom{n}{j}\right\}\\
\overset{(\ref{a})}{=}&\ d(a,b)\sum_{\substack{1 \leq i \leq A,\\ 1 \leq j \leq B}} \alpha_{(a,b)}(i,j)\binom{m}{i}\binom{n}{j-1}\\
=&\ d(a,b)\alpha_{(a,b)}(k,l)\binom{m}{k}\binom{n}{l-1}\\
&\ \qquad+d(a,b)\sum_{\substack{1 \leq i \leq A,\ 1 \leq j \leq B,\\ (i,j)\neq(k,l)}} \alpha_{(a,b)}(i,j)\binom{m}{i}\binom{n}{j-1},
\end{aligned}
\end{equation}
$u_{(a,b)}(m,n+1)-u_{(a,b)}(m,n)$ has a factor $\binom{m}{k}\binom{n}{l-1}$ with a positive coefficient. In particular, it holds that
\begin{equation}
\deg_{n}(u_{(a,b)}(m,n+1)-u_{(a,b)}(m,n)) \geq l-1 \geq b.
\end{equation}
Now, we apply Algorithm~\ref{oa} to $C_{(m,n)}$ repeatedly. Then, by Proposition~\ref{prop: recurrence}, we have a following relation:
\begin{equation}
\begin{aligned}
u_{(a,b)}(m,n+1)&=u_{(a,b)}(m,n)+u_{(a,b)}(m,1)\\
&\ \qquad+(\textup{factors produced in Algorithm~\ref{oa}}).
\end{aligned}
\end{equation}
Because $\deg_n(u_{(a,b)}(m,n+1)-u_{(a,b)}(m,n)) \geq b$, there exists an anti-ordered pair $\mindilog{z}{w}^{g}\mindilog{x}{y}^{f}$ which produces $\mindilog{a}{b}^{h}$ with $\deg_n(h) \geq b$. If $x+y=a+b$ or $z+w=a+b$, then this anti-ordered pair does not produce new factors. Assume $x+y,z+w<a+b$. Let $F=\left(\begin{smallmatrix} x & z\\ y & w\\ \end{smallmatrix}\right)$. Then, Step~1.1~(ii) in Algorithm~\ref{oa} bocomes
\begin{equation}
\begin{aligned}
&\ \dilog{z}{w}^{g}\dilog{x}{y}^{f}\\
\equiv&\ \dilog{x}{y}^{f} \biggl\{\orderedprod_{\substack{(p,q) \in \mathbb{Z}_{\geq 1}^2,\\ \deg(p,q) \leq a+b-1,\\ \deg F\binom{p}{q}\leq a+b}} \left(F\dilog{p}{q}\right)^{u_{(p,q)}(|F|f,|F|g)} \biggr\}\dilog{z}{w}^{f} \mod G^{>a+b}. 
\end{aligned}
\end{equation}
Since the above product on the RHS has a factor $\mindilog{a}{b}^{h}$, there exists $(p,q) \in \mathbb{Z}_{\geq 1}^2$ satisfying
\begin{equation}\label{eq: bounded 1}
F\begin{pmatrix}{p}\\{q}\end{pmatrix}=\begin{pmatrix}{a}\\{b}\end{pmatrix} \Longleftrightarrow \quad \left\{\begin{aligned}
px+qz&=a,\\
py+qw&=b.\\
\end{aligned}\right.
\end{equation}
By using the above $(p,q)$, the exponent of $\mindilog{a}{b}$ is $h=u_{(p,q)}{(|F|f,|F|g)}/|F|$.
Moreover, since $\deg_n(h) \geq b$, we have $\deg_n(u_{(p,q)}(|F|f,|F|g)) \geq b$. Since $\deg(p,q) < \deg(a,b)$ and the smallest assumption of $(a,b)$, we have $\deg_m(u_{(p,q)}(m,n)) \leq p$ and $\deg_n(u_{(p,q)}(m,n)) \leq q$. Now, let $\deg_n (f)=t$ and $\deg_n(g)=t'$. Then, $\deg_n(u_{(p,q)}(|F|f,|F|g)) \leq tp+t'q$. Hence, we have 
\begin{equation}\label{eq: contradiction of main thm}
tp+t'q \geq b.
\end{equation}
On the other hand, since $\deg_{n}(u_{(x,y)}(m,n)) \leq y$, we have $\deg_n(g)=t' \leq y$. Similary, $\deg_n(f)=t \leq w$ holds. Moreover, since $p,q \geq 1$, we have
\begin{equation}
tp+t'q \leq yp+wq\overset{(\ref{eq: bounded 1})}=b.
\end{equation}
Combining the above two inequations, we have $tp+t'q=b$, and it implies $t=y$ and $t'=w$. In particular, $\deg_n(g)=w$ holds. To summarize, $\mindilog{a}{b}^{h}$ is produced when $\mindilog{z}{w}^{g}$ moves to the right hand side, and $\deg_n(g) \geq w$. We apply Algorithm~\ref{oa} to $C_{(m,n)}$ until the greatest dilogarithm element appearing in its unstable part is $\mindilog{z}{w}$. Then, by Lemma~\ref{lem: the form of algo}, $C_{(m,n)}$ becomes
\begin{equation}
\cdots\dilog{z}{w}^{g}\cdots\dilog{z}{w}^{u_{(z,w)}(m,n)}\orderedprod_{(u,v)>(z,w)}\dilog{u}{v}^{u_{(u,v)}(m,n+1)}.
\end{equation}
It implies that
\begin{equation}
u_{(z,w)}(m,n+1)=u_{(z,w)}(m,n)+g+\cdots.
\end{equation}
Thus, $\deg_n(u_{(z,w)}(m,n)) > \deg_n(g)\geq w$ holds. It contradicts $\deg(z,w)<\deg(a,b)$.
\end{proof}
In Section~\ref{subsec: properties of coefficients}, we can see some properties for these coefficients $\alpha_{(a,b)}(i,j)$. 
\section{Further results and examples}
\subsection{Inverse formula}\label{subsec: inverse formula}
Thanks to Theorem~\ref{main thm1}, $u_{(a,b)}(m,n)$ may be recovered from the special values $u_{(a,b)}(i,j)$ for $1 \leq i \leq a$ and $1 \leq j \leq b$.\par
Let $I=\{(i,j) \in \mathbb{Z}\mid1 \leq i \leq a, 1 \leq j \leq b\}$, and let $u_{(a,b)}(m,n)=\sum_{\substack{1 \leq k \leq a,\\ 1 \leq l \leq b}} \alpha_{k,l} \binom{m}{k}\binom{n}{l}$. Namely, by using the notation of Theorem~\ref{main thm1}, we write $\alpha_{k,l}=d(a,b)\alpha_{(a,b)}(k,l)$.\par
First, we define the following three matrices:
\begin{eqnarray}
\alpha&=&(\alpha_{i,j})_{(i,j) \in I} \in \mathrm{Mat}_{I\times1}(\mathbb{Q}),\\
A&=&\left(\binom{i}{i'}\binom{j}{j'}\right)_{((i,j),(i',j')) \in I \times I}, \label{eq: matrix A}\\
u&=&(u_{(a,b)}(i,j))_{(i,j) \in I} \in \mathrm{Mat}_{I\times1}(\mathbb{Q}).
\end{eqnarray}
These three matrices have the relation
\begin{equation}
A\alpha=\left(\sum_{(i',j') \in I} \alpha_{i',j'}\binom{i}{i'}\binom{j}{j'}\right)_{(i,j) \in I}=(u_{(a,b)}(i,j))_{(i,j) \in I}=u.
\end{equation}
This implies
\begin{equation}\label{eq: inverse formula as matrix}
\alpha=A^{-1}u.
\end{equation}
Next, we obtain the inverse matrix $A^{-1}$. We define $P_a \in \mathrm{Mat}_{a}(\mathbb{Z})$ as follows:
\begin{equation}
P_a=\left(\binom{i}{j}\right)_{(i,j)}.
\end{equation}
Similarly, we define $P_b \in \mathrm{Mat}_b(\mathbb{Z})$. Then, $A=P_a \otimes P_b$ holds, where $P_a \otimes P_b =\left(\binom{i}{i'}\binom{j}{j'}\right)_{((i,j),(i',j'))} \in \mathrm{Mat}_{I}(\mathbb{Z})$ is the tensor product. The inverse matrix of $P_a$ is known as follows.
\begin{lem}[{e.g., \cite[Thm.~37]{Spi19}}]
For any $a \in \mathbb{Z}_{>0}$, we have
\begin{equation}
P_a^{-1}=\left((-1)^{i+j}\binom{i}{j}\right).
\end{equation}
\end{lem}
Thus, we obtain the inverse matrix $A^{-1}$ as follows.
\begin{lem}\label{lem: PBCinverse}
Let $A$ be the matrix defined in (\ref{eq: matrix A}). Then, it holds that
\begin{equation}
A^{-1}=\left((-1)^{i+i'+j+j'}\binom{i}{i'} \binom{j}{j'}\right)_{((i,j),(i',j'))}.
\end{equation}
\end{lem}
\begin{proof}
It is immediately shown by 
\begin{equation}
A^{-1}=\left(P_a \otimes P_b\right)^{-1}=P_a^{-1} \otimes P_b^{-1}.
\end{equation}
\end{proof}
By the above arguments, we may write the coefficients $\alpha_{k,l}=d(a,b)\alpha_{(a,b)}(k,l)$ by using special values $u_{(a,b)}(i,j)$.
\begin{prop}[{\em inverse formula}]\label{prop: inverse formula}
Let $a$ and $b$ be positive integers. Then, for any $1 \leq k \leq a$ and $1 \leq l \leq b$, it holds that
\begin{equation}
d(a,b)\alpha_{(a,b)}(k,l) = \sum_{\substack{1 \leq i \leq k,\\ 1 \leq j \leq l}} (-1)^{i+j+k+l} \binom{k}{i} \binom{l}{j} u_{(a,b)}(i,j).
\end{equation}
\end{prop}
\begin{proof}
It is immediately shown by (\ref{eq: inverse formula as matrix}) and Lemma~\ref{lem: PBCinverse}.
\end{proof}
\begin{ex}
In \cite[Algorithm~5.7]{Nak23}, the method to find the exponent of $\mindilog{a}{b}$ in the ordered product of $\mindilog{0}{1}^n\mindilog{1}{0}^m$ is known for certain $m$ and $n$. For example, consider the exponent of $\mindilog{3}{2}$. Then, we can find the special values of $u_{(3,2)}(m,n)$ as follows:
\begin{equation}
\begin{aligned}
u_{(3,2)}(1,1)=0, u_{(3,2)}(1,2)=0, u_{(3,2)}(2,1)=0,\\
u_{(3,2)}(2,2)=2, u_{(3,2)}(3,1)=1, u_{(3,2)}(3,2)=14.
\end{aligned}
\end{equation}
By Proposition~\ref{prop: inverse formula}, we have $u_{(3,2)}(m,n)=2\binom{m}{2}\binom{n}{2}+\binom{m}{3}\binom{n}{1}+6\binom{m}{3}\binom{n}{2}$.
\end{ex}

\subsection{Reciprocity}\label{subsec: reciprocity}
\begin{prop}[reciprocity]\label{prop: reciprocity}
Let $(a,b) \in N^{+}$. Then, for any $m,n \in \mathbb{Z}_{\geq 0}$,
\begin{equation}
u_{(a,b)}(m,n)=u_{(b,a)}(n,m)
\end{equation}
holds.
\end{prop}

\begin{proof}
By definition of $u_{(a,b)}(m,n)$, we have
\begin{equation}
\dilog{0}{1}^n\dilog{1}{0}^m=\orderedprod_{(a,b) \in N^{+}}\dilog{a}{b}^{u_{(a,b)}(m,n)}.
\end{equation}
Let $F= \left(\begin{smallmatrix}0 & 1\\ 1 & 0 \\\end{smallmatrix}\right)$. Then, by acting $F$, we have
\begin{equation}
\dilog{1}{0}^{-n}\dilog{0}{1}^{-m}=\prod_{(a,b) \in N^{+}} \dilog{b}{a}^{-u_{(a,b)}(m,n)}.
\end{equation}
Considering the inverse of the above relation, we have
\begin{equation}\label{eq: reciprocity}
\dilog{0}{1}^m\dilog{1}{0}^n=\prod_{(a,b) \in (N^{+})^{\mathrm{op}}} \dilog{b}{a}^{u_{(a,b)}(m,n)}.
\end{equation}
The index set $(N^{+})^{\mathrm{op}}$ is the opposite ordered set of $N^{+}$. The RHS is not the strongly ordered product because of the parallel dilogarithm elements. However, by using the relation (\ref{exchange}), we may rearrange it to the strongly ordered product without changing these exponents. Thus, we have
\begin{equation*}
\dilog{0}{1}^{m}\dilog{1}{0}^{n} = \orderedprod_{(b,a) \in N^{+}} \dilog{b}{a}^{u_{(a,b)}(m,n)}.
\end{equation*}
It implies that
\begin{equation}
u_{(b,a)}(n,m)=u_{(a,b)}(m,n).
\end{equation}
\end{proof}

\subsection{Properties of coefficients}\label{subsec: properties of coefficients}
By Theorem~\ref{main thm1}, we write
\begin{equation}
u_{(a,b)}(m,n)=d(a,b)\sum_{\substack{1 \leq i \leq a,\\ 1 \leq j \leq b}} \alpha_{(a,b)}(i,j) \binom{m}{i}\binom{n}{j},
\end{equation}
where $\alpha_{(a,b)}(i,j) \in \mathbb{Z}_{\geq 0}$. In the above expression, the upper bound is essential.
\begin{prop}\label{prop: upper bound}
Let $a$ and $b$ be positive integers. Then, it holds that
\begin{equation}
\alpha_{(a,b)}(a,b)>0.
\end{equation}
\end{prop}
\begin{proof}
We show the claim by the induction on $l=a+b$. If $a+b=2$, namely, if $a=b=1$, then, by (\ref{example deg=7}), we have $u_{({1},{1})}(m,n)=mn$. It indicates that $\alpha_{({1},{1})}(1,1)=1$, and the claim holds. For some $l \geq 2$, suppose that the claim holds for any $(a,b) \in \mathbb{Z}_{\geq 1}^2$ with $a+b=l$. Let $a$ and $b$ be positive integers satisfying $a+b=l+1$. By Proposition~\ref{prop: reciprocity}, it suffices to show the claim when $a \leq b$. Then, since $a+b = l+1 \geq 3$, it holds that $b \geq 2$. Consider the product $C_{(m,n)}$ which is defined in (\ref{eq: C of (m,n)}). By applying Algorithm~\ref{oa} to $C_{(m,n)}$, there exists the following operation:
\begin{equation}
\begin{aligned}
\dilog{0}{1}\dilog{a}{b-1}^{u_{(a,b-1)}(m,n)} \equiv \dilog{a}{b-1}^{u_{(a,b-1)}(m,n)} \dilog{a}{b}^{au_{(a,b-1)}(m,n)} \dilog{0}{1}\\
\hspace{230pt} \mod G^{> a+b}.
\end{aligned}
\end{equation}
The above relation follows from applying Proposition~\ref{prop: proceed lemma} to the following relation by $F=\left(\begin{smallmatrix}a & 0\\ b-1 & 1\\\end{smallmatrix}\right)$. Note that the factors $\mindilog{x}{y}$ satisfying $\deg F\mindilog{x}{y} \leq a+b$ are $\mindilog{x}{y}=\mindilog{1}{0},\mindilog{1}{1},\mindilog{0}{1}$, and $u_{(1,1)}(au_{(a,b-1)}(m,n),a)=a^{2}u_{(a,b-1)}(m,n)$.
\begin{equation}
\begin{aligned}
\dilog{0}{1}^{a}\dilog{1}{0}^{au_{(a,b-1)}(m,n)} \equiv \dilog{1}{0}^{au_{(a,b-1)}(m,n)}\cdots\dilog{1}{1}^{a^2u_{(a,b-1)}(m,n)}\cdots\dilog{0}{1}^{a}\\
\hspace{230pt}\mod G^{>a+b-1}.
\end{aligned}
\end{equation}
Hence, by Lemma~\ref{ordering algorithm'}~E, we have
\begin{equation}
u_{(a,b)}(m,n+1)=u_{(a,b)}(m,n)+au_{(a,b-1)}(m,n)+\cdots.
\end{equation}
It holds that
\begin{equation}
\begin{aligned}
&\ u_{(a,b)}(m,n)\overset{(\ref{eq: reccurence formula for m,n})}=\sum_{k=0}^{n-1} (au_{(a,b-1)}(m,k)+\cdots)\\
=&\ ad(a,b-1)\sum_{\substack{1 \leq i \leq a,\\ 1 \leq j \leq b-1}} \alpha_{(a,b-1)}(a,b-1) \binom{m}{i}\sum_{k=0}^{n-1}\binom{k}{j}+\cdots.\\
\overset{(\ref{sum of BCs})}=&\ ad(a,b-1)\sum_{\substack{1 \leq i \leq a,\\ 1 \leq j \leq b-1}} \alpha_{(a,b-1)}(a,b-1) \binom{m}{i}\binom{n}{j+1}+\cdots.\\
\end{aligned}
\end{equation}
By focusing on the coefficient of $\binom{m}{a}\binom{n}{b}$, we have $d(a,b)\alpha_{(a,b)}(a,b) \geq ad(a,b-1)\alpha_{(a,b-1)}(a,b-1)>0$.
\end{proof}
On the other hand, the lower bound is not known yet. However, a partial result can be derived.
\begin{prop}\label{prop: lower bound}
Let $(a,b) \in N^{+}$ with $a > b \geq 1$. Then, for any positive integers $i < \frac{a}{b}$ and $j \leq b$, it holds that
\begin{equation}
\alpha_{(a,b)}(i,j)=0.
\end{equation}
\end{prop}
\begin{proof}
Let $s$ be the largest integer such that $1 \leq s < \frac{a}{b}$. By Lemma~\ref{lem: the ends of ordering products}, the strongly ordered product expression of $\mindilog{0}{1}^{b}\mindilog{1}{0}^{s}$ is
\begin{equation}
\dilog{1}{0}^{s} \dilog{s}{1}^{b} \cdots \dilog{1}{b}^{s}\dilog{0}{1}^{b}.
\end{equation}
Since $\mindilog{a}{b} < \mindilog{s}{1}$ and $\mindilog{a}{b} \neq \mindilog{1}{0}$, we have $u_{(a,b)}(s,b)=0$. By Theorem~\ref{main thm1}, it implies  that
\begin{equation}
u_{(a,b)}(s,b)= d(a,b)\sum_{\substack{1 \leq i \leq s,\\ 1 \leq j \leq b}} \alpha_{(a,b)}(i,j)\binom{s}{i}\binom{b}{j}=0.
\end{equation}
Since $\binom{s}{i}\binom{b}{j}>0$ and $\alpha_{(a,b)}(i,j) \geq 0$, we have $\alpha_{(a,b)}(i,j)=0$ for any $i \leq s$ and $j \leq b$. This completes the proof.
\end{proof}

\subsection{Special cases}\label{subsec: special cases}

\begin{prop}\label{prop: exponent of 1}
Let $a,b \in \mathbb{Z}_{\geq 0}$. Then, for any $m,n \in \mathbb{Z}_{\geq 0}$, we have
\begin{equation}
u_{(a,1)}(m,n)=\binom{m}{a}\binom{n}{1},\quad u_{(1,b)}(m,n)=\binom{m}{1}\binom{n}{b}.
\end{equation}
\end{prop}

\begin{proof}
By Theorem~\ref{main thm1} and Proposition~\ref{prop: lower bound}, $u_{(a,1)}(m,n)$ is expressed as
\begin{equation}
u_{(a,1)}(m,n)=d(a,1)\alpha_{(a,1)}(a,1)\binom{m}{a}\binom{n}{1}=\alpha_{(a,1)}(a,1)\binom{m}{a}\binom{n}{1}.
\end{equation}
By Lemma~\ref{lem: the ends of ordering products}, we have
\begin{equation}
u_{(a,1)}(a,1)=1.
\end{equation}
Thus, we obtain
\begin{equation}
\alpha_{(a,1)}(a,1)\binom{a}{a}\binom{1}{1}=\alpha_{(a,1)}(a,1)=1,
\end{equation}
and it implies that $u_{(a,1)}(m,n)=\binom{m}{a}\binom{n}{1}$.
\end{proof}
We may also find the exponent of $\mindilog{a}{2}$ explicitly. However, the proof is excessively long.
\begin{thm}\label{b=2}
Let $a \in \mathbb{Z}_{\geq 0}$. Then, for any $m,n \in \mathbb{Z}_{\geq 0}$, we have
\begin{equation}\label{relation b=2}
\begin{aligned}
&u_{(a,2)}(m,n)\\
=\ &\sum_{\frac{a}{2} < k \leq a}\left\lceil \frac{2k-a}{2} \right\rceil\binom{2k-a}{\left\lceil \frac{2k-a}{2} \right\rceil}\binom{k}{2k-a} \binom{m}{k}\binom{n}{2}\\
\ &\qquad+\sum_{\frac{a}{2}+1<k \leq a}\left\{\frac{2k-a}{2}\binom{2k-a-1}{\lceil\frac{2k-a-1}{2}\rceil}-2^{2k-a-2}\right\}\binom{k}{2k-a} \binom{m}{k}\binom{n}{1}.
\end{aligned}
\end{equation}
\end{thm}
In the above relation, $\lceil x \rceil$ is the least integer more than or equal to $x \in \mathbb{Q}$. In (\ref{example deg=7}), we can see the examples. Also, we can simplify the above formula as follows:
\begin{equation}\label{eq: exponent of (a,2) 2}
\begin{aligned}
u_{(a,2)}(m,n)&=m\binom{m-1}{\lfloor \frac{a}{2} \rfloor}\binom{m-1}{\lceil\frac{a}{2}-1\rceil}\binom{n}{2}+\frac{m}{2}\binom{m-1}{\lfloor \frac{a-1}{2}\rfloor}\binom{m-1}{\lceil\frac{a-1}{2}\rceil}\binom{n}{1}\\
&\ \qquad-\sum_{\frac{a}{2}<k\leq a}2^{2k-a-2}\binom{k}{2k-a}\binom{m}{k}\binom{n}{1}.
\end{aligned}
\end{equation}
The proof of equivalence in these two expressions is given in Appedix~\ref{app: equivalence between}.
\par
We can check that every coefficient in the relation (\ref{relation b=2}) is nonzero. The proof of Theorem~\ref{b=2} is in Section~\ref{Proof of b=2}.
\section{Proof of Theorem \ref{b=2}}\label{Proof of b=2}
In this section, we express the exponent of $\mindilog{a}{2}$ explicitly. For the sake of simplicity, the proof of some equalities are given in Appendix~\ref{app: proof of lemma equality}. For any $x \in \mathbb{Q}$, $\lfloor x \rfloor$ is the greatest integer less than or equal to $x$, and $\lceil x \rceil$ is the least integer more than or equal to $x$.
First, we derive the recurrence relations enough to determine all $u_{(a,2)}(m,n)$ based on Method~\ref{method: obtain the explicit forms}.
\begin{prop}\label{prop: recurrence relations}
\textup{(a).}\ Let $a \in \mathbb{Z}_{\geq 1}$. Then, for any $m,n \in \mathbb{Z}_{\geq 0}$, the following relation holds.
\begin{equation}
\begin{aligned}
&\ u_{(a,2)}(m,n+1)\\
=&\ u_{(a,2)}(m,n)+u_{(a,2)}(m,1)\\
&\qquad +\sum_{k=\lceil\frac{a}{2}\rceil}^{a} \biggl\{\sum_{x=a-k}^{\lfloor \frac{a}{2} \rfloor} (a-2x)\binom{a-x}{k-x}\binom{k}{a-x}\biggr\}\binom{m}{k}\binom{n}{1}.\\
\end{aligned}
\end{equation}\\
\textup{(b).}\ Let $a \in \mathbb{Z}_{\geq 3}$. Then, for any $m \in \mathbb{Z}_{\geq 0}$, the following relation holds.
\begin{equation}
\begin{aligned}
&\ u_{(a,2)}(m+1,1)\\
=&\ u_{(a,2)}(m,1)+u_{(a-2,2)}(m,1)\\
 &\qquad + \sum_{k=\lceil\frac{a}{2}\rceil}^{a-1}\biggl\{\sum_{x=a-k}^{\lfloor\frac{a}{2}\rfloor} (a-2x)\binom{a-x}{k-x+1}\binom{k}{a-x}\biggr\}\binom{m}{k}.
\end{aligned}
\end{equation}
\end{prop}
\begin{proof}
(a).
Let $C_{(m,n)}$ be the product which is defined in (\ref{eq: C of (m,n)}). Apply Algorithm~\ref{oa} to $C_{(m,n)}$ repeatedly until it becomes strongly ordered. Suppose that an anti-ordered pair $\mindilog{x}{y}^g\mindilog{z}{w}^f$ produces a factor $\mindilog{a}{2}^{*}$ in Step~1.1~(ii). Since every dilogarithm element $\mindilog{s}{t}$ appearing in the initial product $C_{(m,n)}$ satisfies $t \geq 1$, we have $y,w \geq 1$. By (\ref{eq: X}), there exists $(p,q) \in \mathbb{Z}_{\geq 1}^2$ satisfying
\begin{equation}
\left(\begin{matrix} pz+qx \\ pw+qy \end{matrix}\right) = \left(\begin{matrix} a \\ 2 \end{matrix}\right).
\end{equation}
Since $pw+qy=2$ and $p,q,y,w \geq 1$, we have $p=q=y=w=1$. Since $pz+qx=a$, we have $z=a-x$. By $\mindilog{x}{y}>\mindilog{z}{w}$, we have $xw-yz<0$, and these imply that $2x-a<0$. Thus, every anti-ordered pair $\mindilog{x}{y}^g\mindilog{z}{w}^f$ which produces the factor $\mindilog{a}{2}^{*}$ has a following form:
\begin{equation}
\dilog{x}{1}^g \dilog{a-x}{1}^f \quad \left(x<\frac{a}{2}\right).
\end{equation}
Moreover, factors $\mindilog{x}{1}^{*}$ $(x=0,1,2,\dots)$ are not produced when we apply Algorithm~\ref{oa} to $C_{(m,n)}$. Thus, both $\mindilog{x}{1}^{g}$ and $\mindilog{a-x}{1}^f$ should be in the initial $C_{(m,n)}$. So, the anti-ordered pairs that produce $\mindilog{a}{2}^{*}$ are only the following ones:
\begin{equation}
\dilog{x}{1}^{u_{(x,1)}(m,1)}\dilog{a-x}{1}^{u_{({a-x},{1})}(m,n)}\quad\left(x<\frac{a}{2}\right).
\end{equation}
For any $x$, let $F = \left(\begin{smallmatrix}a-x & x \\ 1 & 1\\\end{smallmatrix}\right)$. Then, the following relations hold by Proposition~\ref{prop: proceed lemma}.
\begin{equation}
\begin{aligned}
&\ \dilog{x}{1}^{u_{(x,1)}(m,1)}\dilog{a-x}{1}^{u_{(a-x,1)}(m,n)}\\
=&\ \left(F\dilog{1}{0}\right)^{(a-2x)u_{(x,1)}(m,1)}\left(F\dilog{0}{1}\right)^{(a-2x)u_{(a-x,1)}(m,n)}\\
=&\  \left(F\dilog{0}{1}^{(a-2x)u_{(a-x,1)}(m,n)}\right)\cdots\\
&\qquad\times\left(F\dilog{1}{1}\right)^{u_{(1,1)}\bigl((a-2x)u_{(a-x,1)}(m,n),(a-2x)u_{(x,1)}(m,1)\bigr)}\\
&\qquad\times\cdots\left(F\dilog{1}{0}\right)^{(a-2x)u_{(x,1)}(m,1)}\\
=&\ \dilog{a-x}{1}^{u_{(a-x,1)}(m,n)}\cdots\\
&\qquad\times\dilog{a}{2}^{\frac{1}{a-2x}u_{(1,1)}\bigl((a-2x)u_{(a-x,1)}(m,n),(a-2x)u_{(x,1)}(m,1)\bigr)}\\
&\qquad\times\cdots\dilog{x}{1}^{u_{(x,1)}(m,1)}.
\end{aligned}
\end{equation}
In the above relations, the third and the fourth products are strongly ordered. Moreover, because of $u_{(1,1)}(m,n)=mn$, we have 
\begin{equation}
\begin{aligned}
&\ \frac{1}{a-2x}u_{(1,1)}\bigl((a-2x)u_{(a-x,1)}(m,n),(a-2x)u_{(x,1)}(m,1)\bigr)\\
=&\ (a-2x)u_{(a-x,1)}(m,n)u_{(x,1)}(m,1).
\end{aligned}
\end{equation}
By Proposition~\ref{prop: exponent of 1}, it is
\begin{equation}
\begin{aligned}
&\ (a-2x)\binom{m}{a-x}\binom{n}{1}\binom{m}{x}=(a-2x)\binom{m}{a-x}\binom{m}{x}\binom{n}{1}\\
\overset{(\ref{binomial lemma 2-prod})}{=}&\ (a-2x)\sum_{k=a-x}^{a} \binom{a-x}{k-x}\binom{k}{a-x}\binom{m}{k}\binom{n}{1}.
\end{aligned}
\end{equation}
Thus, we have
\begin{equation}\label{eq: 123}
\begin{aligned}
&\ u_{(a,2)}(m,n+1)\\
=&\ u_{(a,2)}(m,n)+u_{(a,2)}(m,1)\\
&\qquad +\sum_{0 \leq x < \frac{a}{2}}(a-2x)\sum_{k=a-x}^{a} \binom{a-x}{k-x}\binom{k}{a-x}\binom{m}{k}\binom{n}{1}\\
=&\ u_{(a,2)}(m,n)+u_{(a,2)}(m,1)\\
&\qquad +\sum_{0 \leq x \leq \frac{a}{2}}\sum_{k=a-x}^{a} (a-2x)\binom{a-x}{k-x}\binom{k}{a-x}\binom{m}{k}\binom{n}{1}.\\
\end{aligned}
\end{equation}
Since 
\begin{equation}
\begin{aligned}
&\ \left\{ (x,k) \in \mathbb{Z}^2\ |\ 0\leq x \leq \left\lfloor \frac{a}{2} \right\rfloor,\ a-x \leq k \leq a\right\}\\
=&\ \left\{ (x,k) \in \mathbb{Z}^2\ |\ a-\left\lfloor \frac{a}{2} \right\rfloor(=\left\lceil\frac{a}{2}\right\rceil) \leq k \leq a,\  a-k \leq x \leq \left\lfloor \frac{a}{2} \right\rfloor\right\},
\end{aligned}
\end{equation}
the equality (\ref{eq: 123}) can be expressed as follows:
\begin{equation}
\begin{aligned}
&\ u_{(a,2)}(m,n+1)\\
=&\ u_{(a,2)}(m,n)+u_{(a,2)}(m,1)\\
&\qquad +\sum_{k=\lceil\frac{a}{2}\rceil}^{a} \biggl\{\sum_{x=a-k}^{\lfloor \frac{a}{2} \rfloor} (a-2x)\binom{a-x}{k-x}\binom{k}{a-x}\biggr\}\binom{m}{k}\binom{n}{1}.\\
\end{aligned}
\end{equation}
(b). Let $C_{(m,1)}$ be the product which is defined in (\ref{eq: C of (m,1)}). Namely, consider
\begin{equation}\label{eq: n=1 C0}
\begin{aligned}
C_{(m,1)}=&\dilog{1}{0} \underline{\dilog{1}{1} \dilog{1}{0}^m}\biggl(\orderedprod_{\begin{smallmatrix} x+y \leq a+2 \\ x,y \geq 1 \end{smallmatrix}} \dilog{x}{y}^{u_{(x,y)}(m,1)}\biggr)\\
\end{aligned}
\end{equation}
Let $F=\left(\begin{smallmatrix}1 & 1\\ 0 & 1\\ \end{smallmatrix}\right)$. Then, by Proposition~\ref{prop: proceed lemma}, we have
\begin{equation}
\begin{aligned}
&\dilog{1}{1}\dilog{1}{0}^m = \left(F\dilog{0}{1}\right)\left(F\dilog{1}{0}\right)^m\\
\equiv&\left(F\dilog{1}{0}\right)^m\biggl\{\orderedprod_{\substack{c+d \leq a+1,\\ c,d \geq 1}}\left(F\dilog{c}{d}\right)^{u_{(c,d)}(m,1)}\biggr\}\left(F\dilog{0}{1}\right)\\
=&\dilog{1}{0}^m\biggl(\orderedprod_{\substack{c+d \leq a+1,\\ c,d \geq 1}}\dilog{c+d}{d}^{u_{(c,d)}(m,1)}\biggr)\dilog{1}{1} \\
\equiv&\dilog{1}{0}^m \biggl(\orderedprod_{\substack{c+d \leq a+1,\\ c,d \geq 1,\\z=c+d,w=d}} \dilog{z}{w}^{u_{(z-w,w)}(m,1)}\biggr)\dilog{1}{1} \mod G^{>a+2}.
\end{aligned}
\end{equation}
Putting the last expression to the RHS of (\ref{eq: n=1 C0}), we have
\begin{equation}
\begin{aligned}
C_{(m,1)} \equiv&\ \dilog{1}{0}^{m+1} \biggl(\orderedprod\dilog{z}{w}^{u_{(z-w,w)}(m,1)}\biggr)\dilog{1}{1}\\
&\qquad \times\biggl(\orderedprod_{\begin{smallmatrix} x+y \leq a+2 \\ x,y \geq 1 \end{smallmatrix}} \dilog{x}{y}^{u_{(x,y)}(m,1)}\biggr).
\end{aligned}
\end{equation}
Let
\begin{equation}
\begin{aligned}
&\ C'= \biggl(\orderedprod\dilog{z}{w}^{u_{(z-w,w)}(m,1)}\biggr)\dilog{1}{1}\biggl(\orderedprod_{\begin{smallmatrix} x+y \leq a+2 \\ x,y \geq 1 \end{smallmatrix}} \dilog{x}{y}^{u_{(x,y)}(m,1)}\biggr).
\end{aligned}
\end{equation}
Then, $C'$ satisfies the following conditions:
\begin{itemize}
\item Every dilogarithm element $\mindilog{x}{y}$ appearing in $C'$ satisfies $y \geq 1$.
\item The exponents of the factor $\mindilog{a}{2}^{*}$ are $u_{(a,2)}(m,1)$ and $u_{(a-2,2)}(m,1)$.
\end{itemize}
Thus, by a similar argument of (a), anti-ordered pairs producing $\mindilog{a}{2}^{*}$ are
\begin{equation}
\dilog{x}{1}^{u_{(x-1,1)}(m,1)}\dilog{a-x}{1}^{u_{(a-x,1)}(m,1)}\quad\left(1 \leq x<\frac{a}{2}\right).
\end{equation}
Moreover, for each $x=1,2,\cdots$, it produces $\mindilog{a}{2}^{*}$ whose exponent is
\begin{equation}
\begin{aligned}
&\ (a-2x)u_{(x-1,1)}(m,1)u_{(a-x,1)}(m,1)=(a-2x)\binom{m}{x-1}\binom{m}{a-x}\\
\overset{(\ref{binomial lemma 2-prod})}{=}&\ (a-2x)\sum_{k=a-x}^{a-1}\binom{a-x}{k-x+1}\binom{k}{a-x}\binom{m}{k}.
\end{aligned}
\end{equation}
Thus, we have
\begin{equation}
\begin{aligned}
&\ u_{(a,2)}(m+1,1)\\
=&\ u_{(a,2)}(m,1)+u_{(a-2,2)}(m,1)\\
 &\ \qquad + \sum_{1 \leq x <\frac{a}{2}}\sum_{k=a-x}^{a-1} (a-2x)\binom{a-x}{k-x+1}\binom{k}{a-x}\binom{m}{k}\\
=&\ u_{(a,2)}(m,1)+u_{(a-2,2)}(m,1)\\
 &\ \qquad + \sum_{k=\lceil\frac{a}{2}\rceil}^{a-1}\biggl\{\sum_{x=a-k}^{\lfloor\frac{a}{2}\rfloor} (a-2x)\binom{a-x}{k-x+1}\binom{k}{a-x}\biggr\}\binom{m}{k}.
\end{aligned}
\end{equation}
This completes the proof.
\end{proof}
Next, we try to solve this recurrence relations. By Proposition~\ref{prop: recurrence relations}~(a), we have
\begin{equation}\label{eq: u(m,n)'}
\begin{aligned}
&\ u_{(a,2)}(m,n)=u_{(a,2)}(m,0)+\sum_{j=0}^{n-1}\left\{u_{(a,2)}(m,j+1)-u_{(a,2)}(m,j)\right\}\\
\overset{(\ref{eq: exponents of n=0})}{=}&\ \sum_{j=0}^{n-1}\left\{u_{(a,2)}(m,j+1)-u_{(a,2)}(m,j)\right\}\\
=&\ \sum_{j=0}^{n-1}u_{(a,2)}(m,1)\\
&\qquad +\sum_{k=\lceil\frac{a}{2}\rceil}^{a} \biggl\{\sum_{x=a-k}^{\lfloor \frac{a}{2} \rfloor} (a-2x)\binom{a-x}{k-x}\binom{k}{a-x}\biggr\}\binom{m}{k}\sum_{j=1}^{n-1}\binom{j}{1}\\
\overset{(\ref{sum of BCs})}{=}&\ u_{(a,2)}(m,1)\binom{n}{1}\\
&\qquad +\sum_{k=\lceil\frac{a}{2}\rceil}^{a} \biggl\{\sum_{x=a-k}^{\lfloor \frac{a}{2} \rfloor} (a-2x)\binom{a-x}{k-x}\binom{k}{a-x}\biggr\}\binom{m}{k}\binom{n}{2}.\\
\end{aligned}
\end{equation}
Similarly, by Proposition~\ref{prop: recurrence relations}~(b), we have
\begin{equation}\label{eq: u(m,1)'}
\begin{aligned}
&u_{(a,2)}(m,1)\\
=\ &\sum_{j=0}^{m-1}u_{(a-2,2)}(j,1)\\
&\qquad+ \sum_{k=\lceil\frac{a}{2}\rceil}^{a-1}\left\{\sum_{x=a-k}^{\lfloor\frac{a}{2}\rfloor} (a-2x)\binom{a-x}{k-x+1}\binom{k}{a-x}\right\}\binom{m}{k+1}\\
=\ &\sum_{j=0}^{m-1}u_{(a-2,2)}(j,1) + \sum_{k=\lceil\frac{a}{2}\rceil+1}^{a}\left\{\sum_{x=a-k+1}^{\lfloor\frac{a}{2}\rfloor} (a-2x)\binom{a-x}{k-x}\binom{k-1}{a-x}\right\}\binom{m}{k}.\\
\end{aligned}
\end{equation}

These coefficients can be expressed concisely.
\begin{lem}\label{lem: brief expression}
\textup{(a).}\ For any $k=\left\lceil \frac{a}{2} \right\rceil, \dots, a$, the following equality holds.\\
\begin{equation}
\begin{aligned}
\sum_{x=a-k}^{\lfloor\frac{a}{2}\rfloor} (a-2x)\binom{a-x}{k-x} \binom{k}{a-x}
=\left\lceil \frac{2k-a}{2} \right\rceil\binom{2k-a}{\left\lceil \frac{2k-a}{2} \right\rceil}\binom{k}{2k-a}.
\end{aligned}
\end{equation}
\textup{(b).}\ For any $k=\left\lceil \frac{a}{2} \right\rceil+1, \dots, a$, the following equality holds.
\begin{equation}
\begin{aligned}
&\ \sum_{x=a-k+1}^{\lfloor\frac{a}{2}\rfloor} (a-2x)\binom{a-x}{k-x}\binom{k-1}{a-x}\\
=&\ \left\{\frac{2k-a}{2} \binom{2k-a-1}{\lceil\frac{2k-a-1}{2}\rceil} -2^{2k-a-2}\right\}\binom{k-1}{2k-a-1}.
\end{aligned}
\end{equation}
\end{lem}
The proof is given in Appendix. By using these equalities, (\ref{eq: u(m,n)'}) and (\ref{eq: u(m,1)'}) become as follows:
\begin{equation}\label{eq: u(m,n)''}
\begin{aligned}
&\ u_{(a,2)}(m,n)\\
=&\ u_{(a,2)}(m,1)\binom{n}{1}\\
&\qquad +\sum_{k=\lceil\frac{a}{2}\rceil}^{a} \biggl\{\left\lceil \frac{2k-a}{2} \right\rceil\binom{2k-a}{\left\lceil \frac{2k-a}{2} \right\rceil}\binom{k}{2k-a}\biggr\}\binom{m}{k}\binom{n}{2}.\\
\end{aligned}
\end{equation}
\begin{equation}\label{eq: u(m,1)''}
\begin{aligned}
&u_{(a,2)}(m,1)\\
=\ &\sum_{j=0}^{m-1}u_{(a-2,2)}(j,1)\\
&\qquad + \sum_{k=\lceil\frac{a}{2}\rceil+1}^{a}\left\{\frac{2k-a}{2} \binom{2k-a-1}{\lceil\frac{2k-a-1}{2}\rceil} -2^{2k-a-2}\right\}\binom{k-1}{2k-a-1}\binom{m}{k}.\\
\end{aligned}
\end{equation}

By using the above equalities, we show Theorem~\ref{b=2}.
\begin{proof}[Proof of Theorem~\ref{b=2}]
By (\ref{eq: u(m,n)''}), it suffices to show that
\begin{equation}
u_{(a,2)}(m,1)=\sum_{\frac{a}{2}+1 < k \leq a}\left\{\frac{2k-a}{2}\binom{2k-a-1}{\lceil\frac{2k-a-1}{2}\rceil}-2^{2k-a-2}\right\}\binom{k}{2k-a} \binom{m}{k}.
\end{equation}
We prove it by the induction on $a$. If $a=1,2$, then $u_{(a,2)}(m,1)=0$. Thus, the statement holds. Let $a \geq 3$, and suppose that 
\begin{equation}
\begin{aligned}
&\ u_{(a-2,2)}(m,1)\\
=&\ \sum_{\frac{a}{2}<k \leq a-2}\left\{\frac{2k-a+2}{2}\binom{2k-a+1}{\lceil\frac{2k-a+1}{2}\rceil}-2^{2k-a}\right\}\binom{k}{2k-a+2} \binom{m}{k}.
\end{aligned}
\end{equation}
Then, by (\ref{eq: u(m,1)''}), we have
\begin{equation}
\begin{aligned}
&\ u_{(a,2)}(m,1)\\
=&\ \sum_{j=0}^{m-1}\sum_{\frac{a}{2}<k \leq a-2}\left\{\frac{2k-a+2}{2}\binom{2k-a+1}{\lceil\frac{2k-a+1}{2}\rceil}-2^{2k-a}\right\}\binom{k}{2k-a+2} \binom{j}{k}\\
&\qquad + \sum_{k=\lceil\frac{a}{2}\rceil+1}^{a}\left\{\frac{2k-a}{2} \binom{2k-a-1}{\lceil\frac{2k-a-1}{2}\rceil} -2^{2k-a-2}\right\}\binom{k-1}{2k-a-1}\binom{m}{k}\\
\overset{(\ref{c})}=&\ \sum_{\frac{a}{2}<k \leq a-2}\left\{\frac{2k-a+2}{2}\binom{2k-a+1}{\lceil\frac{2k-a+1}{2}\rceil}-2^{2k-a}\right\}\binom{k}{2k-a+2} \binom{m}{k+1}\\
&\qquad + \sum_{\frac{a}{2}+1<k \leq a}\left\{\frac{2k-a}{2} \binom{2k-a-1}{\lceil\frac{2k-a-1}{2}\rceil} -2^{2k-a-2}\right\}\binom{k-1}{2k-a-1}\binom{m}{k}\\
=&\ \sum_{\frac{a}{2}+1<k \leq a-1}\left\{\frac{2k-a}{2}\binom{2k-a-1}{\lceil\frac{2k-a-1}{2}\rceil}-2^{2k-a-2}\right\}\binom{k-1}{2k-a} \binom{m}{k}\\
&\qquad + \sum_{\frac{a}{2}+1 < k \leq a}\left\{\frac{2k-a}{2} \binom{2k-a-1}{\lceil\frac{2k-a-1}{2}\rceil} -2^{2k-a-2}\right\}\binom{k-1}{2k-a-1}\binom{m}{k}\\
\end{aligned}
\end{equation}
Consider the first term. We can add the factor of $k=a$ since $\binom{k-1}{2k-a}=\binom{a-1}{a}=0$. Thus, we have
\begin{equation}
\begin{aligned}
&\ u_{(a,2)}(m,1)\\
=&\ \sum_{\frac{a}{2}+1<k \leq a}\left\{\frac{2k-a}{2}\binom{2k-a-1}{\lceil\frac{2k-a-1}{2}\rceil}-2^{2k-a-2}\right\}\\
&\qquad\qquad\qquad\times\left\{\binom{k-1}{2k-a}+\binom{k-1}{2k-a-1}\right\} \binom{m}{k}\\
\overset{(\ref{a})}=&\ \sum_{\frac{a}{2}+1<k \leq a}\left\{\frac{2k-a}{2}\binom{2k-a-1}{\lceil\frac{2k-a-1}{2}\rceil}-2^{2k-a-2}\right\}\binom{k}{2k-a} \binom{m}{k}.\\
\end{aligned}
\end{equation}
This completes the proof.
\end{proof}

\appendix
\section{Proof of Lemma~\ref{lem: brief expression}}\label{app: proof of lemma equality}
In this appendix, we prove Lemma~\ref{lem: brief expression}.
\begin{lem}\label{important theorems lem-A}
For any $u \in \mathbb{Z}_{\geq 0}$, the following relation holds.
\begin{equation}
\sum_{x=0}^{\lfloor\frac{u}{2}\rfloor}(u-2x)\binom{u}{x}=\left\lceil\frac{u}{2}\right\rceil\binom{u}{\lceil\frac{u}{2}\rceil}.
\end{equation}
\end{lem}
\begin{proof}
We have
\begin{equation}
\begin{aligned}
&\ \sum_{x=0}^{\lfloor\frac{u}{2}\rfloor}(u-2x)\binom{u}{x}=\sum_{x=0}^{\lfloor\frac{u}{2}\rfloor}(u-x)\binom{u}{u-x}-\sum_{x=0}^{\lfloor\frac{u}{2}\rfloor}x\binom{u}{x}\\
=&\ \sum_{x=0}^{\lfloor\frac{u}{2}\rfloor}u\binom{u-1}{u-x-1}-\sum_{x=1}^{\lfloor\frac{u}{2}\rfloor}u\binom{u-1}{x-1}\\
=&\ u\left\{\sum_{x=0}^{\lfloor\frac{u}{2}\rfloor}\binom{u-1}{x}-\sum_{x=0}^{\lfloor\frac{u}{2}\rfloor-1}\binom{u-1}{x}\right\}=u\binom{u-1}{\lfloor\frac{u}{2}\rfloor}\\
=&\ u\binom{u-1}{u-1-\lfloor\frac{a}{2}\rfloor}=u\binom{u-1}{\lceil \frac{a}{2} \rceil - 1}=\left\lceil \frac{a}{2} \right\rceil \binom{u}{\lceil\frac{a}{2}\rceil}.
\end{aligned}
\end{equation}
\end{proof}
\begin{lem}\label{main thm lem-B}
For any $u \in \mathbb{Z}_{\geq 0}$, the following relation holds.
\begin{equation}
\sum_{x=0}^{\lfloor\frac{u}{2}\rfloor} \binom{u}{x} = \begin{cases}
\displaystyle{2^{u-1} + \frac{1}{2}\binom{u}{\frac{u}{2}}} & u: {\rm even},\\
2^{u-1} & u: {\rm odd}.
\end{cases}
\end{equation}
\end{lem}
\begin{proof}
Since $\displaystyle{\sum_{x=0}^{\lfloor\frac{u}{2}\rfloor} \binom{u}{x} =\sum_{x=0}^{\lfloor\frac{u}{2}\rfloor} \binom{u}{u-x}=\sum_{x=u-\lfloor\frac{u}{2}\rfloor}^{u} \binom{u}{x}}$, we have
\begin{equation}
\begin{aligned}
&2\sum_{x=0}^{\lfloor\frac{u}{2}\rfloor} \binom{u}{x} = \sum_{x=0}^{\lfloor\frac{u}{2}\rfloor} \binom{u}{x}+\sum_{x=u-\lfloor\frac{u}{2}\rfloor}^{u} \binom{u}{x}\\
=&\begin{cases}
\displaystyle{\sum_{x=0}^{u} \binom{u}{x} + \binom{u}{\frac{u}{2}} = 2^{u}+\binom{u}{\frac{u}{2}}} & \textup{$u$: even},\\
 \displaystyle{\sum_{x=0}^{u} \binom{u}{x} = 2^u} & \textup{$u$: odd}.\\
\end{cases}
\end{aligned}
\end{equation}
So, Lemma~\ref{main thm lem-B} holds.
\end{proof}

By using the above equality, we obtain the main lemmas.

\begin{proof}[Proof of Lemma~\ref{lem: brief expression}]
(a).
We can easily check $\binom{a-x}{k-x}\binom{k}{a-x}=\binom{k}{2k-a}\binom{2k-a}{k-x}$. Hence, we have
\begin{equation}
\begin{aligned}
\ &\sum_{x=a-k}^{\lfloor\frac{a}{2}\rfloor} (a-2x)\binom{a-x}{k-x} \binom{k}{a-x} = \binom{k}{2k-a}\sum_{x=a-k}^{\lfloor\frac{a}{2}\rfloor} (a-2x)\binom{2k-a}{k-x}\\
=\ &\binom{k}{2k-a}\sum_{x=0}^{\lfloor\frac{a}{2}\rfloor-(a-k)} (a-2(x+a-k))\binom{2k-a}{k-(x+a-k)}\\
=\ &\binom{k}{2k-a}\sum_{x=0}^{\lfloor\frac{2k-a}{2}\rfloor} ((2k-a)-2x)\binom{2k-a}{x}\\
=\ &\left\lceil \frac{2k-a}{2} \right\rceil\binom{2k-a}{\left\lceil \frac{2k-a}{2} \right\rceil}\binom{k}{2k-a}. \quad (\textup{Lemma~\ref{important theorems lem-A}})
\end{aligned}
\end{equation}
(b).
We can easily check $\binom{a-x}{k-x}\binom{k-1}{a-x}=\binom{k-1}{2k-a-1}\binom{2k-a-1}{k-a+x-1}$. Thus, we have
\begin{equation}\label{eq: lem b}
\begin{aligned}
&\ \sum_{x=a-k+1}^{\lfloor\frac{a}{2}\rfloor} (a-2x)\binom{a-x}{k-x}\binom{k-1}{a-x}\\
=&\ \binom{k-1}{2k-a-1}\sum_{x=a-k+1}^{\lfloor\frac{a}{2}\rfloor}(a-2x)\binom{2k-a-1}{k-a+x-1}\\
=&\ \binom{k-1}{2k-a-1}\sum_{x=0}^{\lfloor\frac{2k-a-2}{2}\rfloor} ((2k-a-2)-2x)\binom{2k-a-1}{x}\\
=&\ \binom{k-1}{2k-a-1}\left\{\sum_{x=0}^{\lfloor\frac{2k-a-2}{2}\rfloor} ((2k-a-1)-2x)\binom{2k-a-1}{x}\right.\\
 &\qquad\qquad\qquad\left.-\sum_{x=0}^{\lfloor\frac{2k-a-2}{2}\rfloor}\binom{2k-a-1}{x}\right\}.\\
\end{aligned}
\end{equation}
(i) If $a$ is odd, then $\lfloor\frac{2k-a-2}{2}\rfloor=\frac{2k-a-3}{2}$, $\lfloor\frac{2k-a-1}{2}\rfloor=\frac{2k-a-1}{2}$. Thus, we have
\begin{equation}\label{eq: odd1}
\begin{aligned}
&\ \sum_{x=0}^{\lfloor\frac{2k-a-2}{2}\rfloor} ((2k-a-1)-2x)\binom{2k-a-1}{x}\\
=&\ \sum_{x=0}^{\frac{2k-a-3}{2}} ((2k-a-1)-2x)\binom{2k-a-1}{x}\\
=&\ \sum_{x=0}^{\frac{2k-a-1}{2}} ((2k-a-1)-2x)\binom{2k-a-1}{x}\\
=&\ \frac{2k-a-1}{2}\binom{2k-a-1}{\frac{2k-a-1}{2}} \quad (\textup{Lemma~\ref{important theorems lem-A}}).
\end{aligned}
\end{equation}
Since $2k-a-1$ is even, we have
\begin{equation}
\begin{aligned}
&\ \sum_{x=0}^{\lfloor\frac{2k-a-2}{2}\rfloor}\binom{2k-a-1}{x}
=\ \sum_{x=0}^{\frac{2k-a-3}{2}}\binom{2k-a-1}{x}\\
=&\ \sum_{x=0}^{\frac{2k-a-1}{2}}\binom{2k-a-1}{x}-\binom{2k-a-1}{\frac{2k-a-1}{2}}\\
=&\ 2^{2k-a-2}+\frac{1}{2}\binom{2k-a-1}{\frac{2k-a-1}{2}}-\binom{2k-a-1}{\frac{2k-a-1}{2}}\quad (\textup{Lemma~\ref{main thm lem-B}})\\
=&\ 2^{2k-a-2}-\frac{1}{2}\binom{2k-a-1}{\frac{2k-a-1}{2}}.
\end{aligned}
\end{equation}
Hence, we have
\begin{equation}
\begin{aligned}
\ &\sum_{x=0}^{\lfloor\frac{2k-a-2}{2}\rfloor}((2k-a-1)-2x)\binom{2k-a-1}{x}-\sum_{x=0}^{\lfloor\frac{2k-a-2}{2}\rfloor}\binom{2k-a-1}{x}\\
=\ &\frac{2k-a-1}{2}\binom{2k-a-1}{\frac{2k-a-1}{2}}-\left(2^{2k-a-2}-\frac{1}{2}\binom{2k-a-1}{\frac{2k-a-1}{2}}\right)\\
=\ &\frac{2k-a}{2}\binom{2k-a-1}{\frac{2k-a-1}{2}}-2^{2k-a-2}.
\end{aligned}
\end{equation}
Putting the last expression to (\ref{eq: lem b}), we have
\begin{equation}
\begin{aligned}
&\ \sum_{x=a-k+1}^{\lfloor\frac{a}{2}\rfloor} (a-2x)\binom{a-x}{k-x}\binom{k-1}{a-x}\\
=&\ \left\{\frac{2k-a}{2} \binom{2k-a-1}{\lceil\frac{2k-a-1}{2}\rceil} -2^{2k-a-2}\right\}\binom{k-1}{2k-a-1}.
\end{aligned}
\end{equation}
(ii) If $a$ is even, then $\lfloor\frac{2k-a-2}{2}\rfloor=\lfloor\frac{2k-a-1}{2}\rfloor=\frac{2k-a-2}{2}$. Thus, we have
\begin{equation}
\begin{aligned}
&\ \sum_{x=0}^{\lfloor\frac{2k-a-2}{2}\rfloor} ((2k-a-1)-2x)\binom{2k-a-1}{x}\\
=&\ \sum_{x=0}^{\lfloor\frac{2k-a-1}{2}\rfloor} ((2k-a-1)-2x)\binom{2k-a-1}{x}\\
=&\ \frac{2k-a}{2}\binom{2k-a-1}{\frac{2k-a}{2}}.
\end{aligned}
\end{equation}
Since $2k-a-1$ is odd, we have
\begin{equation}
\begin{aligned}
&\ \sum_{x=0}^{\lfloor\frac{2k-a-2}{2}\rfloor}\binom{2k-a-1}{x}
=\sum_{x=0}^{\lfloor\frac{2k-a-1}{2}\rfloor}\binom{2k-a-1}{x}\\
=&\ 2^{2k-a-2}.
\end{aligned}
\end{equation}
Putting these expressions to (\ref{eq: lem b}), we obtain
\begin{equation}
\begin{aligned}
&\ \sum_{x=a-k+1}^{\lfloor\frac{a}{2}\rfloor} (a-2x)\binom{a-x}{k-x}\binom{k-1}{a-x}\\
=&\ \left\{\frac{2k-a}{2} \binom{2k-a-1}{\frac{2k-a}{2}} -2^{2k-a-2}\right\}\binom{k-1}{2k-a-1}\\
=&\ \left\{\frac{2k-a}{2} \binom{2k-a-1}{\lceil\frac{2k-a-1}{2}\rceil} -2^{2k-a-2}\right\}\binom{k-1}{2k-a-1}.
\end{aligned}
\end{equation}
\end{proof}

\section{The equivalence between (\ref{relation b=2}) and (\ref{eq: exponent of (a,2) 2})}\label{app: equivalence between}
First, we show the following lemma.
\begin{lem}
Let $\alpha,\beta \in \mathbb{Z}_{\geq 0}$ with $\alpha \geq \beta$. Then, we have
\begin{equation}\label{eq: lem B}
m\binom{m-1}{\alpha}\binom{m-1}{\beta}=\sum_{k=\alpha+1}^{\alpha+\beta+1}\binom{\alpha}{k-\beta-1}\binom{k-1}{\alpha}\binom{k}{k-1}\binom{m}{k}.
\end{equation}
\end{lem}
\begin{proof}
By Lemma~\ref{binomial lemma}, we have
\begin{equation}
\binom{m-1}{\alpha}\binom{m-1}{\beta}=\sum_{k=\alpha}^{\alpha+\beta}\binom{\alpha}{k-\beta}\binom{k}{\alpha}\binom{m-1}{k}. 
\end{equation}
By definition, we can check that $m \binom{m-1}{k}=(k+1)\binom{m}{k+1}$. Thus, we have
\begin{equation}
\begin{aligned}
m\binom{m-1}{\alpha}\binom{m-1}{\beta}&=\sum_{k=\alpha}^{\alpha+\beta}\binom{\alpha}{k-\beta}\binom{k}{\alpha}(k+1)\binom{m}{k+1}\\
&=\sum_{k=\alpha+1}^{\alpha+\beta+1}\binom{\alpha}{k-1-\beta}\binom{k-1}{\alpha}k\binom{m}{k}\\
&=\sum_{k=\alpha+1}^{\alpha+\beta+1}\binom{\alpha}{k-1-\beta}\binom{k-1}{\alpha}\binom{k}{k-1}\binom{m}{k}.
\end{aligned}
\end{equation}
This completes the proof.
\end{proof}
By using this formula, we show the following two equalities.
\begin{prop}
Let $a \in \mathbb{Z}_{\geq 0}$. Then, for any $m \in \mathbb{Z}_{\geq 0}$, we have
\begin{align}
&\begin{aligned}
 \sum_{\frac{a}{2} < k \leq a}\left\lceil \frac{2k-a}{2} \right\rceil\binom{2k-a}{\left\lceil \frac{2k-a}{2} \right\rceil}\binom{k}{2k-a} \binom{m}{k}
= m\binom{m-1}{\lfloor \frac{a}{2} \rfloor}\binom{m-1}{\lceil\frac{a}{2}-1\rceil},
\end{aligned}\label{eq: app. B,1}\\
&\begin{aligned}
\sum_{\frac{a}{2}<k\leq a}\frac{2k-a}{2}\binom{2k-a-1}{\lceil \frac{2k-a-1}{2} \rceil}\binom{k}{2k-a}\binom{m}{k}
= \frac{m}{2}\binom{m-1}{\lfloor \frac{a-1}{2} \rfloor}\binom{m-1}{\lceil \frac{a-1}{2}\rceil}.
\end{aligned}\label{eq: app. B,2}
\end{align}
\end{prop}
\begin{proof}
First, we show (\ref{eq: app. B,1}). By definition, we can check
\begin{equation}
\begin{aligned}
&\ \left\lceil\frac{2k-a}{2} \right\rceil\binom{2k-a}{\left\lceil \frac{2k-a}{2} \right\rceil}\binom{k}{2k-a}=\binom{\lfloor \frac{a}{2} \rfloor}{k-1-\lceil\frac{a}{2}-1\rceil}\binom{k-1}{\lfloor\frac{a}{2}\rfloor}\binom{k}{k-1}\\
=&\ \frac{k!}{(a-k)!(k-1-\lfloor\frac{a}{2}\rfloor)!(k-1-\lceil\frac{a}{2}-1\rceil)!}.
\end{aligned}
\end{equation}
In the above equalities, $k!=k(k-1)\cdots2\cdot1$ is the factorial of $k \in \mathbb{Z}_{\geq 0}$. Thus, by setting $\alpha=\lfloor \frac{a}{2} \rfloor$ and $\beta = \lceil \frac{a}{2}-1 \rceil$ in (\ref{eq: lem B}), the equality (\ref{eq: app. B,1}) holds.
\par
Next, we show (\ref{eq: app. B,2}). We have
\begin{equation}
\begin{aligned}
&\ (2k-a)\binom{2k-a-1}{\lceil \frac{2k-a-1}{2} \rceil}\binom{k}{2k-a}=\binom{\lceil \frac{a-1}{2} \rceil}{k-1-\lfloor \frac{a-1}{2} \rfloor}\binom{k-1}{\lceil \frac{a-1}{2} \rceil}\binom{k}{k-1}\\
=&\ \frac{k!}{(a-k)!(k-1-\lceil \frac{a-1}{2} \rceil)!(k-1-\lfloor \frac{a-1}{2} \rfloor)}.
\end{aligned}
\end{equation}
Thus, by setting $\alpha=\lceil \frac{a-1}{2} \rceil$ and $\beta = \lfloor \frac{a-1}{2} \rfloor$ in (\ref{eq: lem B}), we have
\begin{equation}\label{eq: app. 3}
\sum_{\frac{a}{2} < k \leq a} \frac{2k-a}{2}\binom{2k-a-1}{\lceil \frac{2k-a-1}{2} \rceil}\binom{k}{2k-a}\binom{m}{k}=\frac{m}{2}\binom{m-1}{\lfloor \frac{a-1}{2} \rfloor}\binom{m-1}{\lceil \frac{a-1}{2} \rceil}.
\end{equation}
\end{proof}
Last, we show (\ref{eq: exponent of (a,2) 2}).
\begin{proof}[Proof of (\ref{eq: exponent of (a,2) 2})]
By (\ref{eq: app. B,1}), the first term in the RHS of (\ref{relation b=2}) is $m\binom{m-1}{\lfloor \frac{a}{2} \rfloor}\binom{m-1}{\lceil\frac{a}{2}-1\rceil}\binom{n}{2}$. Consider the second term of (\ref{relation b=2}), that is,
\begin{equation}
\sum_{\frac{a}{2}+1<k \leq a}\left\{\frac{2k-a}{2}\binom{2k-a-1}{\lceil\frac{2k-a-1}{2}\rceil}-2^{2k-a-2}\right\}\binom{k}{2k-a} \binom{m}{k}\binom{n}{1}.
\end{equation}
If $k=\lfloor \frac{a}{2}+1 \rfloor$, or equivalently, if $2k-a=1$ or $2$, we can easily check that
\begin{equation}
\frac{2k-a}{2}\binom{2k-a-1}{\lceil\frac{2k-a-1}{2}\rceil}-2^{2k-a-2}=0.
\end{equation}
Thus, we have
\begin{equation}
\begin{aligned}
&\ \sum_{\frac{a}{2}+1<k \leq a}\left\{\frac{2k-a}{2}\binom{2k-a-1}{\lceil\frac{2k-a-1}{2}\rceil}-2^{2k-a-2}\right\}\binom{k}{2k-a} \binom{m}{k}\binom{n}{1}\\
=&\ \sum_{\frac{a}{2}<k\leq a}\left\{\frac{2k-a}{2}\binom{2k-a-1}{\lceil\frac{2k-a-1}{2}\rceil}-2^{2k-a-2}\right\}\binom{k}{2k-a} \binom{m}{k}\binom{n}{1}\\
=&\ \sum_{\frac{a}{2}<k\leq a}\frac{2k-a}{2}\binom{2k-a-1}{\lceil\frac{2k-a-1}{2}\rceil}\binom{k}{2k-a}\binom{m}{k}\binom{n}{1}\\
&\ \quad - \sum_{\frac{a}{2}<k\leq a}2^{2k-a-2}\binom{k}{2k-a}\binom{m}{k}\binom{n}{1}\\
\overset{(\ref{eq: app. B,2})}=&\ \frac{m}{2}\binom{m-1}{\lfloor \frac{a-1}{2} \rfloor}\binom{m-1}{\lceil \frac{a-1}{2}\rceil}-\sum_{\frac{a}{2}<k\leq a}2^{2k-a-2}\binom{k}{2k-a}\binom{m}{k}\binom{n}{1}.
\end{aligned}
\end{equation}
Thus, (\ref{eq: exponent of (a,2) 2}) holds.
\end{proof}
\bibliographystyle{alpha}
\bibliography{bunken}
\end{document}